\documentclass[english]{article} 
\usepackage[utf8]{inputenc}
\usepackage[T1]{fontenc}
\usepackage{lmodern}
\usepackage[a4paper]{geometry}
\usepackage{enumitem}
\usepackage{verbatim}
\usepackage{tikz,pgfplots}
\usepackage{graphicx}
\usepackage{mathtools}
\usepackage{amsmath}
\usepackage{amsthm}
\usepackage{amsfonts}
\usepackage{amssymb}
\usepackage{amsmath}
\usepackage{dsfont}
\usepackage{mathrsfs}
\usepackage{stmaryrd}
\usepackage{bbm}
\usepackage{framed}
\usepackage{xcolor}
\usepackage[english]{babel}
\usepackage[citecolor=blue,colorlinks=true]{hyperref}
\hypersetup{pdfnewwindow}

\usepackage[textsize=tiny,textwidth=27mm]{todonotes}

\newtheorem{thm}{Theorem}[section]
\newtheorem{definition}[thm]{Definition}
\newtheorem{proposition}[thm]{Proposition}
\newtheorem{lem}[thm]{Lemma}
\newtheorem{pro}[thm]{Proposition}
\newtheorem{corollary}[thm]{Corollary}

\theoremstyle{remark}
\newtheorem{remark}{Remark}[section]

\newcommand{\ud}{\mathrm{d}}

\newcommand{\uu}{\mathfrak{u}}  %w
\newcommand{\ww}{\rho}  %\rho
\newcommand{\imp}{\theta}  %\underline{\rho}
\newcommand{\aaa}{a} %p
\newcommand{\BBB}{\bar{B}_\alpha} %\mathcal{B}
\newcommand{\half}{{\textstyle{1\over2}}}

\newcommand{\R}{\mathbb{R}}

\newcommand{\T}{\mathbb{T}}
\newcommand{\Pro}{\mathbb{P}}
\newcommand{\E}{\mathbb{E}}
\newcommand{\Z}{\mathbb{Z}}
\newcommand{\N}{\mathbb{N}}

\usepackage{enumitem}
\newlist{steps}{enumerate}{1}
\setlist[steps, 1]{label = Step \arabic*:}
\newcommand{\eqdef}{\stackrel{\text{\tiny{def}}}{=}}

\numberwithin{equation}{section}
\pgfplotsset{compat=1.14}
\newcommand{\eps}{\varepsilon}
\newcommand{\dual}[2]{\langle #1, #2\rangle}

\newcommand{\footremember}[2]{%
    \footnote{#2}
    \newcounter{#1}
    \setcounter{#1}{\value{footnote}}%
}

\DeclareMathOperator{\Lip}{Lip}
\begin{document}

%%%%%%%%%%%%%%%%%%%%%%%%%%%%%%%%%%%%%%%%%%%%%%%%%%%%%%%%%%%%%%%%%%%%%%%%
%%%%%%%%%%%%%%%%%%%%%%%%%%%%%%%%%%%%%%%%%%%%%%%%%%%%%%%%%%%%%%%%%%%%%%%%

\title{\bf A Smoluchowski--Kramers approximation for the stochastic variational wave equation}

\author{Billel Guelmame\footremember{alley}{UMPA, CNRS, ENS de Lyon, Universit\'e de Lyon, NYU Abu Dhabi, billel.guelmame@nyu.edu}
  and Julien Vovelle\footremember{trailer}{UMPA, CNRS, ENS de Lyon, julien.vovelle@ens-lyon.fr}
  }

%%%%%%%%%%%%%%%%%%%%%%%%%%%%%%%%%%%%%%%%%%%%%%%%%%%%%%%%%%%%%%%%%%%%%%%%
%%%%%%%%%%%%%%%%%%%%%%%%%%%%%%%%%%%%%%%%%%%%%%%%%%%%%%%%%%%%%%%%%%%%%%%%

\maketitle

\begin{abstract}
We investigate the Smoluchowski--Kramers approximation for the one-dimensional periodic variational wave equation with state-dependent damping and additive noise. We show that weak ``dissipative'' solutions converge to solutions of a stochastic quasilinear parabolic equation.

\end{abstract}

\medskip

\medskip


\tableofcontents

%%%%%%%%%%%%%%%%%%%%%%%%%%%%%%%%%%%%%%%%%%%%%%%%%%%%%%%%%%%%%%%%%%%%%%%%
%%%%%%%%%%%%%%%%%%%%%%%%%%%%%%%%%%%%%%%%%%%%%%%%%%%%%%%%%%%%%%%%%%%%%%%%
\section{Introduction}
%%%%%%%%%%%%%%%%%%%%%%%%%%%%%%%%%%%%%%%%%%%%%%%%%%%%%%%%%%%%%%%%%%%%%%%%
%%%%%%%%%%%%%%%%%%%%%%%%%%%%%%%%%%%%%%%%%%%%%%%%%%%%%%%%%%%%%%%%%%%%%%%%

In this paper, we consider the damped variational wave equation with stochastic forcing, given by 
\begin{gather}\label{VWE}
\mu\, \ud u^\mu_t  -  c(u^\mu) \left( c(u^\mu)  u^\mu_x \right)_x \ud t  +  \gamma(u^\mu) u^\mu_t\, \ud t = f(u^\mu)\, \ud t +  \Phi\,  \ud W,  \qquad (t,x) \in  (0,T)\times\T,
\end{gather}
where $\mu>0$ is a positive parameter, $\T = \R/\Z$ is the one-dimensional torus, and $W$ is a cylindrical Wiener process over the filtered probability space $(\Omega, \mathcal{F}, \Pro, (\mathcal{F}_t)_{t\geqslant 0})$. The friction $\gamma$ is assumed to be strictly positive, and the nonlinearity $f$ is a Lipschitz-continuous function.
Our main goals are to prove the existence of global weak dissipative solutions to \eqref{VWE}, and to investigate the small-mass limit as $\mu \to 0$.

The classical (and deterministic) variational wave equation, corresponding to the case $\mu=1$, $f \equiv 0$, $\gamma \equiv 0$, and $\Phi \equiv 0$, reads 
\begin{equation}\label{CWE} 
u_{tt}-c(u)(c(u)u_x)_x=0. 
\end{equation} 
Equation \eqref{CWE} arises in various physical settings, such as the modeling of nematic liquid crystals \cite{Saxton89,HunterSaxton1991,GlasseyHunterZheng1997}, long waves in dipole chains \cite{GlasseyHunterZheng1997,GI92,ZI92}, and also in classical field theories and general relativity \cite{GlasseyHunterZheng1997}.

This equation has been widely studied on the real line $\R$. 
Local-in-time well-posedness follows from standard arguments, while the formation of singularities (finite-time blow-up of smooth solutions) has been proved in \cite{GlasseyHunterZheng1996}. The existence of global rarefactive solutions was proved in \cite{ZhangZheng2001}.
Two distinct types of global weak solutions, conservative and dissipative, are known for \eqref{CWE}. Conservative solutions were constructed in \cite{BressanZheng2006} by reformulating the problem in Lagrangian coordinates, where the uniqueness was established later in \cite{BressanChenZhang2015}.
Dissipative solutions were first obtained via an approximated system and compactness methods in \cite{ZhangZheng2003,ZhangZheng2005}, and later via Lagrangian coordinates in \cite{BressanHuang2016} following the spirit of \cite{BressanZheng2006}. To the best of our knowledge, the uniqueness of dissipative solutions remains an open problem.
In the stochastic setting on the torus $\T$, the authors of the present paper studied \eqref{CWE} with additive noise (i.e., equation \eqref{VWE} with $f \equiv \gamma \equiv 0$) in \cite{GV25}. There, we established local well-posedness, constructed initial data such that the corresponding strong solutions blow-up in finite time, and proved the existence of global weak dissipative solutions.
We present in this paper the missing pieces to establish the existence of global weak solutions in the case $\gamma \geqslant \gamma_1>0$ and $f$ is Lipschitz continuous. 
Another stochastic variant of \eqref{CWE} involving transport noise and viscosity was recently studied in \cite{Pang24}, where global well-posedness was proved.

After establishing the existence of global solutions to \eqref{VWE}, we focus on the small-mass limit $\mu \to 0$, also known as the Smoluchowski--Kramers approximation. The validity of this limit has been investigated in various settings. We refer to \cite{Freidlin04,Spiliopoulos07,CerraiFreidlin11,CerraiWehrZhu20,HerzogHottovyVolpe16,HottovyMcDanielVolpeWehr15,Lee14} for the finite dimensional case (ODEs), and to \cite{BrzezniakCerrai23,CerraiDebussche23,CerraiFreidlin06a,CerraiFreidlin06b,CerraiFreidlinSalins17,CerraiSalins17,CerraiXi22,CerraiXie23,CerraiXie24,LvRoberts12,LvRoberts14,LvRobertsWang14,Nguyen18,Salins19} for the infinite dimensional case (PDEs).

In the context of the stochastic damped wave equation, Cerrai and Freidlin \cite{CerraiFreidlin06a} studied the small-mass limit with constant friction $\gamma$, constant wave speed $c$, and additive noise, for any spatial dimension $d \geqslant 1$. 
In another work \cite{CerraiFreidlin06b}, they addressed multiplicative noise, but only for $d=1$. 
The extension to multiplicative noise in higher dimensions was later carried out by Salins in \cite{Salins19}.
 State-dependent friction $\gamma = \gamma(u)$, in combination with multiplicative noise and arbitrary spatial dimensions, was studied in \cite{CerraiXi22}. More recently, Cerrai and Debussche \cite{CerraiDebussche23} considered a system of stochastic wave equations with non-constant friction. 

In this paper, we consider equation \eqref{VWE} with both state-dependent friction and wave speed, focusing on the one-dimensional case ($d = 1$) with additive noise. This restriction is due to the complexity of the variational wave equation: even in the deterministic setting, global existence results are only known for $d = 1$.
Replacing the damping term $\gamma(u^\mu) u^\mu_t\, \ud t$ in \eqref{VWE} with $\ud \Gamma(u^\mu)$, where $\Gamma'(u) = \gamma (u)$, we prove that as $\mu \to 0$, the solutions of \eqref{VWE} converge in probability to a solution of the stochastic quasilinear parabolic equation
\begin{equation}\label{limeq version rho u}
\ud \Gamma(u) -  c(u) \left( c(u)  u_x \right)_x \ud t  = f(u)\, \ud t +  \Phi\,  \ud W.
\end{equation}
This equation can also be rewritten in terms of $\ww = \Gamma(u) = \int_0^u \gamma(v)\, \ud v$ as 
\begin{equation}\label{limeq version rho rho}
		\ud\ww -  \alpha(\ww) \left( \beta(\ww) \ww_x \right)_x \ud t  = f\circ\Gamma^{-1}(\ww)\, \ud t +  \Phi\,  \ud W, \qquad 		\ww\eqdef\Gamma(u)\eqdef\int_0^u\gamma(v)\, \ud v,
	\end{equation}
	where $\alpha=c\circ\Gamma^{-1}=\beta \gamma\circ\Gamma^{-1}$.  One can also write \eqref{limeq version rho u}
 in terms of the displacement $u$ as 	
	\begin{equation}\label{limeq version u}
\ud u -  \tfrac{ c(u)}{\gamma(u)} \left( c(u)  u_x \right)_x \ud t  = \tfrac{f(u)}{\gamma(u)}\, \ud t -\tfrac{\gamma'(u)}{2\gamma(u)^3} q\, \ud t + \tfrac{1}{\gamma(u)} \Phi\,  \ud W,
	\end{equation}
with $q(x) =\sum_{k \geqslant 1} \sigma_k(x)^2 $ and $\Phi(x) W(t)= \sum_{k\geqslant 1} \sigma_k(x) \beta_k(t) $,  where $(\beta_1(t),\beta_2(t),\dotsc)$ are independent one-dimensional Wiener processes (see Section \ref{sec:stochastic} below).
As noted in \cite{CerraiXi22}, equation \eqref{limeq version u} contains an additional drift term that does not appear in \eqref{VWE} with $\mu = 0$. This term arises from the interaction between the non-constant friction and the stochastic noise. It is in fact It\^o's  correction term. 	
At last, to establish the uniqueness of solutions to \eqref{limeq version rho rho} and \eqref{limeq version u}, we write the quasi-linear parabolic equation in divergence form
		\begin{equation}\label{limeq version rhoc}
		\ud\imp  -  \left(b(\imp)\imp_x \right)_x \ud t  = F(x,\imp)\, \ud t +  \Psi(\imp)\,  \ud W,
	\end{equation}
	where 
\begin{equation*}
			\imp\eqdef\underline{\Gamma}(u)\eqdef\int_0^u\tfrac{\gamma(v)}{c(v)} \, \ud v,
\qquad 
	b=\tfrac{c^2}{\gamma}\circ\underline{\Gamma}^{-1},\qquad 
		F(x,\cdot)=\left(\tfrac{f}{c}- q(x)\tfrac{c'}{2\gamma c^2}\right)\circ\underline{\Gamma}^{-1},\qquad 
		\Psi=\tfrac{\Phi}{c\circ\underline{\Gamma}^{-1}}.
\end{equation*}
Equation \eqref{limeq version rhoc} was previously studied in \cite{HofmanovaZhang17} in the case $F=F(\imp)$, where existence and uniqueness of weak solutions were established. The case $F=F(x,\imp)$ on bounded domains with Dirichlet boundary conditions was studied in \cite{CerraiXi22}. The uniqueness result in our case for \eqref{limeq version rhoc} follows by using similar arguments from both works.
Finally, using the uniqueness of the limiting equation together with the Gy\"ongy--Krylov argument, \cite[Lemma 1.1]{GyongyKrylov96}, we conclude the convergence of $u^\mu$ to $u$ in probability.

The limit from \eqref{VWE} to \eqref{limeq version rho u} (or \eqref{limeq version rho rho}) is however not straightforward. This difficulty arises from the dependence of the wave speed 
$c$ on the state $u$, which leads to the appearance of a non-negative defect measure $\hat{a}$ in \eqref{limeq version rho u}	 (see \eqref{eq ei with defect measure} below). 
Using the energy ``dissipation'' of \eqref{VWE},  together with It\^o's formula, we show that the defect measure is equal to zero (see Proposition \ref{prop:Identification} below).

The paper is organized as follows. In Section \ref{sec:mainresults}, we introduce the stochastic settings, we define an approximated system that is used to obtain global existence of weak solutions to \eqref{VWE}. We also state the main results of the paper. 
Section \ref{sec:energy} is devoted to obtaining some energy estimates, while in Section \ref{sec:ue}, we obtain estimates that are uniform in $\mu$.
In Section \ref{sec:limit}, we prove compactness results and we establish the validity of the Smoluchowski--Kramers approximation. In Appendix \ref{app:equiv fomrulations}, we show the equivalence between \eqref{VWE} and the system defined in Section \ref{sec:stochastic}. Finally, in Appendix \ref{app:Ito}, we prove two It\^o formulas that are used throughout the paper.

%%%%%%%%%%%%%%%%%%%%%%%%%%%%%%%%%%%%%%%%%%%%%%%%%%%%%%%%%%%%%%%%%%%%%%%%
%%%%%%%%%%%%%%%%%%%%%%%%%%%%%%%%%%%%%%%%%%%%%%%%%%%%%%%%%%%%%%%%%%%%%%%%
\section{The equations and main results}\label{sec:mainresults}
%%%%%%%%%%%%%%%%%%%%%%%%%%%%%%%%%%%%%%%%%%%%%%%%%%%%%%%%%%%%%%%%%%%%%%%%
%%%%%%%%%%%%%%%%%%%%%%%%%%%%%%%%%%%%%%%%%%%%%%%%%%%%%%%%%%%%%%%%%%%%%%%%

%%%%%%%%%%%%%%%%%%%%%%%%%%%%%%%%%%%%%%%%%%%%%%%%%%%%%%%%%%%%%%%%%%%%%%%%
\subsection{The stochastic variational wave equation}\label{sec:stochastic}
%%%%%%%%%%%%%%%%%%%%%%%%%%%%%%%%%%%%%%%%%%%%%%%%%%%%%%%%%%%%%%%%%%%%%%%%
Consider the filtered probability space
\begin{equation*}
(\Omega, \mathcal{F}, \Pro, (\mathcal{F}_t)_{t\geqslant 0}).
\end{equation*}
Let $\mathfrak{U}$ be a Hilbert space with an orthonormal basis $(g_k)_{k \geqslant 1}$, and let $\mathfrak{U}_{-1}$ be another Hilbert space such that the injection $\mathfrak{U}\hookrightarrow\mathfrak{U}_{-1}$ is Hilbert--Schmidt and $\mathfrak{U}$ is dense in $\mathfrak{U}_{-1}$. Typically, we will consider the set of linear functions $\varphi\colon\mathfrak{U}\to\R$ satisfying
	\begin{equation*}
		\|\varphi\|_{\mathfrak{U}_{-1}}^2\eqdef \sum_{k\geqslant 1}\tfrac{1}{k^2}|\varphi(g_k)|^2 <\infty.
	\end{equation*}
The injection $i\colon\mathfrak{U}\hookrightarrow\mathfrak{U}_{-1}$ is then provided by the identification of $\mathfrak{U}$ with its topological dual, by $i(g)(h)=\dual{h}{g}_\mathfrak{U}$. Let $W$ be the cylindrical Wiener process defined by
\begin{equation}\label{def cylindrical Wiener}
W(t)  \eqdef  \sum_{k\geqslant 1} g_k  \beta_k(t), \quad  t\geqslant 0,
\end{equation}
where $(\beta_1(t),\beta_2(t),\dotsc)$ are independent one-dimensional Wiener processes. We can consider, equivalently, $W(t)$ as a linear functional given by 
\begin{equation*}
	W(t)(h)  = \sum_{k\geqslant 1} \dual{h}{g_k}_{\mathfrak{U}}  \beta_k(t), 
\end{equation*}
 or $W(t)$ as an element of $\mathfrak{U}_{-1}$, the sum in \eqref{def cylindrical Wiener} being convergent in this larger space,  see Section~4.1.2 in \cite{DaPratoZabczyk14}.
Let $\Phi\colon\mathfrak{U} \to L^2(\T)$ such that for any $k\geqslant 1$ we have $\sigma_k \eqdef \Phi g_k \in C(\T)$ and 
\begin{equation}\label{defq}
q_0  \eqdef  \sum_k \|\sigma_k\|_{W^{1,\infty}(\T)}^2  <  \infty, \qquad q(x)  \eqdef  \sum_k \sigma_k(x)^2.
\end{equation}
By \eqref{defq} and the injection $L^\infty(\T)\hookrightarrow L^2(\T)$, the map $\Phi$ is Hilbert--Schmidt. Let us assume that  the non-linear speed of sound $c$, the friction coefficient $\gamma$, and the source term $f$ are smooth functions satisfying
\begin{gather}
\label{coeff-c} 0  <  c_1  \leqslant  c(u)  \leqslant  c_2, \qquad  0  \leqslant  c'(u)  \leqslant  c_3, \\ 
\label{coeff-gammaf} 0  <  \gamma_1  \leqslant  \gamma(u)  \leqslant  \gamma_2, \qquad \mathrm{Lip}(f)  \leqslant L,
\end{gather}
for some constants $c_1,c_2,c_3,\gamma_1,\gamma_2, L \in(0,\infty)$.
We assume further that for some $\bar{u} \in \R$, we have
\begin{equation}\label{c'g}
\kappa \eqdef 	\liminf_{u \to - \infty} \left( c'(u) \int_{\bar{u}}^u \frac{\ud v}{c(v)} \right) > -1,
\end{equation}
and
\begin{equation}\label{cgamma}
	\sup_{u \in \R} \left( |c''(u)| + |\gamma'(u)|\right) < \infty.
\end{equation}
The condition~\eqref{c'g} is in particular satisfied if $uc'(u)=o(1)$ when $u\to-\infty$. 
Let $\mu >0$, we consider the stochastic variational wave equation with friction and additive noise in $(0,T)\times\T$ 
\begin{subequations}\label{SVWE1} 
\begin{gather}
\mu\, \ud u^\mu_t  -  c(u^\mu) \left( c(u^\mu)  u^\mu_x \right)_x \ud t  +  \gamma(u^\mu) u^\mu_t\, \ud t = f(u^\mu)\, \ud t +  \Phi\,  \ud W,  \\ 
u^\mu(0,\cdot)  =  u_0, \qquad u^\mu_t(t=0,\cdot)  =  v_0, 
\end{gather}
\end{subequations}
where
\begin{equation}\label{Positivecprime00}
\inf_{x \in \T} c'(u_0(x))  >  0.
\end{equation}
The equation~\eqref{SVWE1} admits an equivalent formulation, given by the system 
\begin{subequations}\label{SVWE2}
\begin{gather}\label{Req}
\sqrt{\mu}\, \ud R^\mu  +  c(u^\mu)  R^\mu_x\,  \ud t   +  \gamma(u^\mu) \tfrac{R^\mu+S^\mu}{2\sqrt{\mu}}\, \ud t =  \tilde{c}'(u^\mu) \left[(R^\mu)^2  -  (S^\mu)^2 \right] \ud t  +  f(u^\mu)\, \ud t + \Phi \,  \ud W, \\  \label{Seq}
\sqrt{\mu}\, \ud S^\mu  -  c(u^\mu)  S^\mu_x\,  \ud t   +  \gamma(u^\mu) \tfrac{R^\mu+S^\mu}{2\sqrt{\mu}}\, \ud t =  \tilde{c}'(u^\mu) \left[(S^\mu)^2  -  (R^\mu)^2 \right] \ud t  + f(u^\mu)\, \ud t +  \Phi \,  \ud W, 
\end{gather}
\end{subequations}
completed with the equation
\begin{equation}\label{udef}
u^\mu(t,x)  =  \mathcal{C}^{-1}\left\{ \mathcal{C} \left\{ \int_0^t {\textstyle \left(\frac{R^\mu  +  S^\mu}{2 \sqrt{\mu}}  \right) (s,0)} \, \ud s  +  u_0(0) \right\} +  \int_0^x {\textstyle \frac{S^\mu  -  R^\mu}{2}}(t,y) \, \ud y \right\},
\end{equation}
where
\begin{equation*}
\mathcal{C}(r)  \eqdef  \int_0^r c(\sigma)\,  \ud \sigma, \qquad \tilde{c}(u)  \eqdef  {\textstyle \frac{1}{4}}  \log c(u),
\end{equation*}
which expresses $u$ as a non-local function of $(R^\mu,S^\mu)$. The system \eqref{SVWE2} is deduced from \eqref{SVWE1} by setting
\begin{equation}\label{defRS}
R^\mu  \eqdef \sqrt{\mu} u^\mu_t  -  c(u^\mu)  u^\mu_x, \qquad S^\mu  \eqdef \sqrt{\mu}  u^\mu_t  +  c(u^\mu)  u^\mu_x.
\end{equation}
The equivalence of \eqref{SVWE1} and \eqref{SVWE2}-\eqref{udef} is discussed with more details in Appendix~\ref{app:equiv fomrulations}. Note that the corresponding initial conditions for \eqref{SVWE2} are
\begin{equation}\label{IC}
R^\mu(0,\cdot)  =  R_0^\mu  \eqdef \sqrt{\mu} v_0  -  c(u_0)  u_0', \quad S^\mu(0,\cdot)  =  S_0^\mu  \eqdef  \sqrt{\mu} v_0  +  c(u_0)  u_0'.
\end{equation}
Finally, we define the energy 
\begin{equation*}
\mathcal{E}^\mu \eqdef \int_\T \left((R^\mu)^2+(S^\mu)^2 \right) \ud x = 2 \int_\T \left(\mu (u_t^\mu)^2+c(u^\mu)^2(u_x^\mu)^2 \right) \ud x.
\end{equation*}

%%%%%%%%%%%%%%%%%%%%%%%%%%%%%%%%%%%%%%%%%%%%%%%%%%%%%%%%%%%%%%%%%%%%%%%%%
\begin{remark}[Notations]\label{rk:notations-intro} The subscript $t$, as in $u_t$, always denotes the partial derivative with respect to $t$, and never the value at the given time $t$ of a stochastic process $X$ (the latter being simply denoted by $X(t)$).
\end{remark}
%%%%%%%%%%%%%%%%%%%%%%%%%%%%%%%%%%%%%%%%%%%%%%%%%%%%%%%%%%%%%%%%%%%%%%%%%

%%%%%%%%%%%%%%%%%%%%%%%%%%%%%%%%%%%%%%%%%%%%%%%%%%%%%%%%%%%%%%%%%%%%%%%%
\subsection{An approximated system}
%%%%%%%%%%%%%%%%%%%%%%%%%%%%%%%%%%%%%%%%%%%%%%%%%%%%%%%%%%%%%%%%%%%%%%%%

The system \eqref{SVWE2} is locally (in time) well-posed, in the class of classical solutions. However, singularities may appear in finite time with a high probability \cite{GV25}. In order to obtain global weak solutions, we define the cut-off function  
\begin{equation*}
\chi_\varepsilon (\xi)  \eqdef  \left(\xi  -  \tfrac{1}{\varepsilon} \right)^2 \mathds{1}_{[\frac{1}{\varepsilon}, \infty)} (\xi)  =  
\begin{cases}
\left(\xi  -  \tfrac{1}{\varepsilon} \right)^2, & \xi \geqslant 1/\varepsilon, \\ 
0, & \xi < 1/\varepsilon,
\end{cases}
\end{equation*}
for any $\varepsilon>0$.  Then, the system \eqref{SVWE2} can be approximated by
\begin{subequations}\label{SVWEep}
\begin{gather}\nonumber
\sqrt{\mu}\, \ud R^{\mu, \varepsilon}  +  c(u^{\mu, \varepsilon})  R^{\mu, \varepsilon}_x  \ud t +  \gamma(u^{\mu, \varepsilon}) \tfrac{R^{\mu, \varepsilon}+ S^{\mu, \varepsilon}}{2 \sqrt{\mu}}\, \ud t \\ \label{Reqep}
 =  \tilde{c}'(u^{\mu, \varepsilon}) \left[(R^{\mu, \varepsilon})^2  -  (S^{\mu, \varepsilon})^2  -  \chi_\varepsilon(R^{\mu, \varepsilon})  +  2  R^{\mu, \varepsilon}  \Theta^{\mu, \varepsilon} \right] \ud t + f(\uu^{\mu, \varepsilon})\, \ud t +  \Phi^{\varepsilon}   \ud W, \\  \nonumber
\sqrt{\mu}\, \ud S^{\mu, \varepsilon}  -  c(u^{\mu, \varepsilon})  S^{\mu, \varepsilon}_x  \ud t +  \gamma(u^{\mu, \varepsilon}) \tfrac{R^{\mu, \varepsilon}+ S^{\mu, \varepsilon}}{2 \sqrt{\mu}}\, \ud t \\  \label{Seqep} 
=  \tilde{c}'(u^{\mu, \varepsilon}) \left[ (S^{\mu, \varepsilon})^2  -  (R^{\mu, \varepsilon})^2  -  \chi_\varepsilon(S^{\mu, \varepsilon})  -  2  S^{\mu, \varepsilon}  \Theta^{\mu, \varepsilon} \right] \ud t + f(\uu^{\mu, \varepsilon})\, \ud t +  \Phi^{\varepsilon}   \ud W,\\ 
R^{\mu, \varepsilon}(0,\cdot)  =  R^{\mu, \varepsilon}_0  \eqdef   J_\varepsilon R_0^\mu, \quad \qquad S^{\mu, \varepsilon}(0,\cdot)  =  S^{ \mu,\varepsilon}_0  \eqdef    J_\varepsilon S_0^\mu,
\end{gather}
\end{subequations}
coupled with the equations ($x\in [0,1)$)
\begin{gather}\label{udefep}
u^{\mu, \varepsilon}(t,x)  =  \mathcal{C}^{-1}\left\{ \mathcal{C} \left\{ \int_0^t {\textstyle \left(\frac{R^{\mu, \varepsilon}  +  S^{\mu, \varepsilon}}{2 \sqrt{\mu}}  \right) (s,0)} \, \ud s  +  u^{ \varepsilon}_0(0) \right\} +  \int_0^x \left[{\textstyle \frac{S^{\mu, \varepsilon}  -  R^{\mu, \varepsilon}}{2}}(t,y)  -  \Theta^{\mu, \varepsilon}(t)\right] \ud y\right\},\\ \label{wdefep}
\uu^{\mu, \varepsilon}(t,x)  \eqdef   \int_0^t {\textstyle \left(\frac{R^{\mu, \varepsilon}  +  S^{\mu, \varepsilon}}{2 \sqrt{\mu}}  \right) (s,x)} \, \ud s  +  u^{\varepsilon}_0(x).
\end{gather}
In \eqref{SVWEep}-\eqref{udefep}, we have introduced the ``correction term'' (the necessity of this correction term is manifest if we reproduce the analysis given in Appendix~\ref{app:equiv fomrulations}, see in particular the condition~\eqref{NSC int ux})
	\begin{equation}\label{psieps}
		\Theta^{\mu, \varepsilon}(t)  \eqdef   \int_0^1\tfrac{S^{\mu, \varepsilon}-R^{\mu, \varepsilon}}{2}(t,y) \, \ud y.
	\end{equation}
The modification of $u^{\mu, \varepsilon}$ into $\uu^{\mu, \varepsilon}$ at several instances in \eqref{SVWEep} can also be seen as a correction of the value of $u^{\mu, \varepsilon}$ by a quantity which satisfies the identity (compare to \eqref{ut by RS})
\begin{equation}\label{ut by RSep}
	\uu^{\mu, \varepsilon}_t=\tfrac{S^{\mu, \varepsilon}+R^{\mu, \varepsilon}}{2\sqrt{\mu}}.
\end{equation} 
In \eqref{SVWEep}-\eqref{udefep}, we have also used the Friedrichs mollifier $J_\varepsilon$, defined as the convolution operator $R\mapsto R\ast\varrho_\varepsilon$, where $(\varrho_\eps)$ is an approximation of the unit.
In \eqref{udefep} and \eqref{wdefep}, $u_0^\varepsilon$ is defined as $u_0^\varepsilon \eqdef \mathcal{C}^{-1}(J_\varepsilon \mathcal{C}(u_0))$.
In \eqref{SVWEep},  $\Phi^{ \varepsilon} : \mathfrak{U} \to \cap_{s \geqslant 0} H^s(\T)$ is defined as $ \Phi^{ \varepsilon} g_k \eqdef \sigma_k^{\varepsilon} \eqdef J_\varepsilon\sigma_k \in C^\infty(\T)$ for any $k\geqslant 1$ and $\varepsilon>0$.
Clearly, we have the domination $|\sigma_k^{\varepsilon}| \leqslant |\sigma_k|$, and thus (see \eqref{defq})
\begin{equation*}
\sum_k \|\sigma_k^\varepsilon\|_{C(\T)}^2  \leqslant  \sum_k \|\sigma_k\|_{C(\T)}^2  \leqslant  q_0, \qquad q^\varepsilon(x)  \eqdef  \sum_k \sigma_k^\varepsilon(x)^2.
\end{equation*}

%%%%%%%%%%%%%%%%%%%%%%%%%%%%%%%%%%%%%%%%%%%%%%%%%%%%%%%%%%%%%%%%%%%%%%%%
\subsection{Global solutions}
%%%%%%%%%%%%%%%%%%%%%%%%%%%%%%%%%%%%%%%%%%%%%%%%%%%%%%%%%%%%%%%%%%%%%%%%

%%%%%%%%%%%%%%%%%%%%%%%%%%%%%%%%%%%%%%%%%%%%%%%%%%%%%%%%%%%%%%%%%%%%%%%%
\begin{thm}[Global existence of regular solutions]\label{thm:glo-existRSep} Let $c,f$ and $\gamma$ satisfying \eqref{coeff-c} and \eqref{coeff-gammaf}. Let $\Phi$ satisfy \eqref{defq}. Then \eqref{SVWEep} admits a unique global smooth solution $(R^{\mu, \varepsilon}, S^{\mu, \varepsilon})$, in the following sense: 
\begin{enumerate}
	\item for all $s> 3/2$, for all $p\in[1,\infty)$, for all $T>0$,
	\begin{equation*}
		R^{\mu, \varepsilon}, S^{\mu, \varepsilon}\in L^p_\mathcal{P}(\Omega; C([0,T]; H^{s}(\T))\cap C^1([0,T]; H^{s-1}(\T))),
	\end{equation*}
	where $\mathcal{P}$ denotes the predictable $\sigma$-algebra,
	\item almost surely, for all $t\geqslant 0$ and $x\in\R$, $R^{\mu, \varepsilon}$ and $S^{\mu, \varepsilon}$ satisfy
		\begin{gather*}\nonumber
		\sqrt{\mu}\, R^{\mu, \varepsilon}(t,x)=\sqrt{\mu}\, R^{\mu, \varepsilon}_0 (x)+\int_0^t A^{\mu, \varepsilon}(s,x) \, \ud s+\Phi^\eps(x)W(t),\\
			\sqrt{\mu}\, S^{\mu, \varepsilon}(t,x)=\sqrt{\mu}\, S^{\mu, \varepsilon}_0 (x)+\int_0^t B^{\mu, \varepsilon}(s,x)\, \ud s+\Phi^\eps(x)W(t),
	\end{gather*}
	where the drift terms $A^{\mu, \varepsilon}$, $B^{\mu, \varepsilon}$ are those given in \eqref{SVWEep}.
\end{enumerate}

\end{thm}
%%%%%%%%%%%%%%%%%%%%%%%%%%%%%%%%%%%%%%%%%%%%%%%%%%%%%%%%%%%%%%%%%%%%%%%%

If the regularized problem admits smooth solutions, we consider much weaker solutions for the original problem \eqref{SVWE1}, since, basically, the corresponding $R^\mu$ and $S^\mu$ are essentially $L^2$ in space. Also, by lack of a uniqueness result, we have to consider martingale solutions.

%%%%%%%%%%%%%%%%%%%%%%%%%%%%%%%%%%%%%%%%%%%%%%%%%%%%%%%%%%%%%%%%%%%%%%%%
\begin{definition}[Weak martingale solution]\label{def:WeakSol} Assume \eqref{defq}, \eqref{coeff-c} and \eqref{coeff-gammaf}. Let $u_0 \in H^1(\T)$ and $v_0 \in L^2(\T)$. We say that the problem \eqref{SVWE1} admits a weak martingale solution if there exists first a stochastic basis 
\begin{equation}\label{StochasticBasis}
\left(\tilde{\Omega},\tilde{\mathcal{F}},\tilde{\Pro},\left(\tilde{\mathcal{F}}_t\right),\left(\tilde{W}(t)\right)\right),
\end{equation}
where $\left(\tilde{W}(t)\right)$ is a cylindrical Wiener process on $\mathfrak{U}$, and, second, an adapted stochastic process $(u^\mu(t))$ with value in $H^1(\T)$ such that, for all $T>0$,
\begin{enumerate}
\item  $u^\mu_t,u^\mu_x\in C([0,T]; L^2(\T))$ $\tilde{\Pro}$-a.s., and
	\begin{equation*}
		\tilde{\E}\left[\sup_{t\in[0,T]}\left(\|u^\mu_t(t)\|^2_{L^2(\T)}+\|u^\mu_x(t)\|^2_{L^2(\T)} \right)\right] \leqslant C(\mu),
	\end{equation*}

\item $\tilde{\Pro}$-a.s., $u^\mu(0,\cdot)=u_0$ and for all $\varphi\in C^1(\T)$ and $t \in [0,T]$,
\begin{gather}\nonumber
\mu \int_\T u^\mu_t(t)  \varphi \, \ud x  - \mu \int_\T v_0  \varphi \, \ud x
  +  \int_0^t \int_\T \left( c(u^\mu(s))  \varphi \right)_x c(u^\mu(s))   u^\mu_x\,  \ud x \, \ud s\\
  \label{weaku}
      + \int_0^t \int_\T \gamma(u^\mu) u_t \varphi\, \ud x\, \ud t =   \int_0^t \int_\T f(u^\mu) \varphi\, \ud x\, \ud t +  \int_\T \varphi  \Phi  \tilde{W}(t)\, \ud x ,
\end{gather}
\item $\tilde{\Pro}$-a.s., the solution $u^\mu$ satisfies the energy ``dissipation'' inequality
\begin{gather*}\nonumber
\mu\, \mathcal{E}^\mu(t_2) + 4 \mu \int_{t_1}^{t_2}\|\gamma(u^\mu) (u_t^\mu)^2\|_{L^1(\T)}\, \ud t \leqslant \mu\, \mathcal{E}^\mu(t_1)\\ 
+  2 \int_{t_1}^{t_2}\|q\|_{L^1(\T)}\,  \ud t  + 4 \mu \int_{t_1}^{t_2} \int_\T u_t^\mu f(u^\mu)\, \ud x\, \ud t + 4 \mu \int_{t_1}^{t_2}\int_\T u^\mu_t  \Phi \, \ud x \, \ud \tilde{W},
\end{gather*}
for almost all $t_1 \in [0,\infty)$ and any $t_2 \geqslant t_1$.
\item $\tilde{\Pro}$-almost surely, for almost all $t_0 \in [0,\infty)$ we have 
\begin{equation}\label{rightcontinuity}
\lim_{t \downarrow t_0} \left\| \left( u^\mu_t(t)-u^\mu_t(t_0), u^\mu_x(t)-u^\mu_x(t_0) \right) \right\|_{L^2}  =  0.
\end{equation}
\end{enumerate}
\end{definition}
%%%%%%%%%%%%%%%%%%%%%%%%%%%%%%%%%%%%%%%%%%%%%%%%%%%%%%%%%%%%%%%%%%%%%%%%

\begin{remark}
The right-continuity condition \eqref{rightcontinuity} can be interpreted as a dissipation condition. Indeed, Dafermos \cite{Dafermos} proved that, in the case of the Hunter--Saxton equation, the right-continuity condition is equivalent to the dissipation of the energy. This plays a crucial role in the uniqueness of solutions \cite{Dafermos}.
\end{remark}

Note that the stochastic basis \eqref{StochasticBasis} in Definition~\ref{def:WeakSol} depends on $\mu$ a priori. However, as long as we consider a countable collection of parameters $\mu$, a common stochastic basis can be selected. This is one of the assertions of the following result.

%%%%%%%%%%%%%%%%%%%%%%%%%%%%%%%%%%%%%%%%%%%%%%%%%%%%%%%%%%%%%%%%%%%%%%%%
\begin{thm}[Global existence of weak martingale solutions]\label{thm:global-existR2SS} Let $u_0 \in H^1(\T)$ and $v_0 \in L^2(\T)$. Assume \eqref{defq}, \eqref{coeff-c}, \eqref{coeff-gammaf} and \eqref{Positivecprime00}. Let $\Lambda$ be a countable subset of $(0,\infty)$.
Then, for every $\mu\in\Lambda$, the problem \eqref{SVWE1} admits a weak global martingale solution, with a stochastic basis \eqref{StochasticBasis} independent on $\mu\in\Lambda$. Moreover, the solution satisfies 
\begin{itemize}
\item for all $p \in [1,3)$ we have 
\begin{equation}\label{L3estimates}
\tilde{\E} \int_{[0,T] \times \T} c'(u^\mu) \left[|u^\mu_t|^p  +  |u^\mu_x|^p \right] \ud x \, \ud t  \leqslant  C(T,p,\mu),
\end{equation} 
\item for all $p \in [1,2]$, there exists $C(p)>0$ such that for all $t \in (0,T]$, we have the entropy inequality 
\begin{equation}\label{Oleinik}
\tilde{\E} \left\| \left[ \sqrt{\mu} u^\mu_t  \pm  c(u^\mu)  u^\mu_x \right]^- \right\|_{L^\infty}^p\! (t)  \leqslant  C(p,T,\mu) \left( 1  +  t^{-p} \right),
\end{equation}
\item for any $p \geqslant 1$, there exist $C(T,p)>0$ and $\mu_0=\mu_0(T) \in (0,1)$, such that for any $\mu \in (0,\mu_0)$ we have 
\begin{equation}\label{main_estimates}
\tilde{\E} \left[  \sup_{t \in [0,T]} \left( \sqrt{\mu} \mathcal{E}^\mu + \|u^\mu\|_{L^2(\T)}^2  \right) + \int_0^T \left( \| u_x^{\mu}\|_{L^2(\T)}^2 + \mu^2  \|u_t^\mu\|_{L^2(\T)}^4\right) \ud t \right]^p \leqslant C(T,p),
\end{equation}
\item for any non-negative $\psi\in C^2_c((0,T) \times \T)$, we have 
\begin{equation}\label{ene_mu0}
\lim_{\mu} \tilde{\E} \left( \int_0^T \int_\T \left( 2 \mu \gamma(u^{\mu}) (u^{\mu}_t)^2 - q \right) \psi \, \ud x\,  \ud t \right)^+ = 0,
 \end{equation}
where $q$ is defined in \eqref{defq}.
\end{itemize}
\end{thm}
%%%%%%%%%%%%%%%%%%%%%%%%%%%%%%%%%%%%%%%%%%%%%%%%%%%%%%%%%%%%%%%%%%%%%%%%

%%%%%%%%%%%%%%%%%%%%%%%%%%%%%%%%%%%%%%%%%%%%%%%%%%%%%%%%%%%%%%%%%%%%%%%%
\begin{proof}[Proof of Theorem~\ref{thm:global-existR2SS}] We will not give all the details of the existence of a global weak martingale solutions to \eqref{SVWE1} for fixed $\mu$. This is part of the reference \cite{GV25}, in which a martingale solution to problem \eqref{SVWE1} is obtained as a limit point of a sequence $(u^{\mu,\eps})$ of solutions to the truncated problem \eqref{SVWEep}. Let us comment the following points however.
\begin{itemize}
	\item This is the case without friction, \textit{i.e.} $\gamma=0$, which is treated in \cite{GV25}. Yet, the inclusion of friction with $\gamma \geqslant 0$ does not introduce additional difficulties in the proof. Indeed, the friction term has a good sign which, overall, will help to obtain better estimates. 
	Notice also that, thanks to \eqref{defRS}, the additional term $\gamma(u^\mu) u^\mu_t$ in \eqref{SVWE1} appears as a ``linear'' term in $R^\mu$ and $S^\mu$ in \eqref{Positivecprime00}, whereas the main difficulty in proving the existence theorem in \cite{GV25} arises from the quadratic terms in $R^\mu$ and $S^\mu$ in \eqref{Positivecprime00}. 
	\item In Section \ref{sec:energy} below,  we establish energy estimates for a fixed $\mu>0$ that are similar to those obtained in \cite{GV25}. In Section \ref{sec:ue}, we derive the estimates \eqref{main_estimates} and \eqref{ene_mu0} that are uniform on $\mu$, see in particular Section~\ref{sec:proof Th martingales}. 
\end{itemize}	
To obtain a stochastic basis independent on $\mu\in\Lambda$, one has simply to consider the whole collection $U^\eps:=(u^{\mu,\eps})_{\mu\in\Lambda}$. More precisely, the limiting procedure in \cite{GV25} is based on the Skorokhod--Jabukowski theorem applied, at fixed $\mu$, to the couple $(Z^{\mu,\eps},W)$, where $Z^{\mu,\eps}$ is an extended unknown based on $(R^{\mu,\eps},S^{\mu,\eps})$ considered in a quasi-Polish space $\mathcal{Z}^\mu$: see (5.17)-(5.18) in \cite{GV25} for instance (actually, $\mathcal{Z}^\mu$ is independent on $\mu$, but this does not matter here). Since a countable product of quasi-polish space, endowed with the product topology, is quasi-Polish (see Remark~5.2 in \cite{GV25}), we can consider the couple
$(\mathbf{Z}^\eps,W)$ in the space $\mathcal{Z}\times C([0,T];\mathfrak{U}_{-1})$, with
\begin{equation}\label{collective mu}
	\mathbf{Z}^\eps=(Z^{\mu,\eps})_{\mu\in\Lambda},\quad \mathcal{Z}=\prod_{\mu\in\Lambda}\mathcal{Z}^\mu.
\end{equation}
Each component $Z^{\mu,\eps}$ gives rise to a tight law in $\mathcal{Z}^\mu$, so $\mathbf{Z}^\eps$ as well. Indeed, indexing $\Lambda=\{\mu_k,k=1,2,\ldots\}$, we can select, for $\eps>0$ and each $k\geqslant 1$ a compact $K^{\mu_k}\in \mathcal{Z}^\mu$ such that $\Pro(Z^{\mu_k,\eps}\in K^{\mu_k})\geqslant 1-\eps/2^k$. The product $K$ over $\mu\in\Lambda$ of the compact sets $K^\mu$ is compact in $\mathcal{Z}$, and
\begin{equation*}
	\Pro(\mathbf{Z}^{\eps}\in K)\geqslant 1-\sum_{k\geqslant 1}\eps/2^k\geqslant 1-\eps.
\end{equation*}
We may then apply the Skorokhod-Jabukowski theorem to the couple $(\mathbf{Z}^\eps,W)$ and conclude the proof exactly as in \cite{GV25}. Indeed, each component of $\mathbf{Z}^{\eps}$ will have the desired convergence properties. We gain however the fact that the stochastic basis constructed at the end is independent on $\mu$.
\end{proof}
%%%%%%%%%%%%%%%%%%%%%%%%%%%%%%%%%%%%%%%%%%%%%%%%%%%%%%%%%%%%%%%%%%%%%%%%

%%%%%%%%%%%%%%%%%%%%%%%%%%%%%%%%%%%%%%%%%%%%%%%%%%%%%%%%%%%%%%%%%%%%%%%%
\subsection{The small-mass limit}
%%%%%%%%%%%%%%%%%%%%%%%%%%%%%%%%%%%%%%%%%%%%%%%%%%%%%%%%%%%%%%%%%%%%%%%%

	Let $(\mu_k)$ be a sequence of positive numbers decreasing to $0$ and let $\Lambda=\{\mu_k;k=1,2,\ldots\}$. We consider a sequence of martingale solutions $(u^\mu)_{\mu\in\Lambda}$ to \eqref{SVWE1}, our aim being to determine the limit of $u^{\mu_k}$ when $k\to\infty$. By Theorem \ref{thm:global-existR2SS}, such a sequence exists and is defined on a stochastic basis  
	\begin{equation}\label{Stochastic Basis again}
		\left(\tilde{\Omega},\tilde{\mathcal{F}},\tilde{\Pro},\left(\tilde{\mathcal{F}}_t\right),\left(\tilde{W}(t)\right)\right).
	\end{equation}
	We emphasize here for the last time that we cannot consider a sequence of solutions to \eqref{SVWE1} on the original stochastic basis; at any rate, we cannot justify the existence of such a sequence. This being said, and since this modification of the probabilistic data are irrelevant in what follows, \emph{we will now remove the tildas on the components of the stochastic basis} \eqref{Stochastic Basis again}. From a pratical point of view, taking such liberties with the notations will allow us to use again tildas (and not double-tildas...) when we make appeal to Skorokhod representatives to establish the limit of $u^\mu$ when $\mu\to0$ in Section~\ref{subsec:CV}.

Before we come to a result of convergence, let us specify the notion of solutions to \eqref{limeq version rhoc} and \eqref{limeq version u} which we consider, and which are similar to \cite[Definition~2.1]{HofmanovaZhang17}.

%%%%%%%%%%%%%%%%%%%%%%%%%%%%%%%%%%%%%%%%%%%%%%%%%%%%%%%%%%%%%%%%%%%%%%%%
\begin{definition}[Weak solution to the limit stochastic quasi-linear equation on $\imp$]\label{def:weak sol limit equation p} 
Let $\imp_0\in L^2(\T)$. An $(\mathcal{F}_t)$-adapted, $L^2(\T)$-valued process $(\imp(t))$ is said to be a weak solution to \eqref{limeq version rhoc} with initial datum $\imp_0$ if 
		\begin{enumerate}
			\item $\imp\in L^2(\Omega\times[0,T];H^1(\T))$, for all $T>0$,
			\item for any $\varphi\in C^\infty(\T)$, $\Pro$-almost surely,
			\begin{gather}\nonumber
				\dual{\imp(t)}{\varphi}_{L^2(\T)}-\dual{\imp_0}{\varphi}_{L^2(\T)}\\
				\nonumber
				=-\int_0^t \left(\dual{b(\imp(s))\imp_x(s)}{\varphi_x}_{L^2(\T)}+\dual{F(\cdot,\imp(s))}{\varphi}_{L^2(\T)}\right)\ud s
				+\int_0^t \dual{\Psi(\imp(s))}{\varphi}_{L^2(\T)}\, \ud W(s).
			\end{gather}
		\end{enumerate}
\end{definition}
%%%%%%%%%%%%%%%%%%%%%%%%%%%%%%%%%%%%%%%%%%%%%%%%%%%%%%%%%%%%%%%%%%%%%%%%
Similarly, we can define the notion of solutions to \eqref{limeq version u}.
%%%%%%%%%%%%%%%%%%%%%%%%%%%%%%%%%%%%%%%%%%%%%%%%%%%%%%%%%%%%%%%%%%%%%%%%
\begin{definition}[Weak solution to the limit stochastic quasi-linear equation on $u$]\label{def:weak sol limit equation u} 
Let $u_0\in L^2(\T)$. An $(\mathcal{F}_t)$-adapted, $L^2(\T)$-valued process $(u(t))$ is said to be a weak solution to \eqref{limeq version u} with initial datum $u_0$ if 
		\begin{enumerate}
			\item $u \in L^2(\Omega\times[0,T];H^1(\T))$, for all $T>0$,
			\item for any $\varphi\in C^\infty(\T)$, $\Pro$-almost surely,
			\begin{align*}\nonumber
				\dual{u(t)-u_0}{\varphi}_{L^2(\T)}
				&=-\int_0^t \left( \left\langle c(u(s))u_x(s), \left(\tfrac{c(u)}{\gamma(u)}\varphi\right)_x \right\rangle_{L^2(\T)}
				+ \left\langle \tfrac{f(u)}{\gamma(u)} -\tfrac{\gamma'(u)q}{2\gamma(u)^3} , \varphi \right\rangle_{L^2(\T)}\right)\ud s\\
				&\quad +\int_0^t \left\langle \tfrac{1}{\gamma(u(s))} \Phi, \varphi\right\rangle_{L^2(\T)} \ud W(s).
		\end{align*}
		\end{enumerate}
\end{definition}
%%%%%%%%%%%%%%%%%%%%%%%%%%%%%%%%%%%%%%%%%%%%%%%%%%%%%%%%%%%%%%%%%%%%%%%%
Following \cite[Theorem 3.1]{HofmanovaZhang17} and \cite[Theorem 6.2]{CerraiXi22} one can prove that \eqref{limeq version rhoc} admits at most one solution in the sense of Definition \ref{def:weak sol limit equation p}. Using It\^o's formula in Proposition \ref{Proposition:Ito1} we deduce also the uniqueness of solutions to \eqref{limeq version u}.    

We can now state the main theorem of this paper.

%%%%%%%%%%%%%%%%%%%%%%%%%%%%%%%%%%%%%%%%%%%%%%%%%%%%%%%%%%%%%%%%%%%%%%%%
\begin{thm}[Smoluchowski--Kramers approximation]\label{thm:SK}
	Let $T>0$. Let $(\mu_k)$ be a sequence of positive numbers decreasing to $0$ and let $\Lambda=\{\mu_k;k=1,2,\ldots\}$.
%Consider the probability space \eqref{ProbaSpace}.
Assume \eqref{defq}, \eqref{coeff-c}, \eqref{coeff-gammaf}, \eqref{c'g}, \eqref{cgamma} and \eqref{Positivecprime00}. Let $(u^\mu)_{\mu\in\Lambda}$ be a sequence of solutions of \eqref{SVWE1} satisfying
\begin{itemize}
\item $\Pro$-a.s., for all $\mu\in\Lambda$, $u^\mu(0,\cdot)=u_0$ and for all $\varphi\in C^1(\T)$ and $t \in [0,T]$,
\begin{gather}\nonumber
\mu \int_\T u^\mu_t(t)  \varphi \, \ud x  - \mu \int_\T v_0  \varphi \, \ud x
  +  \int_0^t \int_\T \left( c(u^\mu(s))  \varphi \right)_x c(u^\mu(s))   u^\mu_x\,  \ud x \, \ud s\\
  \label{weakuSK}
      + \int_0^t \int_\T \gamma(u^\mu) u_t \varphi\, \ud x\, \ud t =   \int_0^t \int_\T f(u^\mu) \varphi\, \ud x\, \ud t +  \int_\T \varphi  \Phi W(t) \, \ud x,
\end{gather}
\item for any $p \geqslant 1$, there exist $C(T,p)>0$  such that, for any $\mu \in \Lambda$, we have
\begin{equation}\label{main_estimates2}
\E \left[  \sup_{t \in [0,T]} \left( \sqrt{\mu} \mathcal{E}^\mu +  \int_\T \left(u^{\mu}\right)^2 \ud x \right) + \int_0^T \left( \| u_x^{\mu}\|_{L^2(\T)}^2 + \mu^2  \|u_t^\mu\|_{L^2(\T)}^4\right) \ud t \right]^p \leqslant C(T,p),
\end{equation}
\item for any non-negative $\psi\in C^2_c((0,T) \times \T)$, we have 
\begin{equation}\label{ene_mu02}
\lim_{k\to\infty} \E \left( \int_0^T \int_\T \left( 2 \mu \gamma(u^{\mu_k}) (u^{\mu_k}_t)^2 - q \right) \psi \, \ud x\,  \ud t \right)^+ = 0,
 \end{equation}
where $q$ is defined in \eqref{defq}.
\end{itemize}
Then, for any $\eta>0$, $p\in [1,\infty)$ and $\delta \in (0,1)$, we have 
	\begin{equation}\label{CV in proba umu}
\lim_{k\to\infty} \Pro \left(\|u^{\mu_k} - u\|_{L^2([0,T]; H^{\delta}(\T))} + \|u^{\mu_k} - u\|_{L^p([0,T]; L^2(\T))} > \eta \right)  = 0,
\end{equation}
where $u$ is the unique solution to \eqref{limeq version u} in the sense of Definition \ref{def:weak sol limit equation u}.
\end{thm}
%%%%%%%%%%%%%%%%%%%%%%%%%%%%%%%%%%%%%%%%%%%%%%%%%%%%%%%%%%%%%%%%%%%%%%%%

Note that we obtain the convergence of $u^\mu$ to $u$ and not only the convergence of a Skorokhod representation. The fact that $u^\mu$ converges in the same probability space follows in a standard way from the pathwise uniqueness of solutions to \eqref{limeq version u}, together with the Gy\"ongy--Krylov argument, \cite[Lemma 1.1]{GyongyKrylov96}, see Section \ref{sec:uniquness} below.

%%%%%%%%%%%%%%%%%%%%%%%%%%%%%%%%%%%%%%%%%%%%%%%%%%%%%%%%%%%%%%%%%%%%%%%%
\section{The energy estimates}\label{sec:energy}
%%%%%%%%%%%%%%%%%%%%%%%%%%%%%%%%%%%%%%%%%%%%%%%%%%%%%%%%%%%%%%%%%%%%%%%%

In this section we derive the standard energy estimate for \eqref{SVWEep}. This provides various bounds which may be singular when $\mu\to 0$. In the next section~\ref{sec:ue}, we will see how to exploit this first set of inequalities, and the gain due to the friction term, to derive some estimates that are uniform with respect to $\mu$.

%%%%%%%%%%%%%%%%%%%%%%%%%%%%%%%%%%%%%%%%%%%%%%%%%%%%%%%%%%%%%%%%%%%%%%%%
\begin{proposition}[Conservation law for the energy]\label{prop:energy identity} Let $\eps,\mu\in(0,1)$. The solution $(R^{\mu, \varepsilon}, S^{\mu, \varepsilon})$ to \eqref{SVWEep} satisfies the energy identity
	\begin{gather}\nonumber
		((R^{\mu, \varepsilon})^2+(S^{\mu, \varepsilon})^2)(\tau)  +  \tfrac{2}{\sqrt{\mu}} \int_0^{\tau} \tilde{c}'(u^{\mu, \varepsilon}) \left[  R^{\mu, \varepsilon}  \chi_\varepsilon(R^{\mu, \varepsilon})  +  S^{\mu, \varepsilon}  \chi_\varepsilon(S^{\mu, \varepsilon})
		\right] \ud t \\ \nonumber
		+  4 \int_0^{\tau}  \gamma(u^{\mu, \varepsilon}) \left(\uu_t^{\mu, \varepsilon} \right)^2\, \ud t =  (R^{\mu, \varepsilon})^2(0)  +  (S^{\mu, \varepsilon})^2(0)  + 
		\tfrac{1}{\sqrt{\mu}} \int_0^{\tau} \left[c(u^{\mu, \varepsilon})  ((S^{\mu, \varepsilon})^2-(R^{\mu, \varepsilon})^2)  \right]_x \ud t   \\  \label{sumsquare-theta}
		\quad +  \tfrac{2}{\mu} q^\varepsilon \tau  + 4 \int_0^{\tau} \uu_t^{\mu, \varepsilon} f(\uu^{\mu, \varepsilon})\, \ud t + 4 \int_0^{\tau}  \uu_t^{\mu, \varepsilon} \Phi^\varepsilon \, \ud W,
	\end{gather}
	for all stopping time $\tau>0$.
\end{proposition}	
%%%%%%%%%%%%%%%%%%%%%%%%%%%%%%%%%%%%%%%%%%%%%%%%%%%%%%%%%%%%%%%%%%%%%%%%	

%%%%%%%%%%%%%%%%%%%%%%%%%%%%%%%%%%%%%%%%%%%%%%%%%%%%%%%%%%%%%%%%%%%%%%%%	
\begin{proof}[Proof of Proposition~\eqref{prop:energy identity}] Note first that it follows from \eqref{udefep}, by differentiation with respect to $x$, that
\begin{equation}\label{cuxthetaeps}
c(u^{\mu, \varepsilon}) u^{\mu, \varepsilon}_x
=\tfrac{S^{\mu, \varepsilon}-R^{\mu, \varepsilon}}{2}-\Theta^{\mu, \varepsilon}.
\end{equation}
Let $h\in C^2(\R)$. By the It\^o formula, we deduce from \eqref{SVWEep} that
\begin{gather}\nonumber
\sqrt{\mu}\, \ud h(R^{\mu, \varepsilon})  +  c(u^{\mu, \varepsilon})  h(R^{\mu, \varepsilon})_x \, \ud t  + h'(R^{\mu, \varepsilon}) \gamma(u^{\mu, \varepsilon})\uu_t^{\mu, \varepsilon} \, \ud t =\\  \nonumber  
\tilde{c}'(u^{\mu, \varepsilon}) \left[h'(R^{\mu, \varepsilon})  ((R^{\mu, \varepsilon})^2  -  (S^{\mu, \varepsilon})^2  -  \chi_\varepsilon(R^{\mu, \varepsilon})  +2  R^{\mu, \varepsilon}  \Theta^{\mu, \varepsilon} )   \right]  \ud t +  \tfrac{1}{2\sqrt{\mu}}  q^\varepsilon  h''(R^{\mu, \varepsilon})\, \ud t \\ \nonumber
+   h'(R^{\mu, \varepsilon}) f(\uu^{\mu, \varepsilon})\, \ud t +  h'(R^{\mu, \varepsilon})  \Phi^\varepsilon \, \ud W, 
\end{gather}
and
\begin{gather}\nonumber
\sqrt{\mu}\, \ud h(S^{\mu, \varepsilon})  -  c(u^{\mu, \varepsilon})  h(S^{\mu, \varepsilon})_x \, \ud t  + h'(S^{\mu, \varepsilon}) \gamma(u^{\mu, \varepsilon}) \uu_t^{\mu, \varepsilon} \, \ud t= \\  \nonumber \tilde{c}'(u^{\mu, \varepsilon}) \left[h'(S^{\mu, \varepsilon})  ((S^{\mu, \varepsilon})^2  -  (R^{\mu, \varepsilon})^2  -  \chi_\varepsilon(S^{\mu, \varepsilon})  -2  S^{\mu, \varepsilon}  \Theta^{\mu, \varepsilon} )  \right]  \ud t +  \tfrac{1}{2\sqrt{\mu}}  q^\varepsilon  h''(S^{\mu, \varepsilon}) \, \ud t \\ \nonumber
 +   h'(S^{\mu, \varepsilon}) f(\uu^{\mu, \varepsilon})\, \ud t +  h'(S^{\mu, \varepsilon})  \Phi^\varepsilon \, \ud W.
\end{gather}
We can then use the identity \eqref{cuxthetaeps} to obtain the conservative form
\begin{gather}\nonumber
\sqrt{\mu}\, \ud h(R^{\mu, \varepsilon})  + \left( c(u^{\mu, \varepsilon})  h(R^{\mu, \varepsilon})\right)_x  \ud t  + h'(R^{\mu, \varepsilon}) \gamma(u^{\mu, \varepsilon}) \uu_t^{\mu, \varepsilon} \, \ud t =\\ \nonumber
  \tilde{c}'(u^{\mu, \varepsilon}) \left[(S^{\mu, \varepsilon}  -  R^{\mu, \varepsilon}) B_h(R^{\mu, \varepsilon},S^{\mu, \varepsilon})  -  h'(R^{\mu, \varepsilon}) \chi_\varepsilon(R^{\mu, \varepsilon})   \right]\ud t
  +  \tfrac{1}{2\sqrt{\mu}}  q^\varepsilon  h''(R^{\mu, \varepsilon})\, \ud t\\ \label{ReqepthetaItoConservative} 
+ 2  \tilde{c}'(u^{\mu, \varepsilon})  (R^{\mu, \varepsilon}   h'(R^{\mu, \varepsilon} )  -  2  h(R^{\mu, \varepsilon} ))  \Theta^{\mu, \varepsilon}  \, \ud t  +   h'(R^{\mu, \varepsilon}) f(\uu^{\mu, \varepsilon})\, \ud t +  h'(R^{\mu, \varepsilon})  \Phi^\varepsilon \, \ud W, 
\end{gather}
and
\begin{gather}\nonumber
\sqrt{\mu}\, \ud h(S^{\mu, \varepsilon})  - \left( c(u^{\mu, \varepsilon})  h(S^{\mu, \varepsilon}) \right)_x  \ud t  + h'(S^{\mu, \varepsilon}) \gamma(u^{\mu, \varepsilon}) \uu_t^{\mu, \varepsilon} \, \ud t = \\ \nonumber \tilde{c}'(u^{\mu, \varepsilon}) \left[(R^{\mu, \varepsilon}  -  S^{\mu, \varepsilon}) B_h(S^{\mu, \varepsilon},R^{\mu, \varepsilon})  -  h'(S^{\mu, \varepsilon}) \chi_\varepsilon(S^{\mu, \varepsilon})  \right] \ud t  +  \tfrac{1}{2\sqrt{\mu}}  q^\varepsilon h''(S^{\mu, \varepsilon})\, \ud t \\ 
 \label{SeqepthetaItoConservative}
- 2  \tilde{c}'(u^{\mu, \varepsilon})  (S^{\mu, \varepsilon}   h'(S^{\mu, \varepsilon} )  -  2  h(S^{\mu, \varepsilon} ))  \Theta^{\mu, \varepsilon}  \, \ud t  +   h'(S^{\mu, \varepsilon}) f(\uu^{\mu, \varepsilon})\, \ud t +  h'(S^{\mu, \varepsilon})  \Phi^\varepsilon \, \ud W,
\end{gather}
where
\begin{equation*}
B_h(R,S)  \eqdef  2  h(R)  -  (R+S)  h'(R).
\end{equation*}
Taking $h(R)=R^2$, we obtain 
\begin{equation*}
B_h(R,S)=-2RS=B_h(S,R),\quad R  h'(R)  -  2  h(R)=0,
\end{equation*}
and, by adding \eqref{ReqepthetaItoConservative} to \eqref{SeqepthetaItoConservative} and using the identity \eqref{ut by RSep}, we get \eqref{sumsquare-theta}.
\end{proof}	
%%%%%%%%%%%%%%%%%%%%%%%%%%%%%%%%%%%%%%%%%%%%%%%%%%%%%%%%%%%%%%%%%%%%%%%%

%%%%%%%%%%%%%%%%%%%%%%%%%%%%%%%%%%%%%%%%%%%%%%%%%%%%%%%%%%%%%%%%%%%%%%%%
\begin{proposition}[Correction terms]\label{prop:ut} Let $\Theta^{\mu, \varepsilon}$ be defined by \eqref{psieps}. Set also
\begin{equation}\label{xiut}
\Xi^{\mu, \varepsilon}  \eqdef \sqrt{\mu} \left(u^{\mu, \varepsilon}_t-\uu^{\mu, \varepsilon}_t\right)  =\sqrt{\mu}  u^{\mu, \varepsilon}_t  -  \tfrac{R^{\mu, \varepsilon}+S^{\mu, \varepsilon}}{2}.
\end{equation}
We have then
\begin{equation}\label{psiprime}
\sqrt{\mu} \frac{\ud\ }{\ud t} \Theta^{\mu, \varepsilon}(t)=\alpha_\chi^{\mu, \varepsilon}(t)+\beta^{\mu, \varepsilon}(t)\Theta^{\mu, \varepsilon} (t),
\end{equation}
where
\begin{equation*}
\alpha_\chi^{\mu, \varepsilon}(t)=\tfrac12\int_\T \tilde{c}'(u^{\mu, \varepsilon})\left(\chi_\varepsilon(R^{\mu, \varepsilon})-\chi_\varepsilon(S^{\mu, \varepsilon})\right) \ud y,
\quad\beta^{\mu, \varepsilon}(t)=\int_\T \tilde{c}'(u^{\mu, \varepsilon})\left(R^{\mu, \varepsilon}+S^{\mu, \varepsilon}\right) \ud y,
\end{equation*}
and
\begin{equation}\label{utOK}
\Xi^{\mu, \varepsilon} =\tfrac{1}{c(u^{\mu, \varepsilon})}\int_0^x \left[\zeta^{\mu, \varepsilon}(t,y)-\bar{\zeta}^{\mu, \varepsilon}(t)\right]\ud y,
\end{equation}
where
\begin{equation}\label{zeta}
\zeta^{\mu, \varepsilon}=\tilde{c}'(u^{\mu, \varepsilon})\left[\tfrac{\chi_\varepsilon(R^{\mu, \varepsilon})-\chi_\varepsilon(S^{\mu, \varepsilon})}{2}+(R^{\mu, \varepsilon}+S^{\mu, \varepsilon})\Theta^{\mu, \varepsilon} \right],\quad \bar{\zeta}^{\mu, \varepsilon}(t)=\int_0^1\zeta^{\mu, \varepsilon}(t,y) \, \ud y.
\end{equation}
\end{proposition}
%%%%%%%%%%%%%%%%%%%%%%%%%%%%%%%%%%%%%%%%%%%%%%%%%%%%%%%%%%%%%%%%%%%%%%%%

%%%%%%%%%%%%%%%%%%%%%%%%%%%%%%%%%%%%%%%%%%%%%%%%%%%%%%%%%%%%%%%%%%%%%%%%
\begin{proof}[Proof of Proposition~\ref{prop:ut}] We consider the equations \eqref{ReqepthetaItoConservative}-\eqref{SeqepthetaItoConservative} with $h(R)=R$, and subtract the first equation from the second one to obtain
\begin{multline}\label{SRconservative}
\sqrt{\mu} \partial_t(S^{\mu, \varepsilon}-R^{\mu, \varepsilon})-(c(u^{\mu, \varepsilon})(R^{\mu, \varepsilon}+S^{\mu, \varepsilon}))_x\\
=\tilde{c}'(u^{\mu, \varepsilon})\left[(\chi_\varepsilon(R^{\mu, \varepsilon})-\chi_\varepsilon(S^{\mu, \varepsilon}))+2  (R^{\mu, \varepsilon}+S^{\mu, \varepsilon})  \Theta^{\mu, \varepsilon} \right].
\end{multline}
Since $R^{\mu, \varepsilon}, S^{\mu, \varepsilon}$ and $u^{\mu, \varepsilon}$ are $1$-periodic, we obtain \eqref{psiprime}
by integration on $\T$. Then we observe that, taking $x=0$ in \eqref{udefep} gives
\begin{equation}\label{ut0}
u^{\mu, \varepsilon}(t,0)=\int_0^t  \left(\tfrac{R^{\mu, \varepsilon}  +  S^{\mu, \varepsilon}}{2\sqrt{\mu}}  \right) (s,0) \, \ud s  +  u_0^\varepsilon(0),
\end{equation}
so that
\begin{equation}\label{CCu}
\mathcal{C}(u^{\mu, \varepsilon}(t,x))=\mathcal{C}(u^{\mu, \varepsilon}(t,0))+\int_0^x \left[\tfrac{S^{\mu, \varepsilon}-R^{\mu, \varepsilon}}{2}(t,y)-\Theta^{\mu, \varepsilon} (t)\right]\ud y.
\end{equation}
By differentiation in \eqref{ut0} and \eqref{CCu} with respect to $t$, we obtain
\begin{equation*}
c(u^{\mu, \varepsilon})u^{\mu, \varepsilon}_t=c(u^{\mu, \varepsilon}(t,0))\left(\tfrac{R^{\mu, \varepsilon}+S^{\mu, \varepsilon}}{2\sqrt{\mu}}  \right) (t,0)+\int_0^x \left[\tfrac{\partial_t(S^{\mu, \varepsilon}-R^{\mu, \varepsilon})}{2}(t,y)- \frac{\ud \ }{\ud t}\Theta^{\mu, \varepsilon} (t)\right]\ud y.
\end{equation*}
We use \eqref{SRconservative} and \eqref{psiprime} to get
\begin{equation*}
\sqrt{\mu} c(u^{\mu, \varepsilon})u^{\mu, \varepsilon}_t=c(u^{\mu, \varepsilon})\tfrac{R^{\mu, \varepsilon}+S^{\mu, \varepsilon}}{2}+\int_0^x \left[\zeta^{\mu, \varepsilon}(t,y)-\bar{\zeta}^{\mu, \varepsilon}(t)\right]\ud y,
\end{equation*}
where $\zeta^{\mu, \varepsilon}$ and $\bar{\zeta}^{\mu, \varepsilon}$ are defined in \eqref{zeta}. Dividing by $c(u^{\mu, \varepsilon})$ yields \eqref{utOK}.
\end{proof}
%%%%%%%%%%%%%%%%%%%%%%%%%%%%%%%%%%%%%%%%%%%%%%%%%%%%%%%%%%%%%%%%%%%%%%%%

%%%%%%%%%%%%%%%%%%%%%%%%%%%%%%%%%%%%%%%%%%%%%%%%%%%%%%%%%%%%%%%%%%%%%%%%
\begin{proposition}[Energy estimate]\label{prop:GlobalEnergyep} Let $(R^{\mu, \varepsilon}, S^{\mu, \varepsilon})$ be the solution to \eqref{SVWEep} given by Theorem~\ref{thm:glo-existRSep}. Let
\begin{equation*}
\mathcal{E}^{\mu, \varepsilon}(t)\eqdef \|(R^{\mu, \varepsilon},S^{\mu, \varepsilon})(t)\|_{L^2(\T)}^2  =  \int_\T ((R^{\mu, \varepsilon})^2  +  (S^{\mu, \varepsilon})^2)(t,x) \, \ud x
\end{equation*}
denote the total energy of the system. Then for any $T>0$ and $p \in [1, \infty)$ there exists $C>0$ such that for all $(\varepsilon,\mu) \in (0,1)^2$ we have 
\begin{gather}\label{eq:TotalEnergyep}
\E \left[ \sup_{t \in [0,T]}  \mathcal{E}^{\mu, \varepsilon}(t) +  \int_0^T \| \uu_t^{\mu, \varepsilon}\|_{L^2(\T)}^{2}\, \ud t 
+ \tfrac{1}{\sqrt{\mu}}  \int_0^{T} \int_\T \left[  R^{\mu, \varepsilon}  \chi_\varepsilon(R^{\mu, \varepsilon})  +  S^{\mu, \varepsilon}  \chi_\varepsilon(S^{\mu, \varepsilon})
 \right] \ud x\, \ud t  \right]^p \leqslant \tfrac{C}{\mu^p}.
\end{gather}
Moreover, there exists a constant $\eta_0$, such that for all $\eta \leqslant \eta_0$, we have
\begin{equation}\label{EXPene}
\E\left[\exp\left(\eta\sup_{0 \leqslant s\leqslant T} \mathcal{E}^{\mu, \varepsilon}(s)\right)\right] \leqslant  \mathtt{C}(T,\mu).
\end{equation}
\end{proposition}
%%%%%%%%%%%%%%%%%%%%%%%%%%%%%%%%%%%%%%%%%%%%%%%%%%%%%%%%%%%%%%%%%%%%%%%%

%%%%%%%%%%%%%%%%%%%%%%%%%%%%%%%%%%%%%%%%%%%%%%%%%%%%%%%%%%%%%%%%%%%%%%%%
\begin{remark}[Terminology] Note that \eqref{eq:TotalEnergyep} contains actually various information. First, a bound on the \emph{wave energy} (sum of the kinetic and potential energy)
	\begin{equation}\label{Wave energy ep}
		\E \left[ \sup_{t \in [0,T]}  \mu \mathcal{E}^{\mu, \varepsilon}(t)  \right]^p \leqslant C(T,p),
	\end{equation}
	but also a bound on the \emph{frictional energy}
	\begin{equation}\label{Frictional energy}
		\E \left[ \int_0^T \mu \| \uu_t^{\mu, \varepsilon}\|_{L^2(\T)}^{2}\, \ud t \right]^p \leqslant C(T,p),
	\end{equation}
	and a bound on the truncation terms
	\begin{equation}\label{Truncation Energy}
		\E\left[ \sqrt{\mu} \int_0^{T} \int_\T \left[  R^{\mu, \varepsilon}  \chi_\varepsilon(R^{\mu, \varepsilon})  +  S^{\mu, \varepsilon}  \chi_\varepsilon(S^{\mu, \varepsilon})
		\right] \ud x\, \ud t  \right]^p \leqslant C(T,p).
	\end{equation}
\end{remark}
%%%%%%%%%%%%%%%%%%%%%%%%%%%%%%%%%%%%%%%%%%%%%%%%%%%%%%%%%%%%%%%%%%%%%%%%

%%%%%%%%%%%%%%%%%%%%%%%%%%%%%%%%%%%%%%%%%%%%%%%%%%%%%%%%%%%%%%%%%%%%%%%%
\begin{proof}[ Proof of Proposition~\ref{prop:GlobalEnergyep}] We will denote by $C$ any constant that may depend on the  parameters $c_1$, $c_2$, $c_3$, $\gamma_1$, $\gamma_2$, $L$ (see \eqref{coeff-c}-\eqref{coeff-gammaf}), as well as $q_0$ (\textit{cf.} \eqref{defq}), $\|u_0\|_{H^1(\T)}$, $\|v_0\|_{L^2(\T)}$ and the final time $T$. The constant $C$ being however independent of $\eps$ and $\mu$. \medskip

\textbf{Proof of the moments bound \eqref{eq:TotalEnergyep}.} We integrate over $\T$ the conservation law for the energy \eqref{sumsquare-theta} to obtain
		\begin{multline} \label{EnergyBalance}
			\mu\, \mathcal{E}^{\mu, \varepsilon}(t)+ 4 \mu \tilde{\Gamma}^{\mu, \varepsilon}(t)
			 + \mathcal{D}_\chi^{\mu, \varepsilon}(t)
			 \\
			= \mu\, \mathcal{E}^{\mu, \varepsilon}(0)
			+2\|q^\varepsilon\|_{L^1(\T)}\,  t+ 4 \mu \int_0^t \left(\int_\T \uu^{\mu, \varepsilon}_t f(\uu^{\mu, \varepsilon})\, \ud x\right) \ud s +	\mathcal{M}^{\mu, \varepsilon}(t),
		\end{multline}
		where
		\begin{equation*}
		\tilde{\Gamma}^{\mu, \varepsilon}(t) \eqdef \int_0^t \int_\T \gamma(u^{\mu,\varepsilon})\left( \uu_t^{\mu, \varepsilon}  \right)^2 \ud x\, \ud s \geqslant \gamma_1 			\Gamma^{\mu, \varepsilon}(t) \eqdef \gamma_1 \int_0^t \left\| \uu_t^{\mu, \varepsilon}  \right\|_{L^2(\T)}^2 \ud s,
		\end{equation*}
		and
		\begin{equation*}
\mathcal{D}_\chi^{\mu, \varepsilon}(t) \eqdef 2 \sqrt{\mu} \int_0^t \left(\int_\T \tilde{c}'(u^{\mu, \varepsilon}) \left[  R^{\mu, \varepsilon}  \chi_\varepsilon(R^{\mu, \varepsilon})  +  S^{\mu, \varepsilon}  \chi_\varepsilon(S^{\mu, \varepsilon})
\right] \ud x\right)  \ud s,
		\end{equation*}
		and  the martingale $\mathcal{M}^{\mu, \varepsilon}(t) $ is given by 
		\begin{equation*}
			\mathcal{M}^{\mu, \varepsilon}(t)  \eqdef 4 \mu \int_0^t \int_\T \sum_{k \geqslant 1 }  \uu_t^{\mu, \varepsilon}(s,x)  \sigma_k^{\mu, \varepsilon}(x)\,  \ud x\,  \ud \beta_k(s).
		\end{equation*}
Next, we use the bounds
\begin{equation*}
	\int_\T \uu^{\mu, \varepsilon}_t f(\uu^{\mu, \varepsilon})\, \ud x
	\leqslant \tfrac{\gamma_1}{2}\left\| \uu_t^{\mu, \varepsilon}  \right\|_{L^2(\T)}^2+ \tfrac{1}{2\gamma_1}\left\| f(\uu^{\mu, \varepsilon}) \right\|_{L^2(\T)}^2,
\end{equation*}		
and (recall that $\Lip(f)\leqslant L$, \textit{cf.} \eqref{coeff-gammaf})
\begin{equation*}
	|f(u)|^2\leqslant 2 f(0)^2+2 L^2|u^2|,
\end{equation*}
to get 
\begin{equation*}
	\int_\T \uu^{\mu, \varepsilon}_t f(\uu^{\mu, \varepsilon})\, \ud x
	\leqslant \tfrac{\gamma_1}{2}\left\| \uu_t^{\mu, \varepsilon}  \right\|_{L^2(\T)}^2+ \tfrac{1}{\gamma_1} f(0)^2+ \tfrac{L^2}{\gamma_1}\left\| \uu^{\mu, \varepsilon} \right\|_{L^2(\T)}^2,
\end{equation*}	
We expand (with the help of \eqref{wdefep}, which gives $\uu^{\mu, \varepsilon}(0)=u^{\varepsilon}(0)$)
\begin{equation*}
	\uu^{\mu, \varepsilon}(t)=u^{\varepsilon}(0)+\int_0^t \uu^{\mu, \varepsilon}_t\, \ud s\,,
\end{equation*}
to get
\begin{equation*}
	 \int_0^t \left\| \uu^{\mu, \varepsilon} \right\|_{L^2(\T)}^2 \ud s
	 \leqslant 2\left\| u^{ \varepsilon}(0) \right\|_{L^2(\T)}^2 t +2\int_0^t \Gamma^{\mu, \varepsilon}(s)\, \ud s,
\end{equation*}
and  
\begin{multline*}
	\mu\, \mathcal{E}^{\mu, \varepsilon}(t)+ 2 \mu\gamma_1 \Gamma^{\mu, \varepsilon}(t)
	+ \mathcal{D}_\chi^{\mu, \varepsilon}(t)
	\\
	\leqslant 
		\mu\, \mathcal{E}^{\mu, \varepsilon}(0)
	+2\left(\|q^\varepsilon\|_{L^1(\T)}+4 L^2 \tfrac{\mu}{\gamma_1}\left\| u^{ \varepsilon}(0) \right\|_{L^2(\T)}^2 + 2 \tfrac{\mu}{\gamma_1}f(0)^2\right) t
	\\
	+  \tfrac{8 \mu L^2}{\gamma_1}  \int_0^t \Gamma^{\mu, \varepsilon}(s) \, \ud s +	\mathcal{M}^{\mu, \varepsilon}(t).
\end{multline*}
Since $\|q^\varepsilon\|_{L^1(\T)}\leqslant C$ and $\left\| u^{ \varepsilon}(0) \right\|_{L^2(\T)}^2\leqslant C$, we obtain, for $t\in[0,T]$,
\begin{equation*}
	\mu\, \mathcal{E}^{\mu, \varepsilon}(t)+ 2 \mu\gamma_1 \Gamma^{\mu, \varepsilon}(t)
	+ \mathcal{D}_\chi^{\mu, \varepsilon}(t)
	\leqslant C+ C \mu\int_0^t \Gamma^{\mu, \varepsilon}(s)\, \ud s +\left|\mathcal{M}^{\mu, \varepsilon}(t)\right|.
\end{equation*}
By the Gr\"onwall Lemma, we deduce, for $0\leqslant t\leqslant T$
\begin{equation*}
	\mu\, \mathcal{E}^{\mu, \varepsilon}(t)+ \mu\Gamma^{\mu, \varepsilon}(t)
	+ \mathcal{D}_\chi^{\mu, \varepsilon}(t)
	\leqslant C +\left|\mathcal{M}^{\mu, \varepsilon}(t)\right|+C\int_0^t e^{Cs}\left|\mathcal{M}^{\mu, \varepsilon}(s)\right|\ud s,
\end{equation*}
and thus
\begin{equation}\label{EnergyIneq5}
	\mu\, \sup_{t \in [0,T]}\mathcal{E}^{\mu, \varepsilon}(t)+ \mu\Gamma^{\mu, \varepsilon}(T)
	+ \mathcal{D}_\chi^{\mu, \varepsilon}(T)
	\leqslant 
	C\left(1+ \sup_{t \in [0,T]} \left|\mathcal{M}^{\mu, \varepsilon}(t)\right|\right).
\end{equation}
Using It\^o's isometry, the Burkholder--Davis--Gundy inequality, and the bound \eqref{defq}, we have, for an exponent $p\geqslant 1$,
\begin{equation*}
	\E\left[ \sup_{t \in [0,T]} [\mathcal{M}^{\mu, \varepsilon}(t)]^{p}\right]  \leqslant  C \mu^{p} \E \left( \int_{0}^{T}  \int_\T (\uu_t^{\mu, \varepsilon})^2 \, \ud x    \, \ud s \right)^{p/2}.
\end{equation*}
By the definition \eqref{wdefep} of $\uu^{\mu, \varepsilon}$, we also have 
\begin{equation*}
	\mu  \left\| \uu_t^{\mu, \varepsilon}  \right\|_{L^2(\T)}^2\leqslant C \mathcal{E}^{\mu, \varepsilon}(t),
\end{equation*}
so 
\begin{align}\label{EsupM2}
	\E\left[ \sup_{t \in [0,T]} [\mathcal{M}^{\mu, \varepsilon}(s)]^{p}\right]  \leqslant  C \E\left[
		\mu
	\sup_{t \in [0,T]} \mathcal{E}^{\mu, \varepsilon}(t)
\right]^{p/2}.
\end{align}
The estimate \eqref{eq:TotalEnergyep} follows from \eqref{EnergyIneq5} and \eqref{EsupM2}.\medskip

\textbf{Proof of the exponential bound \eqref{EXPene}.} Let $\eta\in (0,1]$. 
	Let $C_0$ be the constant in \eqref{EnergyIneq5}. Then, \eqref{EnergyIneq5} gives
	\begin{equation}\label{EXPenergy1}
		\mu \eta\sup_{0 \leqslant t \leqslant T} \mathcal{E}^{\mu, \varepsilon}(t)\leqslant C+\eta C_0
		\sup_{0 \leqslant t \leqslant T}|\mathcal{M}^{\mu, \varepsilon}(t)|.
	\end{equation}
	Let 
	\begin{equation*}
		\dual{\mathcal{M}^{\mu, \varepsilon}}{\mathcal{M}^{\mu, \varepsilon}}(t)  \eqdef  16 \mu^2 \int_0^t \sum_{k\geqslant 1 } \left|\int_\T \sigma_k^\varepsilon \uu_t^{\mu, \varepsilon} \, \ud x\right|^2 \ud s
	\end{equation*} 
	denote the quadratic variation of $\mathcal{M}^{\mu, \varepsilon}$. We use the exponential martingale inequality
	\begin{equation*}
		\Pro\left(\sup_{0 \leqslant t \leqslant  T} |\mathcal{M}^{\mu, \varepsilon}(t)|
		\geqslant \left(a+b\dual{\mathcal{M}^{\mu, \varepsilon}}{\mathcal{M}^{\mu, \varepsilon}}(T)\right)\lambda\right)
		\leqslant 2 \exp\left\{-2a b \lambda^2\right\},
	\end{equation*}
	with $\lambda=1$ to obtain
	\begin{equation}\label{OKZ}
		\E\left[\exp\left(2\eta C_0 Z\right)\right]  \leqslant  C, \qquad Z  \eqdef  \sup_{0 \leqslant t \leqslant  T} |\mathcal{M}^{\mu, \varepsilon}(t)|-b \dual{\mathcal{M}^{\mu, \varepsilon}}{\mathcal{M}^{\mu, \varepsilon}}( T),
	\end{equation}
	if $\eta\leqslant b/(2 C_0)$. 
	By \eqref{EXPenergy1} and the Cauchy--Schwarz inequality, we have then
	\begin{equation*}
		\E\left[\exp\left( \mu \eta\sup_{0 \leqslant t \leqslant  T} \mathcal{E}^{\mu, \varepsilon}(t)\right)\right]
		\leqslant C \left\{\E\left[\exp\left(2\eta C_0 Z\right)\right]\right\}^{1/2}
		\left\{\E\left[\exp\left(2b\eta C_0 \dual{\mathcal{M}^{\mu, \varepsilon}}{\mathcal{M}^{\mu, \varepsilon}}( T)\right)\right]\right\}^{1/2}.
	\end{equation*}
	The bound \eqref{OKZ} therefore gives
	\begin{equation*}
		\E\left[\exp\left(\mu \eta\sup_{0 \leqslant t \leqslant  T} \mathcal{E}^{\mu, \varepsilon}(t)\right)\right]
		\leqslant C
		\left\{\E\left[\exp\left(2b\eta C_0 \dual{\mathcal{M}^{\mu, \varepsilon}}{\mathcal{M}^{\mu, \varepsilon}}( T)\right)\right]\right\}^{1/2}.
	\end{equation*}
	By the Cauchy--Schwarz inequality, we also have 
	\begin{equation*}
		\dual{\mathcal{M}^{\mu, \varepsilon}}{\mathcal{M}^{\mu, \varepsilon}}( T)
		\leqslant 8 q_0 \mu  \mathcal{E}^{\mu, \varepsilon}( T)\leqslant  8 q_0 \mu \sup_{0 \leqslant t\leqslant  T} \mathcal{E}^{\mu, \varepsilon}(t),
	\end{equation*}
	so 
	\begin{equation*}
		\E\left[\exp\left(\mu \eta\sup_{0 \leqslant t \leqslant  T} \mathcal{E}^{\mu, \varepsilon}(t)\right)\right]
		\leqslant C 
		\left\{\E\left[\exp\left(16 b q_0 C_0  \eta \mu \sup_{0 \leqslant t\leqslant  T} \mathcal{E}^{\mu, \varepsilon}(t)\right)\right]\right\}^{1/2}.
	\end{equation*}
	We can choose $b=1/(16 C_0 q_0)$ to conclude to \eqref{EXPene}, under the condition $32 C_0^2 q_0\eta\leqslant 1$. 
\end{proof}
%%%%%%%%%%%%%%%%%%%%%%%%%%%%%%%%%%%%%%%%%%%%%%%%%%%%%%%%%%%%%%%%%%%%%%%%

%%%%%%%%%%%%%%%%%%%%%%%%%%%%%%%%%%%%%%%%%%%%%%%%%%%%%%%%%%%%%%%%%%%%%%%%
\begin{corollary}[Bound on the correction terms]\label{cor:boundpsi} 
Let 
$(R^{\mu, \varepsilon}(t),S^{\mu, \varepsilon}(t))_{t\geqslant 0}$ be the solution to \eqref{SVWEep}.
The corrective term
$\Theta^{\mu, \varepsilon}$ defined in \eqref{psieps} satisfies the bound
\begin{equation}\label{eq:boundpsi}
\E\left[\|\Theta^{\mu, \varepsilon}\|_{L^\infty([0, T])}^p\right]\leqslant C (p,T,\mu)  \varepsilon^{1/2},
\end{equation}
for all $p\geqslant 1$, while, for $\uu^{\mu, \varepsilon}$ given by \eqref{wdefep}, we have
\begin{equation}\label{uteps0}
\E\left[\left\|u^{\mu, \varepsilon}_t-\uu^{\mu, \varepsilon}_t\right\|^p_{L^1([0, T];L^\infty(\T))}\right]\leqslant  C (T,p,\mu)  \varepsilon^{1/4}.
\end{equation}
and
\begin{equation}\label{uteps01}
\E\left[\left\|u^{\mu, \varepsilon}_t-\uu^{\mu, \varepsilon}_t\right\|^p_{L^\infty([0,T] \times \T)}\right]\leqslant  C (T,p,\mu),
\end{equation}
for all $p\geqslant 1$.
\end{corollary}
%%%%%%%%%%%%%%%%%%%%%%%%%%%%%%%%%%%%%%%%%%%%%%%%%%%%%%%%%%%%%%%%%%%%%%%%

%%%%%%%%%%%%%%%%%%%%%%%%%%%%%%%%%%%%%%%%%%%%%%%%%%%%%%%%%%%%%%%%%%%%%%%%
\begin{proof}[Proof of Corollary~\ref{cor:boundpsi}] We assume $2\leqslant p$. By integration in \eqref{psiprime}, we obtain
\begin{equation*}
\Theta^{\mu, \varepsilon}(t)  =  \int_0^t \tfrac{\alpha^{\mu, \varepsilon}_\chi(s)}{\sqrt{\mu}} \exp\left\{\int_s^t \tfrac{\beta^{\mu, \varepsilon}(\sigma)}{\sqrt{\mu}}\, \ud\sigma \right\}  \ud s,
\end{equation*}
which gives
\begin{equation*}
\|\Theta^{\mu, \varepsilon}\|_{L^\infty([0, T])}^p  \leqslant  \int_0^{T} \left|\tfrac{\alpha^{\mu, \varepsilon}_\chi(s)}{\sqrt{\mu}}\right|^p  \ud s  \left(\int_0^{ T} \exp\left\{p' \int_s^t \left| \tfrac{\beta^{\mu, \varepsilon}(\sigma)}{\sqrt{\mu}}\right| \ud\sigma \right\}  \ud s\right)^{p/p'},
\end{equation*}
(where $p'$ is the conjugate exponent to $p$), and then
\begin{multline}\label{psieps1}
\E\left[\|\Theta^{\mu, \varepsilon}\|_{L^\infty([0,T])}^p\right]\\ 
\leqslant 
\left\{\E\left[\int_0^{ T} \left|\tfrac{\alpha^{\mu, \varepsilon}_\chi(s)}{\sqrt{\mu}}\right|^{2p} \ud s\right]\right\}^{1/2}
\left\{\E\left[\left(\int_0^{ T} \exp\left\{p'\int_s^{ T}\left| \tfrac{\beta^{\mu, \varepsilon}(\sigma)}{\sqrt{\mu}}\right| \ud\sigma\right\} \ud s\right)^{(2p)/p'}\right]\right\}^{1/2}.
\end{multline}
Using  the estimate 
\begin{equation}\label{truncation gets eps}
	\chi_\eps(R)\leqslant R\chi_\eps(R)\eps,
\end{equation}
we obtain 
\begin{multline}\label{chiep}
\int_0^T \int_\T \tilde{c}'(u^{\mu, \varepsilon})\left(\chi_\varepsilon(R^{\mu, \varepsilon})  + \chi_\varepsilon(S^{\mu, \varepsilon})  \right) \ud x \, \ud t \\\leqslant \varepsilon \int_0^{T} \int_\T \tilde{c}'(u^{\mu, \varepsilon}) \left[  R^{\mu, \varepsilon}  \chi_\varepsilon(R^{\mu, \varepsilon})  +  S^{\mu, \varepsilon}  \chi_\varepsilon(S^{\mu, \varepsilon})
 \right] \ud x\, \ud t,
\end{multline}
which implies with the bound on $\mathcal{D}_\chi^{\mu, \varepsilon}(t)$ in \eqref{EnergyIneq5} that 
\begin{equation*}
	\int_0^{T}|\alpha^{\mu, \varepsilon}_\chi(s)| \, \ud s \leqslant C (T,\mu)\left[1+\sup_{t \in [0,T]}\left|\mathcal{M}^{\mu, \varepsilon}(s)\right| \right].
\end{equation*}
Using now $\chi_\varepsilon(R) \leqslant R^2$ with the bound on $	\mathcal{E}^{\mu, \varepsilon}(t)$ in \eqref{EnergyIneq5}, we get
\begin{equation*}
 \sup_{s \in [0,T]} |\alpha^{\mu, \varepsilon}_\chi(s)|^{2p-1}  
\leqslant  C (T,p,\mu)\left[1+\sup_{t \in [0,T]}\left|\mathcal{M}^{\mu, \varepsilon}(s)\right|^{2p-1} \right].
\end{equation*}
All in all, we have
\begin{equation*}
\int_0^{T}|\alpha^{\mu, \varepsilon}_\chi(s)|^{2p} \, \ud s \leqslant  \sup_{s \in [0,T]} |\alpha^{\mu, \varepsilon}_\chi(s)|^{2p-1}  \int_0^{T}|\alpha^{\mu, \varepsilon}_\chi(s)|  \, \ud s
\leqslant  C (T,p,\mu)\left[1+\sup_{t \in [0,T]}\left|\mathcal{M}^{\mu, \varepsilon}(s)\right|^{2p} \right] \eps.
\end{equation*}
By \eqref{EsupM2} and \eqref{eq:TotalEnergyep} we deduce
\begin{equation}\label{alphaeps3}
\E \int_0^{T}|\alpha^{\mu, \varepsilon}_\chi(s)|^{2p} \, \ud s
\leqslant  C (T,p,\mu) \eps.
\end{equation}
To estimate the remaining term in \eqref{psieps1}, we first note that 
\begin{equation*}
\E\left[\left(\int_0^{ T} \exp\left\{p'\int_s^{ T}\left| \tfrac{\beta^{\mu, \varepsilon}(\sigma)}{\sqrt{\mu}}\right| \ud\sigma\right\} \ud s\right)^{(2p)/p'}\right]
\leqslant  C (p,T) \E \exp\left( C (p,T,\mu) \sup_{0\leqslant t\leqslant  T}\mathcal{E}(t)^{1/2}\right),
\end{equation*}
and
\begin{equation*}
 C (p,T) \sup_{0\leqslant t\leqslant  T}\mathcal{E}(t)^{1/2}
\leqslant  C (p,T,\mu)+ \eta_0 \sup_{0\leqslant t \leqslant T}\mathcal{E}(t).
\end{equation*}
Using the exponential bound \eqref{EXPene} yields
\begin{equation*}
\E\left[\left(\int_0^{ T} \exp\left\{p'\int_s^{ T}\left| \tfrac{\beta^{\mu, \varepsilon}(\sigma)}{\sqrt{\mu}}\right| \ud\sigma\right\} \ud s\right)^{(2p)/p'}\right]
\leqslant  C (p,T).
\end{equation*}
Summing up with \eqref{psieps1} and \eqref{alphaeps3} we obtain \eqref{eq:boundpsi}.
The estimates \eqref{uteps0} and \eqref{uteps01} are a consequence of \eqref{utOK}, \eqref{truncation gets eps} and the bounds
\begin{equation*}
\E\left[\left\|\zeta^\eps\right\|^p_{L^1([0, T] \times \T)}\right]\leqslant C (T,p,\mu)  \eps^{1/4},
\end{equation*}
and
\begin{equation*}
\E\left[\left\|\zeta^\eps\right\|^p_{L^\infty([0, T];L^1(\T))}\right]\leqslant C (T,p,\mu),
\end{equation*}
where $\zeta^{\mu, \varepsilon}$ is defined in \eqref{zeta}.
This follows easily from the bounds \eqref{eq:TotalEnergyep} and \eqref{eq:boundpsi}.
\end{proof}
%%%%%%%%%%%%%%%%%%%%%%%%%%%%%%%%%%%%%%%%%%%%%%%%%%%%%%%%%%%%%%%%%%%%%%%%

%%%%%%%%%%%%%%%%%%%%%%%%%%%%%%%%%%%%%%%%%%%%%%%%%%%%%%%%%%%%%%%%%%%%%%%%
\section{Uniform estimates}\label{sec:ue}
%%%%%%%%%%%%%%%%%%%%%%%%%%%%%%%%%%%%%%%%%%%%%%%%%%%%%%%%%%%%%%%%%%%%%%%%

As announced at the beginning of Section~\ref{sec:energy}, we will now exploit the gain due to the friction term (and exploit also the bounds already established), to derive some estimates that are uniform with respect to $\mu$. 

%%%%%%%%%%%%%%%%%%%%%%%%%%%%%%%%%%%%%%%%%%%%%%%%%%%%%%%%%%%%%%%%%%%%%%%%
\subsection{Parabolic estimates}\label{sec:ue-C0L2-L2H1}
%%%%%%%%%%%%%%%%%%%%%%%%%%%%%%%%%%%%%%%%%%%%%%%%%%%%%%%%%%%%%%%%%%%%%%%%
 
Here, by parabolic estimates, we mean estimates in the ``energy space'' 
	\begin{equation*}
		C([0,T];L^2(\T))\cap L^2([0,T];H^1(\T)),
	\end{equation*}
	which are standard for parabolic equations. A bound (independent of $\mu$) in such an energy space is quite expected, as the limit equation is a parabolic equation. To establish this bound, we need to take into account the non-linear sound speed $u\mapsto c(u)$, which is done via the following result.

%%%%%%%%%%%%%%%%%%%%%%%%%%%%%%%%%%%%%%%%%%%%%%%%%%%%%%%%%%%%%%%%%%%%%%%%
\begin{lem}\label{Pro:Hdef} 
		There exists a function $H$ and $C>0$ depending on the parameters $c_1$, $c_2$, $c_3$, $\gamma_1$, $\gamma_2$, $\kappa$, $\bar{u}$ in \eqref{coeff-c}-\eqref{c'g}  only, such that if  $h(u) = H'(u)/\gamma(u)$, we have 
\begin{gather*}
C^{-1} \leqslant h'(u) \leqslant C, \qquad
C^{-1} \leqslant (hc)'(u),\label{goodh}\\
H(u) \leqslant C (u^2 +1),\qquad
u^2 \leqslant C  (H(u) +1),\label{h and H ok}\\
h(u)^2 \leqslant C (u^2 +1), \qquad
u^2 \leqslant  C(h(u)^2 +1).\label{h and H ok 2}
\end{gather*}
\end{lem}
%%%%%%%%%%%%%%%%%%%%%%%%%%%%%%%%%%%%%%%%%%%%%%%%%%%%%%%%%%%%%%%%%%%%%%%%

%%%%%%%%%%%%%%%%%%%%%%%%%%%%%%%%%%%%%%%%%%%%%%%%%%%%%%%%%%%%%%%%%%%%%%%%
\begin{proof}[Proof of Lemma~\ref{Pro:Hdef}] If $c$ is a constant, then one can simply take $h(u)=u/c$ and $H(u) = \int_0^u v \gamma(v)/c\, \ud v$. In the general case, with regard to \eqref{c'g}, we define
\begin{equation*}
g(u) \eqdef \int_{\bar{u}}^u \frac{\ud v}{c(v)}.
\end{equation*}
We deduce then the existence of $\underline{u} \in (-\infty,\bar{u}]$, such that for any $u \leqslant \underline{u}$, we have $c'(u) g(u) \geqslant (\kappa-1)/2$. Now, we can define $h$ as
\begin{equation}\label{hdef}
h(u) \eqdef g(u) + \sup_{v \in [\underline{u},\bar{u}]} |g(v)|.
\end{equation}
Let then $u^\ast \eqdef h^{-1}(0)$. We can rewrite $h$ and define $H$ as follows:
\begin{equation*}
h(u) = \int_{u^\ast}^u \frac{\ud v}{c(v)}, \qquad H(u) \eqdef \int_{u^\ast}^u h(v) \gamma(v)\, \ud v.
\end{equation*}
Using the definition \eqref{hdef}, and considering the three cases $u \in(-\infty,\underline{u}]$, $u \in [\underline{u},\bar{u}]$ and $u\in [\bar{u},\infty)$, we can check that  
\begin{equation*}
0<\tfrac{1 +\kappa}{2} \leqslant (hc)'(u).
\end{equation*} 
The other inequalities follow in a straighforward way.
\end{proof}
%%%%%%%%%%%%%%%%%%%%%%%%%%%%%%%%%%%%%%%%%%%%%%%%%%%%%%%%%%%%%%%%%%%%%%%%

%%%%%%%%%%%%%%%%%%%%%%%%%%%%%%%%%%%%%%%%%%%%%%%%%%%%%%%%%%%%%%%%%%%%%%%%
\begin{pro}[Parabolic energy]\label{prop:friction E} 
Let $h$ and $H$ be defined in Lemma \ref{Pro:Hdef}. We have then
\begin{gather}\nonumber
 \ud \left[\mu \uu^{\mu, \varepsilon}_t h(u^{\mu, \varepsilon})+H(u^{\mu, \varepsilon})\right] + \left[ c(u^{\mu, \varepsilon}) h(u^{\mu, \varepsilon}) [\tfrac{R^{\mu, \varepsilon}-S^{\mu, \varepsilon}}{2}] \right]_x \ud t   + (ch)'(u^{\mu, \varepsilon}) c(u^{\mu, \varepsilon})\left(u^{\mu, \varepsilon}_x\right)^2\, \ud t \\
  \label{eqh}
 = \mu h'(u^{\mu, \varepsilon}) \left(\uu_t^{\mu, \varepsilon}\right)^2  \ud t  + \eta_h^{\mu, \varepsilon}\, \ud t \
+ h(u^{\mu, \varepsilon}) f(u^{\mu, \varepsilon})\, \ud t  + h(u^{\mu, \varepsilon}) \Phi^{\varepsilon}   \ud W,
\end{gather}
where the remainder term $\eta_h^{\mu, \varepsilon}$ satisfies the estimate
\begin{equation}\label{Error}
\E \left| \int_0^T \int_\T \left|\eta_h^{\mu, \varepsilon}\right| \ud x\, \ud t \right|^p \leqslant C(p,T,\mu) \varepsilon^{\frac{1}{8}},
\end{equation}
for any $p \in [1,\infty)$.
\end{pro}
%%%%%%%%%%%%%%%%%%%%%%%%%%%%%%%%%%%%%%%%%%%%%%%%%%%%%%%%%%%%%%%%%%%%%%%%

%%%%%%%%%%%%%%%%%%%%%%%%%%%%%%%%%%%%%%%%%%%%%%%%%%%%%%%%%%%%%%%%%%%%%%%%
\begin{proof}[Proof of Proposition~\ref{prop:friction E}] Adding \eqref{Reqep} to \eqref{Seqep} gives the following evolution equation for $\uu_t^{\mu, \varepsilon}$:
\begin{gather*}
\mu\, \ud \uu_t^{\mu, \varepsilon} + c(u^{\mu, \varepsilon}) [\tfrac{R^{\mu, \varepsilon}-S^{\mu, \varepsilon}}{2}]_x\, \ud t + \gamma(u^{\mu, \varepsilon}) \uu^{\mu, \varepsilon}_t\, \ud t\\
= - \half \tilde{c}'(u^{\mu, \varepsilon}) \left[   \chi_\varepsilon(R^{\mu, \varepsilon}) +\chi_\varepsilon(S^{\mu, \varepsilon})\right] \ud t + \tilde{c}'(u^{\mu, \varepsilon}) \Theta^{\mu, \varepsilon} \left[  R^{\mu, \varepsilon}-S^{\mu, \varepsilon}   \right] \ud t + f(\uu^{\mu, \varepsilon})\, \ud t +  \Phi^{\varepsilon}   \ud W.
\end{gather*}
Multiplying by $h(u^{\mu, \varepsilon})$, we obtain
\begin{gather*}
\mu\, \ud \left[\uu^{\mu, \varepsilon}_t h(u^{\mu, \varepsilon})\right] -\mu h'(u^{\mu, \varepsilon}) \uu_t^{\mu, \varepsilon}  u_t^{\mu, \varepsilon}\, \ud t  +  \left[ c(u^{\mu, \varepsilon}) h(u^{\mu, \varepsilon}) [\tfrac{R^{\mu, \varepsilon}-S^{\mu, \varepsilon}}{2}] \right]_x \ud t \\
+ (ch)'(u^{\mu, \varepsilon}) u^{\mu, \varepsilon}_x \tfrac{S^{\mu, \varepsilon}-R^{\mu, \varepsilon}}{2}\, \ud t +h(u^{\mu, \varepsilon}) \gamma(u^{\mu, \varepsilon}) \uu^{\mu, \varepsilon}_t \ud t = - \half h(u^{\mu, \varepsilon}) \tilde{c}'(u^{\mu, \varepsilon}) \left[   \chi_\varepsilon(R^{\mu, \varepsilon}) +\chi_\varepsilon(S^{\mu, \varepsilon})\right] \ud t\\
 + h(u^{\mu, \varepsilon}) \tilde{c}'(u^{\mu, \varepsilon}) \Theta^{\mu, \varepsilon} \left[  R^{\mu, \varepsilon}-S^{\mu, \varepsilon}   \right] \ud t + h(u^{\mu, \varepsilon}) f(\uu^{\mu, \varepsilon})\, \ud t + h(u^{\mu, \varepsilon}) \Phi^{\varepsilon}   \ud W.
\end{gather*}
We express $R^{\mu, \varepsilon}-S^{\mu, \varepsilon}$ in function of $u^{\mu, \varepsilon}_x$ and $\uu^{\mu, \varepsilon}_t$ in function of $u^{\mu, \varepsilon}_t$ via the identities  \eqref{cuxthetaeps} and \eqref{xiut}. This introduces some corrector terms with factor $\Theta^{\mu, \varepsilon}$ and $\Xi^{\mu, \varepsilon}$ and gives \eqref{eqh} with
\begin{multline*}
	\eta_h^{\mu, \varepsilon}
	=
	\sqrt{\mu} h'(u^{\mu, \varepsilon}) \uu_t^{\mu, \varepsilon}  \Xi^{\mu, \varepsilon}\,
	+(ch)'(u^{\mu, \varepsilon}) u^{\mu, \varepsilon}_x \Theta^{\mu, \varepsilon}\,
	\\
	+\mu^{-1/2} h(u^{\mu, \varepsilon}) \gamma(u^{\mu, \varepsilon}) \Xi^{\mu, \varepsilon}
	 -\half h(u^{\mu, \varepsilon}) \tilde{c}'(u^{\mu, \varepsilon}) \left[   \chi_\varepsilon(R^{\mu, \varepsilon}) +\chi_\varepsilon(S^{\mu, \varepsilon})\right]
	 \\
	   + h(u^{\mu, \varepsilon}) \tilde{c}'(u^{\mu, \varepsilon}) \Theta^{\mu, \varepsilon} \left[  R^{\mu, \varepsilon}-S^{\mu, \varepsilon}   \right] 
	   + h(u^{\mu, \varepsilon}) \left[f(\uu^{\mu, \varepsilon})-f(u^{\mu, \varepsilon})\right].
\end{multline*}
To bound the remainder term $\eta_h^{\mu, \varepsilon}$, we proceed as follows. First, by \eqref{cuxthetaeps} and  \eqref{xiut} we have the following control on the derivatives of $u^{\mu, \varepsilon}$:
\begin{equation*}
c(u^{\mu, \varepsilon})^2 \left(u^{\mu, \varepsilon}_x\right)^2+ \mu \left(u^{\mu, \varepsilon}_t\right)^2 \leqslant C \left[ \left(R^{\mu, \varepsilon} \right)^2+\left(S^{\mu, \varepsilon} \right)^2 + \left(\Theta^{\mu, \varepsilon} \right)^2 + \left(\Xi^{\mu, \varepsilon} \right)^2 \right].
\end{equation*}
Next, using
\begin{equation*}
\|u^{\mu,\varepsilon}(t)\|_{L^2(\T)} \leqslant \|u_0^{\varepsilon}\|_{L^2(\T)} + \int_0^T \|u^{\mu,\varepsilon}_t(s)\|_{L^2(\T)}\, \ud s
\end{equation*}
in addition to the estimate \eqref{Frictional energy} on the frictional energy and the bounds \eqref{eq:boundpsi} and \eqref{uteps01} on the corrective terms $\Theta^{\mu, \varepsilon}$ and $\Xi^{\mu, \varepsilon}$, we obtain
\begin{equation}\label{uinfty}
\E \|u^{\mu, \varepsilon}\|^p_{L^\infty([0,T) \times \T)} \leqslant \E \|u^{\mu, \varepsilon}\|^p_{L^\infty([0,T); H^1(\T))} \leqslant C(p,T,\mu).
\end{equation}
The estimate \eqref{Error} then follows from \eqref{uinfty}, from the estimate on the wave energy \eqref{Wave energy ep}, from \eqref{chiep} and the estimate on the truncation terms \eqref{Truncation Energy}, and from the bounds \eqref{eq:boundpsi} and \eqref{uteps0} on the corrective terms $\Theta^{\mu, \varepsilon}$ and $\Xi^{\mu, \varepsilon}$.
\end{proof}
%%%%%%%%%%%%%%%%%%%%%%%%%%%%%%%%%%%%%%%%%%%%%%%%%%%%%%%%%%%%%%%%%%%%%%%%

%%%%%%%%%%%%%%%%%%%%%%%%%%%%%%%%%%%%%%%%%%%%%%%%%%%%%%%%%%%%%%%%%%%%%%%%
\begin{pro}[Parabolic estimate]\label{Pro:uniform_mu}
For a fixed $T > 0$ and $p \in [1, \infty)$, there exists a constant $C > 0$ such that, for any $\mu \in (0, 1)$, there exists $\varepsilon_\mu \in (0, 1)$ satisfying the following: for any $\varepsilon \in (0, \varepsilon_\mu)$, we have
\begin{gather}\label{eq:by-product friction}
\E \left[ \sup_{t \in [0,T]} \| \uu^{\mu, \varepsilon}\|_{L^2(\T)}^2  + \sup_{t \in [0,T]} \| u^{\mu, \varepsilon}\|_{L^2(\T)}^2+ \int_0^T  \| u_x^{\mu, \varepsilon}\|_{L^2(\T)}^2 \, \ud t  \right]^p \leqslant C.
\end{gather}
\end{pro}
%%%%%%%%%%%%%%%%%%%%%%%%%%%%%%%%%%%%%%%%%%%%%%%%%%%%%%%%%%%%%%%%%%%%%%%%

%%%%%%%%%%%%%%%%%%%%%%%%%%%%%%%%%%%%%%%%%%%%%%%%%%%%%%%%%%%%%%%%%%%%%%%%
\begin{proof}[Proof of Proposition~\ref{Pro:uniform_mu}] Using the periodicity condition,  let us integrate \eqref{eqh} over $(0,t)\times\T$. We use the growth properties of $h$ and $H$ and the bound from below \eqref{coeff-gammaf} on $c$ to obtain
		\begin{multline*}
			\tfrac{1}{C}\|u^{\mu, \varepsilon}(t)\|_{L^2(\T)}^2  
			 +\tfrac{1}{C}\int_0^t  \| u_x^{\mu, \varepsilon}\|_{L^2(\T)}^2 \, \ud s 
			 \\
			 \leqslant C+ \mu    \int_\T |\uu^{\mu, \varepsilon}_t(t)| | h(u^{\mu, \varepsilon}(t))| \, \ud x +
			 C\int_0^t  \mu\| \uu_t^{\mu, \varepsilon}\|_{L^2(\T)}^2\, \ud s 
			 +\int_0^T\int_\T \left|\eta_h^{\mu, \varepsilon}\right| \ud x\,\ud s \\ 
			+C\int_0^t\int_\T |h(u^{\mu, \varepsilon})| |f(u^{\mu, \varepsilon})|\, \ud x\, \ud s
			+\mathcal{M}_h^{\mu, \varepsilon}(t) ,
		\end{multline*}
where $C$ is a constant depending on the parameters $c_1$, $c_2$, $c_3$, $\gamma_1$, $\gamma_2$, $\kappa$, $\bar{u}$ in \eqref{coeff-c}-\eqref{c'g} only and $\mathcal{M}_h^{\mu, \varepsilon}(t)$ is the martingale
		\begin{equation*}
			\mathcal{M}_h^{\mu, \varepsilon}(t) \eqdef   \int_0^t \int_\T \sum_{k \geqslant 1 }  h(u^{\mu, \varepsilon}(s,x))  \sigma_k^{\mu, \varepsilon}(x)\,  \ud x\,  \ud \beta_k(s).
		\end{equation*}
		Next, we use the bound
		\begin{equation*}
			|h(u^{\mu, \varepsilon})| |f(u^{\mu, \varepsilon})|
			\leqslant C \left(1+|u^{\mu, \varepsilon}| \right) \left(|f(0)|+L|u^{\mu, \varepsilon}| \right)
			\leqslant C\left(1+|u^{\mu, \varepsilon}|^2 \right),
		\end{equation*}
		with 
		\begin{equation*}
		\mu  \int_\T |\uu^{\mu, \varepsilon}_t(t)| | h(u^{\mu, \varepsilon}(t))| \, \ud x  	\leqslant C \mu^2    \| \uu^{\mu, \varepsilon}_t(t) \|^2_{L^2(\T)}   + \tfrac{1}{2C}   \|  u^{\mu, \varepsilon}(t) \|^2_{L^2(\T)} + C,
		\end{equation*}
		to get
			\begin{multline*}
			\tfrac{1}{2C}\|u^{\mu, \varepsilon}(t)\|_{L^2(\T)}^2
			+\tfrac{1}{C}\int_0^t  \| u_x^{\mu, \varepsilon}\|_{L^2(\T)}^2 \, \ud s
			\\
			\leqslant
			C(T)+ C \mu^2  \sup_{t \in [0,T]}  \| \uu^{\mu, \varepsilon}_t(t) \|^2_{L^2(\T)} +
			C\int_0^t  \mu\| \uu_t^{\mu, \varepsilon}\|_{L^2(\T)}^2\, \ud s\\
			+\int_0^T\int_\T \left|\eta_h^{\mu, \varepsilon}\right| \ud x\,\ud t 
			+\int_0^t \|u^{\mu, \varepsilon}(t)\|_{L^2(\T)}^2\, \ud s
			+\mathcal{M}_h^{\mu, \varepsilon}(t) .
		\end{multline*}
By the Gr\"onwall lemma with the estimate \eqref{Wave energy ep}, \eqref{Frictional energy} and \eqref{Error}, we obtain
		\begin{multline}\label{by-product friction 1}
			\E \left[ \sup_{t \in [0,T]} \| u^{\mu, \varepsilon}\|_{L^2(\T)}^2+ \int_0^T  \| u_x^{\mu, \varepsilon}\|_{L^2(\T)}^2 \, \ud t  \right]^p 
			\\
			\leqslant
			C(T,p)
			+C(p,T,\mu) \varepsilon^{\frac{1}{8}}
			+C(T,p)\E \left[
							\sup_{t \in [0,T]}\left|\mathcal{M}_h^{\mu, \varepsilon}(t)\right|^p
			\right].
		\end{multline}
		By the Burkholder--Davis--Gundy inequality, we can estimate the sup of the martingale as follows
		\begin{gather}\nonumber
			\E\left[ \sup_{t \in [0,T]} \left|\mathcal{M}_h^{\mu, \varepsilon}(t)\right|^p\right] \leqslant C(T,p) \E \left(\int_0^T \int_\T |h(u^{\mu, \varepsilon}(t,x))|^2\, \ud x\, \ud t\right)^{p/2}\\
			\nonumber
			\leqslant C(T,p)+C(T,p) \E\left[\sup_{t \in [0,T]} \| u^{\mu, \varepsilon}\|_{L^2(\T)}^{p}\right],
		\end{gather}
		and infer from \eqref{by-product friction 1} the estimate
		\begin{equation*}
			\E \left[ \sup_{t \in [0,T]} \| u^{\mu, \varepsilon}\|_{L^2(\T)}^2+ \int_0^T  \| u_x^{\mu, \varepsilon}\|_{L^2(\T)}^2 \, \ud t  \right]^p 
			\leqslant
			C(T,p)
			+C(p,T,\mu) \varepsilon^{\frac{1}{8}}.
		\end{equation*}
	Therefore, using \eqref{uteps0}, we obtain 	
			\begin{equation*}
			\E \left[\sup_{t \in [0,T]} \| \uu^{\mu, \varepsilon}\|_{L^2(\T)}^2 + \sup_{t \in [0,T]} \| u^{\mu, \varepsilon}\|_{L^2(\T)}^2+ \int_0^T  \| u_x^{\mu, \varepsilon}\|_{L^2(\T)}^2 \, \ud t  \right]^p 
			\leqslant
			C(T,p)
			+C(p,T,\mu) \varepsilon^{\frac{1}{8}},
		\end{equation*}	
		and the estimate \eqref{eq:by-product friction} follows for $\eps<\eps_\mu$ small enough.
\end{proof}
%%%%%%%%%%%%%%%%%%%%%%%%%%%%%%%%%%%%%%%%%%%%%%%%%%%%%%%%%%%%%%%%%%%%%%%%

%%%%%%%%%%%%%%%%%%%%%%%%%%%%%%%%%%%%%%%%%%%%%%%%%%%%%%%%%%%%%%%%%%%%%%%%
	\begin{corollary}[Bound on the integral of the wave energy]\label{cor:integral E}
		For a fixed $T > 0$ and $p \in [1, \infty)$, there exists a constant $C > 0$ such that, for any $\mu \in (0, 1)$, there exists $\varepsilon_\mu \in (0, 1)$ satisfying the following: for any $\varepsilon \in (0, \varepsilon_\mu)$, we have
		\begin{gather}\label{eq:integral E}
			\E \left[ \int_0^T \mathcal{E}^{\mu, \varepsilon}(t)\, \ud t  \right]^p \leqslant C.
		\end{gather}
	\end{corollary}
%%%%%%%%%%%%%%%%%%%%%%%%%%%%%%%%%%%%%%%%%%%%%%%%%%%%%%%%%%%%%%%%%%%%%%%%
	
%%%%%%%%%%%%%%%%%%%%%%%%%%%%%%%%%%%%%%%%%%%%%%%%%%%%%%%%%%%%%%%%%%%%%%%%
		\begin{proof}[Proof of Corollary~\ref{cor:integral E}] We use the equations (deduced from \eqref{cuxthetaeps} and \eqref{wdefep})
			\begin{gather*}
				S^{\mu, \varepsilon}=\sqrt{\mu}\uu^{\mu, \varepsilon}_t+c(u^{\mu, \varepsilon}) u^{\mu, \varepsilon}_x+\Theta^{\mu, \varepsilon}\, ,\label{S by ux and uut}\\
				R^{\mu, \varepsilon}=\sqrt{\mu}\uu^{\mu, \varepsilon}_t-c(u^{\mu, \varepsilon}) u^{\mu, \varepsilon}_x-\Theta^{\mu, \varepsilon}\,\label{R by ux and uut},
			\end{gather*}
			to get 
			\begin{equation*}
							\mathcal{E}^{\mu, \varepsilon}\leqslant C\left(\mu\|\uu^{\mu, \varepsilon}_t\|_{L^2(\T)}^2+\|u^{\mu, \varepsilon}_x\|_{L^2(\T)}^2+\left|\Theta^{\mu, \varepsilon}\right|^2 \right). 
			\end{equation*}
The estimate \eqref{eq:integral E} then follows from the bounds \eqref{Frictional energy} (frictional energy), \eqref{eq:boundpsi} (corrective term) and \eqref{eq:by-product friction} (parabolic estimate).
		\end{proof}
%%%%%%%%%%%%%%%%%%%%%%%%%%%%%%%%%%%%%%%%%%%%%%%%%%%%%%%%%%%%%%%%%%%%%%%%

%%%%%%%%%%%%%%%%%%%%%%%%%%%%%%%%%%%%%%%%%%%%%%%%%%%%%%%%%%%%%%%%%%%%%%%%
\subsection{Improved estimate on the wave energy}\label{sec:better ernergy estimate}
%%%%%%%%%%%%%%%%%%%%%%%%%%%%%%%%%%%%%%%%%%%%%%%%%%%%%%%%%%%%%%%%%%%%%%%%

We have established in Section~\ref{sec:energy}, via an energy estimate, the bound
	\begin{equation}\label{Eestimate bad}
		\E\left[ \sup_{t \in [0,T]} \mathcal{E}^{\mu, \varepsilon}(t)\right]^p \leqslant \tfrac{C}{\mu^p}.
	\end{equation}
The analysis of the contribution of the friction term leads to the additional uniform bound \eqref{eq:integral E} on $\mathcal{E}^{\mu, \varepsilon}$.
We will see that we can exploit \eqref{eq:integral E} to improve \eqref{Eestimate bad} into the estimate 
	\begin{equation}\label{Eestimate good}
		\E \left[  \sup_{t \in [0,T]} \mathcal{E}^{\mu, \varepsilon}(t)\right]^p \leqslant  \tfrac{C}{\mu^{p/2}}.  
	\end{equation}
As we will see, in the course proof of \eqref{Eestimate good}, the friction term will once again have a positive contribution, that will yield the estimate 
\begin{equation}\label{eq:TotalEnergyep good}
	 \E \int_0^T \left[ \mu^2 \|\uu_t^{\mu, \varepsilon}\|^4_{L^2(\T)}\right]^p \ud t \leqslant C.  
\end{equation}

%%%%%%%%%%%%%%%%%%%%%%%%%%%%%%%%%%%%%%%%%%%%%%%%%%%%%%%%%%%%%%%%%%%%%%%%
\begin{pro}[Improved energy estimate]\label{prop:energy estimate good}
For fixed $T > 0$ and $p \in [1, \infty)$, 
 there exists a constant $C > 0$ such that, for any $\mu \in (0, 1)$, there exists $\varepsilon_\mu \in (0, 1)$ satisfying the following: for any $\varepsilon \in (0, \varepsilon_\mu)$ the estimates \eqref{Eestimate good} and \eqref{eq:TotalEnergyep good} are satisfied.
\end{pro}
%%%%%%%%%%%%%%%%%%%%%%%%%%%%%%%%%%%%%%%%%%%%%%%%%%%%%%%%%%%%%%%%%%%%%%%%

%%%%%%%%%%%%%%%%%%%%%%%%%%%%%%%%%%%%%%%%%%%%%%%%%%%%%%%%%%%%%%%%%%%%%%%%
\begin{proof}[Proof of Proposition~\ref{prop:energy estimate good}] In a deterministic framework, the basis of the proof is the expansion
		\begin{equation}\label{expand E2}
			E(t)^2=2\int_0^t E(s)E'(s)\, \ud s,
		\end{equation}
		which leads, for $E\geqslant 0$, to the estimate
		\begin{equation*}
			\sup_{t\in[0,T]}E(t)^2\leqslant 2\sup_{t\in[0,T]}(E'(t))^+\int_0^T E(t)\, \ud t.
		\end{equation*}
		In our stochastic context, we start from the equation \eqref{EnergyBalance} on $\mathcal{E}^{\mu, \varepsilon}$ and use It\^o's formula to get the counterpart of \eqref{expand E2}. Setting
		\begin{equation*}
			E^{\mu, \varepsilon}=\sqrt{\mu}\, \mathcal{E}^{\mu, \varepsilon},
		\end{equation*}
		we have
		\begin{gather}\nonumber
			 (E^{\mu, \varepsilon})^2(t)+ 8 \gamma_1 \sqrt{\mu} \int_0^t \left\| \uu_t^{\mu, \varepsilon} \right\|_{L^2(\T)}^2 E^{\mu, \varepsilon}\, \ud s  
			\leqslant
			(E^{\mu, \varepsilon})^2(0)+ 4 q_0 \int_0^t \mathcal{E}^{\mu, \varepsilon}\, \ud s\\ \label{after Ito Energy}
			  + 16 \mu q_0 \int_0^t \left( \int_\T (\uu_t^{\mu, \varepsilon})^2\, \ud x\right) \ud s
			+ 8 \sqrt{\mu} \int_0^t \left(\int_\T \uu^{\mu, \varepsilon}_t f(\uu^{\mu, \varepsilon})\, \ud x\right) E^{\mu, \varepsilon}\, \ud s+\mathcal{N}^{\mu, \varepsilon}(t), 
		\end{gather}
		where
		\begin{equation*}
			\mathcal{N}^{\mu, \varepsilon}(t)  \eqdef 8  \int_0^t E^{\mu, \varepsilon} \sum_{k \geqslant 1} \left(\int_\T \sqrt{\mu}\uu^{\mu, \varepsilon}_t \sigma_k^\varepsilon\, \ud x \right) \ud \beta_k(t).
		\end{equation*}
		We deduce from \eqref{after Ito Energy} that 
		\begin{multline*}
			 (E^{\mu, \varepsilon})^2(t)+ 4 \gamma_1 \sqrt{\mu} \int_0^t \left\| \uu_t^{\mu, \varepsilon} \right\|_{L^2(\T)}^2 E^{\mu, \varepsilon}\, \ud s  
			\leqslant C+ C \int_0^t \mathcal{E}^{\mu, \varepsilon}\, \ud s  + C \mu \int_0^t \|\uu^{\mu, \varepsilon}_t\|_{L^2(\T)}^2\, \ud s
			\\ 
			+ C \sqrt{\mu} \int_0^t \left(1+\|\uu^{\mu, \varepsilon}\|_{L^2(\T)}^2\right) E^{\mu, \varepsilon}\, \ud s + 	\mathcal{N}^{\mu, \varepsilon}(t).
		\end{multline*}
		Since $E^{\mu, \varepsilon} \geqslant 2 \mu^{3/2} \| \uu_t^{\mu, \varepsilon} \|_{L^2(\T)}^2$, we get
		\begin{gather}
			\E\left[\sup_{t\in[0,T]}(E^{\mu, \varepsilon})^2\right]^p+ \E\left[ \int_0^T\mu^2 \left\| \uu_t^{\mu, \varepsilon} \right\|_{L^2(\T)}^4\,\ud t \right]^p
			  \nonumber\\
			\leqslant C+ C \E \left[ \int_0^T \mathcal{E}^{\mu, \varepsilon}\, \ud t \right]^p  + C\E\left[\int_0^T \mu \|\uu^{\mu, \varepsilon}_t\|_{L^2(\T)}^2\, \ud t\right]^p 
			\nonumber\\
			+ C\E\left[\int_0^T \sqrt{\mu} \left(1+\|\uu^{\mu, \varepsilon}\|_{L^2(\T)}^2\right) E^{\mu, \varepsilon}\, \ud t\right]^p  
			+ C \E\left[\sup_{t\in[0,T]}\left|\mathcal{N}^{\mu, \varepsilon}(t)\right|\right]^p.\label{after Ito Energy 3}
		\end{gather}
		We use the estimate
		\begin{equation*}
		\sqrt{\mu} \left(1+\|\uu^{\mu, \varepsilon}\|_{L^2(\T)}^2\right) E^{\mu, \varepsilon}
			\leqslant
		\left(1+\|\uu^{\mu, \varepsilon}\|_{L^2(\T)}^4\right) 
			+(\mu \mathcal{E}^{\mu, \varepsilon})^2 ,
		\end{equation*}
		together with the wave energy estimate \eqref{Wave energy ep} and the parabolic bound \eqref{eq:by-product friction} to get
		\begin{equation*}
			\E \left[ \int_0^T \sqrt{\mu} \left(1+\|\uu^{\mu, \varepsilon}\|_{L^2(\T)}^2\right) E^{\mu, \varepsilon}\, \ud t\right]^p \leqslant C.
		\end{equation*}
		We also use the crucial bound on the integrated energy \eqref{eq:integral E} and the estimate on $\mu \|\uu^{\mu, \varepsilon}_t\|_{L^2(\T)}^2$ given by \eqref{Frictional energy} to deduce from \eqref{after Ito Energy 3} the inequality
		\begin{equation}\label{after Ito Energy 4}
			\E\left[\sup_{t\in[0,T]}(E^{\mu, \varepsilon})^2\right]^p+ \E \left[ \int_0^T\mu^2 \left\| \uu_t^{\mu, \varepsilon} \right\|_{L^2(\T)}^4 \ud t \right]^p   
			\leqslant
			C+C\E\left[\sup_{t\in[0,T]}\left|\mathcal{N}^{\mu, \varepsilon}(t)\right|\right]^p.
		\end{equation}
		By the Burkholder--Davis--Gundy inequality, we can bound the martingale term in \eqref{after Ito Energy 4} as
		\begin{equation*}
			\E\left[\sup_{t\in[0,T]}\left|\mathcal{N}^{\mu, \varepsilon}(t)\right|\right]^p
			\leqslant C \E\left[\int_0^T\left|E^{\mu, \varepsilon}\right|^2 \mu\left\| \uu_t^{\mu, \varepsilon} \right\|_{L^2(\T)}^2 \ud t \right]^{p/2}.
		\end{equation*}
		Using again the estimate \eqref{Frictional energy} on the frictional energy, we get
		\begin{equation*}
			\E\left[\sup_{t\in[0,T]}\left|\mathcal{N}^{\mu, \varepsilon}(t)\right|\right]^p
			\leqslant C+\tfrac12\E\left[\sup_{t\in[0,T]}(E^{\mu, \varepsilon})^2\right]^p.
		\end{equation*}
		Inserting this in \eqref{after Ito Energy 4} gives the result.
\end{proof}
%%%%%%%%%%%%%%%%%%%%%%%%%%%%%%%%%%%%%%%%%%%%%%%%%%%%%%%%%%%%%%%%%%%%%%%%

%%%%%%%%%%%%%%%%%%%%%%%%%%%%%%%%%%%%%%%%%%%%%%%%%%%%%%%%%%%%%%%%%%%%%%%%
\subsection{Existence of martingale solutions with good controls}\label{sec:proof Th martingales}
%%%%%%%%%%%%%%%%%%%%%%%%%%%%%%%%%%%%%%%%%%%%%%%%%%%%%%%%%%%%%%%%%%%%%%%%

In this section, we give the missing pieces to establish Theorem \ref{thm:global-existR2SS}.\medskip

%%%%%%%%%%%%%%%%%%%%%%%%%%%%%%%%%%%%%%%%%%%%%%%%%%%%%%%%%%%%%%%%%%%%%%%
\textbf{Step 1.} Following \cite{GV25} (possibly with the slight generalization given in the proof of Theorem~\ref{thm:global-existR2SS}, see \eqref{collective mu}), where we use the stochastic compactness method, via the Prokhorov theorem and the Skorokhod--Jakubowski theorem, we are in the following situation: there exists a stochastic basis \eqref{StochasticBasis},
where $\left(\tilde{W}(t)\right)$ is a cylindrical Wiener process on $\mathfrak{U}$, and there exists some sequences $(\tilde{u}^{\mu, \varepsilon_k})_k$, $(\tilde{\uu}^{\mu, \varepsilon_k})_k$ such that:
\begin{itemize}
	\item $(\tilde{u}^{\mu, \varepsilon_k},\tilde{\uu}^{\mu, \varepsilon_k})$ has the same law as $(u^{\mu, \varepsilon_k},\uu^{\mu, \varepsilon_k})$,
	\item there exists a random variable $u^\mu$ which is a weak martingale solutions of \eqref{SVWE1} in the sense of Definition \ref{def:WeakSol} and such that, $\tilde{\Pro}$-almost surely, 
	\begin{equation*}
		\lim_{k \to \infty} \left\{ \|\tilde{u}^{\mu, \varepsilon_k} - u^\mu\|_{C([0,T] \times \T)}+ \|\tilde{\uu}^{\mu, \varepsilon_k} - u^\mu\|_{C([0,T] \times \T)}  \right\} = 0,
	\end{equation*}
	and
		\begin{equation*}
		\lim_{k \to \infty} \left\{ \|\tilde{u}_x^{\mu, \varepsilon_k} - u_x^\mu\|_{L^p([0,T] \times \T)} + \|\tilde{\uu}_t^{\mu, \varepsilon_k} - u_t^\mu\|_{L^p([0,T] \times \T)} \right\} = 0,
	\end{equation*}
	for any $p<2$.
\end{itemize}
The estimates \eqref{L3estimates} and \eqref{Oleinik} are then proved as in \cite{GV25}, while \eqref{main_estimates} follows from \eqref{eq:by-product friction}, \eqref{Eestimate good} and \eqref{eq:TotalEnergyep good}. \medskip

\textbf{Step 2.} There remains to prove the one-sided estimate \eqref{ene_mu0}. 
Using the It\^o formula, the energy equation \eqref{sumsquare-theta} and the inequality $\mu \uu_t f(\uu) \leqslant C (\mu^{3/2} \uu_t^2 + \mu^{1/2}f(\uu)^2)$, we have for any non-negative $\psi \in C^2_c((0,T)\times \T)$, and for $\eps<\eps_\mu$,
\begin{gather} \nonumber
\int_0^T \int_\T \left( 4 \mu \gamma(\tilde{u}^{\mu, \varepsilon}) (\tilde{\uu}^{\mu, \varepsilon}_t)^2 - 2 q^\varepsilon \right) \psi\, \ud x\,  \ud t \\ \nonumber
		 \leqslant C \sqrt{\mu} \left[  \int_0^T \left( \tilde{\mathcal{E}}^{\mu, \varepsilon} + \left\|\tilde{\uu}^{\mu, \varepsilon}\right\|_{L^2(\T)}^2\right)   \ud \sigma +  \sup_{t \in [0,T]} |Y^{\mu, \varepsilon}| +1 \right] ,
\end{gather}
where $Y^{\mu, \varepsilon}$ is the martingale
\begin{equation*}
			Y^{\mu, \varepsilon}(t) \eqdef \sqrt{\mu}  \sum_{k\geqslant 1} \int_0^{t}   \int_\T \tilde{\uu}_t^{\mu, \varepsilon}  \sigma_k^\varepsilon(x) \psi\, \ud x\, \ud \tilde{\beta}_k(s).
\end{equation*}
Doob's martingale inequality and the bound \eqref{Frictional energy} on the frictional energy imply 
\begin{equation*}
\E\left[\sup_{t \in [0,T]} |Y^{\mu, \varepsilon}(t)|^2\right] \leqslant C.
\end{equation*}
By the parabolic estimate \eqref{eq:by-product friction} and \eqref{eq:integral E}, we have therefore
\begin{equation*}
	\E \left( \int_0^T \int_\T \left( 2 \mu \gamma(\tilde{u}^{\mu, \varepsilon}) (\tilde{\uu}^{\mu, \varepsilon}_t)^2 - q^\varepsilon \right) \psi \, \ud x\,  \ud t \right)^+ \leqslant C\sqrt{\mu}.
\end{equation*}
By Fatou's lemma and the inequality $(\liminf a_n)^+\leqslant \liminf a_n^+$, we obtain
\begin{equation*}
0 \leqslant	\E \left( \int_0^T \int_\T \left( 2 \mu \gamma(\tilde{u}^{\mu}) (\tilde{u}^{\mu}_t)^2 - q \right) \psi \, \ud x\,  \ud t \right)^+ \leqslant C\sqrt{\mu},
\end{equation*}
which gives \eqref{ene_mu0}.
%%%%%%%%%%%%%%%%%%%%%%%%%%%%%%%%%%%%%%%%%%%%%%%%%%%%%%%%%%%%%%%%%%%%%%%%

%%%%%%%%%%%%%%%%%%%%%%%%%%%%%%%%%%%%%%%%%%%%%%%%%%%%%%
\section{The small-mass limit: the Smoluchowski--Kramers approximation}\label{sec:limit}
%%%%%%%%%%%%%%%%%%%%%%%%%%%%%%%%%%%%%%%%%%%%%%%%%%%%%%%%%%%%%%%%%%%%%%%%
 
	In this section, we will establish the Smoluchowski--Kramers approximation, as described in Theorem \ref{thm:SK}. 
%	To that purpose, we consider the probability space \eqref{ProbaSpace} and a sequence of solutions $u^\mu$ to \eqref{SVWE1} satisfying the conditions of Theorem \ref{thm:SK}.\medskip

%%%%%%%%%%%%%%%%%%%%%%%%%%%%%%%%%%%%%%%%%%%%%%%%%%%%%%%%%%%%%%%%%%%%%%%%
\subsection{Compactness results}
%%%%%%%%%%%%%%%%%%%%%%%%%%%%%%%%%%%%%%%%%%%%%%%%%%%%%%%%%%%%%%%%%%%%%%%%

	Let us consider the following set of unknowns and auxiliary functions, where $\alpha \in (1/2,1)$ is a fixed parameter:
\begin{gather*}
\ww^\mu \eqdef  \Gamma(u^\mu) \eqdef \int_0^{u^\mu} \gamma(v)\, \ud v ,
 \qquad \qquad \imp^\mu \eqdef  \underline{\Gamma}(u^\mu) \eqdef \int_0^{u^\mu} \tfrac{\gamma(v)}{c(v)}\, \ud v , \\
r^\mu \eqdef \mu \gamma(u^\mu) (u^\mu_t)^2,  \qquad \mathscr{S}^\mu \eqdef \|r^\mu\|_{L^2([0,T];H^{-\alpha}(\T))}+1, 
\qquad \mathscr{R}^\mu \eqdef \tfrac{r^\mu}{\mathscr{S}^\mu},
\qquad \mathcal{V}^\mu \eqdef \mu u_t^\mu,\\
\mathcal{C}_2(u) \eqdef \int_0^u c(v)^2\, \ud v, \qquad k(u) \eqdef \int_0^u \sqrt{c'(v)c(v)}\, \ud v, \qquad \aaa^\mu \eqdef \left(k(u^\mu)_x\right)^2= c'(u^\mu)c(u^\mu) (u^\mu_x)^2.
\end{gather*}
Let $\mathcal{X} $ denote the Fr\'echet space
\begin{equation*}
	\mathcal{X} \eqdef \cap_{n \geqslant 1} \left( C([0,T]; H^{-\frac1n}(\T)) \cap L^n([0,T]; L^2(\T))   \right).
\end{equation*}
For a fixed $\alpha \in (1/2,1)$, let $\BBB$ denote the unit ball of the space $L^2([0,T];H^{-\alpha}(\T))$ equipped with the weak topology. We will consider the sequence
\begin{equation*}
	Z^\mu\eqdef \left(\imp^\mu, \mathcal{V}^\mu, \mathscr{R}^\mu, \mathscr{S}^\mu, \aaa^\mu, W^\mu \right)_{\mu\in\Lambda}
\end{equation*}
in the space 
\begin{equation}\label{global space for tightness}
	\mathscr{Z}\eqdef\mathcal{X} \times C([0,T];L^2(\T)) \times \BBB \times \R \times H^{-2}((0,T) \times \T) \times C([0,T]; \mathfrak{U}_{-1}).
\end{equation}
The space $H^{-2}((0,T) \times \T) $ is the dual to $H^2_0((0,T) \times \T)$ and endowed with its subordinate norm. As such, it is a Polish space, as well as the other components of the product in \eqref{global space for tightness} (in particular $\BBB$, since $L^2([0,T];H^{\alpha}(\T))$ is separable and reflexive, so the weak topology of its dual $L^2([0,T];H^{-\alpha}(\T))$, which coincides with the weak-star topology, is metrizable on balls).

%%%%%%%%%%%%%%%%%%%%%%%%%%%%%%%%%%%%%%%%%%%%%%%%%%%%%%%%%%%%%%%%%%%%%%%%
\begin{pro}[Tightness result]\label{prop:tightness}
	The law of the sequence $(Z^\mu)_{\mu\in\Lambda}$ is tight on $	\mathscr{Z}$.
\end{pro}
%%%%%%%%%%%%%%%%%%%%%%%%%%%%%%%%%%%%%%%%%%%%%%%%%%%%%%%%%%%%%%%%%%%%%%%%

%%%%%%%%%%%%%%%%%%%%%%%%%%%%%%%%%%%%%%%%%%%%%%%%%%%%%%%%%%%%%%%%%%%%%%%%
\begin{proof}[Proof of Proposition~\ref{prop:tightness}]

The proof falls into several steps.\medskip

\textbf{Step 1.}
From the weak formulation \eqref{weakuSK} satisfied by $u^\mu$ and the It\^o formula given in Proposition \ref{Proposition:Ito2}, we obtain the two following equations on $ \ww^\mu$ and $\imp^\mu$, understood in $H^{-1}(\T)$:
\begin{equation*}
\ud \left( \mathcal{V}^\mu + \ww^\mu\right)  +  \aaa^\mu \, \ud t 
= \left[\mathcal{C}_2(u^\mu) \right]_{xx} \ud t+ f(u^\mu)\, \ud t +  \Phi \, \ud W^\mu,
\end{equation*}
and
\begin{equation} \label{weaku5}
	\ud \left( \mu \tfrac{u_t^\mu}{c(u^\mu)} +  \imp^\mu \right) + \tfrac{c'(u^\mu)}{c(u^\mu)^2 \gamma(u^\mu)}  r^\mu\, \ud t 
	= \left[c(u^\mu) u_x^\mu \right]_x \ud t+ \tfrac{f(u^\mu)}{c(u^\mu)}\, \ud t + \tfrac{1}{c(u^\mu)} \Phi \, \ud W^\mu.
\end{equation}
The improved estimate on the energy in \eqref{main_estimates2} (see also \eqref{Eestimate good}) shows that the ``perturbations'' $\mu \tfrac{u_t^\mu}{c(u^\mu)}$ and $ \mathcal{V}^\mu$ satisfy
	\begin{equation}\label{perturbations are perturbations}
		\lim_{\mu \to 0} \E \left\|    \mu \tfrac{u_t^\mu}{c(u^\mu)} \right\|_{C([0,T]; L^2(\T))} \leqslant \tfrac{1}{c_1} \lim_{\mu \to 0}   \E \left\|  \mathcal{V}^\mu \right\|_{C([0,T]; L^2(\T))} = 0.
	\end{equation}
	Therefore, both the laws of $(\mathcal{V}^\mu)_{ \mu\in\Lambda}$ and $\left( \mu \tfrac{u_t^\mu}{c(u^\mu)} \right)_{\mu\in\Lambda}$ are tight in $C([0,T]; L^2(\T))$. 
	From Equation \eqref{weaku5}, we can deduce that, for any $\delta \in (0,1]$, 
	\begin{equation}\label{pmu plus perturb tight}
		 \left(\mu \tfrac{u_t^\mu}{c(u^\mu)} + \imp^\mu\right)_{\mu\in\Lambda}\mbox{ is tight in }C([0,T]; H^{-\delta}(\T)).
	\end{equation}
	Indeed, using the embedding $L^1(\T) \hookrightarrow H^{-1}(\T)$, the parabolic estimate and the estimate on the frictional energy contained in \eqref{main_estimates2}, we obtain for any $t>s$ and $p\geqslant 1$ the bound
\begin{align}\nonumber
\E \left[\left\|  \int_s^t \left[ \left[c(u^\mu) u_x^\mu \right]_x - \tfrac{c'(u^\mu)}{c(u^\mu)^2 \gamma(u^\mu)} r^\mu \right] \ud \sigma \right\|_{H^{-1}}^p \right]
&\leqslant C(p)  \E \left[ \left(\int_s^t \left(  \| c(u^\mu) u^\mu_x \|_{L^2} +  \left\| r^\mu  \right\|_{L^1} \right) \ud \sigma\right)^p\right]\\ \nonumber
&\leqslant C(p)  (t-s)^{p/2}.
\end{align}
Using also the $L^\infty_t L^2_x$ bound on $u^\mu$ in \eqref{main_estimates2} we have
\begin{align}\nonumber
\E\left[ \left\|  \int_s^t  \tfrac{f(u^\mu)}{c(u^\mu)}\, \ud \sigma \right\|_{H^{-1}}^p\right] \leqslant C(p)  \E\left[\left( \int_s^t \left( \| u^\mu  \|_{L^2}^2 + 1 \right) \ud \sigma\right)^p\right] \leqslant C(p)  (t-s)^p.
\end{align}
On another side, using a Kolmogorov argument (see \cite[Theorem 5.11 and Theorem
5.15]{DaPratoZabczyk14}, and see \cite[Theorem 3.3]{DaPratoZabczyk14}), one can show that for any $\alpha \in (0,1/2)$ 
\begin{equation*}
\E \left\| \int_0^\cdot \tfrac{1}{c(u^\mu)} \Phi\, \ud W^\mu(\sigma)  \right\|_{C^\alpha ([0,T]; H^{-1}(\T))}  \leqslant C_\alpha .
\end{equation*}
Therefore, we have 
\begin{equation*}
\E \left\|    \mu \tfrac{u_t^\mu}{c(u^\mu)} + \imp^\mu \right\|_{C^\alpha ([0,T]; H^{-1}(\T))} \leqslant C_\alpha.
\end{equation*}
Moreover, using \eqref{main_estimates2} and \eqref{coeff-c} we obtain 
\begin{equation*}
\E \left\|    \mu \tfrac{u_t^\mu}{c(u^\mu)} + \imp^\mu \right\|_{C([0,T]; L^2(\T))} \leqslant C.
\end{equation*}
Using \cite[Theorem 5]{Simon87}, we obtain \eqref{pmu plus perturb tight}. Finally, combining \eqref{perturbations are perturbations} and  \eqref{pmu plus perturb tight}, we can conclude that the sequence $(\imp^\mu)_{ \mu\in\Lambda}$ is tight in $C([0,T]; H^{-\delta}(\T))$ for any $\delta \in (0,1]$. \medskip

\textbf{Step 2.} Let $A_R$ denote the closed ball in $L^2([0,T]; H^1(\T))$ of center $0$ and radius $R$. By \eqref{main_estimates2} and by Step 1., we have, for $R$ large, $\imp^\mu\in A_R\cap K_R$ with high probability, uniformly w.r.t. $\mu$, where $K_R$ is a compact subset of the space $C([0,T]; H^{-\delta}(\T))$. By interpolation, $A_R\cap K_R$ is compact in $L^{2(1+\delta)/\delta}([0,T];L^2(\T))$, and so, since $\delta\in(0,1]$ is arbitrary, the law of $(\imp^\mu)_{\mu\in\Lambda}$ is tight in $L^m([0,T];L^2(\T))$ for any $m \in [1,\infty)$. To justify the previous assertion in full details, we use the following uniform bounds on $K_R$ (see \cite[Theorem 1]{Simon87} for instance):
\begin{equation*}
	\sup_{\imp \in K_R} \sup_{t \in [0,T]} \left\| \imp(t)\right\|_{H^{-\delta}(\T)}\leqslant C_R,\qquad\lim_{h \to 0} \sup_{\imp \in K_R} \sup_{t \in [0,T-h]} \left\| \imp(t+h,\cdot) - \imp(t,\cdot) \right\|_{H^{-\delta}(\T)} =0.
\end{equation*}
We also use the interpolation inequality
\begin{equation*}
\|\imp\|_{L^2(\T)} \leqslant C \|\imp\|_{H^{-\delta}(\T)}^{\frac{1}{1+\delta}} \|\imp\|_{H^1}^{\frac{\delta}{1+\delta}(\T)},
\end{equation*}
to deduce, for $m_\delta \eqdef \frac{2(1+\delta)}{\delta}$, 
\begin{equation*}
	\left\| \imp\right\|_{L^{m_\delta}([0,T];L^2(\T))}\leqslant C_R,
\end{equation*}
for all $\imp \in A_R\cap K_R$, and
\begin{equation*}
\lim_{h \to 0} \sup_{\imp \in A_R\cap  K_R} \left\| \imp(\cdot+h) - \imp \right\|_{L^{m_\delta}([0,T-h];L^2(\T))} =0.
\end{equation*}
We conclude with the Riesz--Fréchet--Kolmogorov theorem.\medskip

\textbf{Step 3.} We consider the remaining components of $Z^\mu$. Only the factor $(\aaa^\mu)_{\mu\in\Lambda}$ really needs to be discussed, but, by \eqref{main_estimates2}, we have that $(\aaa^\mu)_{\mu\in\Lambda}$ is bounded in $L^1(\Omega \times (0,T) \times \T)$. By compact injection of $L^1((0,T) \times \T)$ into $H^{-2}((0,T) \times \T)$, the sequence of laws of $\aaa^\mu$ is tight in $H^{-2}((0,T) \times \T)$. Similarly, combining the embedding $L^1(\T) \hookrightarrow H^{-\alpha}(\T)$ with \eqref{main_estimates2} we deduce that $\mathscr{S}^\mu$ is bounded in $L^2(\Omega)$, so $(\mathscr{R}^\mu,\mathscr{S}^\mu)$ is tight in $\bar{B}_\alpha\times\R$. Finally, using \cite[Theorem 1.3]{BillingsleyBook}, we obtain that the law of $W$ is tight in $C([0,T]; \mathfrak{U}_{-1})$.
\end{proof}
%%%%%%%%%%%%%%%%%%%%%%%%%%%%%%%%%%%%%%%%%%%%%%%%%%%%%%%%%%%%%%%%%%%%%%%%

%%%%%%%%%%%%%%%%%%%%%%%%%%%%%%%%%%%%%%%%%%%%%%%%%%%%%%%%%%%%%%%%%%%%%%%%
\subsection{The convergence result}\label{subsec:CV}
%%%%%%%%%%%%%%%%%%%%%%%%%%%%%%%%%%%%%%%%%%%%%%%%%%%%%%%%%%%%%%%%%%%%%%%%

%%%%%%%%%%%%%%%%%%%%%%%%%%%%%%%%%%%%%%%%%%%%%%%%%%%%%%%%%%%%%%%%%%%%%%%%
\subsubsection{Some elements of convergence}\label{subsubsec:CV}
%%%%%%%%%%%%%%%%%%%%%%%%%%%%%%%%%%%%%%%%%%%%%%%%%%%%%%%%%%%%%%%%%%%%%%%%

Using the Skorokhod theorem, we obtain that for any two sequences $(\mu_k^1)_k$, $(\mu_k^2)_k$ converging to $0$, there exist two subsequences (denoted also $(\mu_k^1)_k$ and $(\mu_k^2)_k$), a probability space $(\tilde{\Omega}, \tilde{\mathcal{F}}, \tilde{\Pro})$, and a sequence of random variables
\begin{equation*}
\tilde{Y}_k \eqdef \left(\imp^{1}_k, \imp^{2}_k, \mathcal{V}^{1}_k, \mathcal{V}^{2}_k, \mathscr{R}^{1}_k, \mathscr{R}^{2}_k, \mathscr{S}^{1}_k, \mathscr{S}^{2}_k, \aaa^{1}_k,\aaa^{2}_k, \tilde{W}_k \right)_k
\end{equation*} 
that has the same law as
\begin{equation*}
Y_k\eqdef\left(\imp^{\mu_k^1}, \imp^{\mu_k^2}, \mathcal{V}^{\mu_k^1}, \mathcal{V}^{\mu_k^2}, \mathscr{R}^{\mu^1_k}, \mathscr{R}^{\mu^2_k}, \mathscr{S}^{\mu^1_k}, \mathscr{S}^{\mu^2_k}, \aaa^{\mu_k^1}, \aaa^{\mu_k^2}, W \right)_k,
\end{equation*} 
for any $k \in \N$. Moreover, $\tilde{\Pro}$-almost surely, the sequence $Y_k$ converges in 
\begin{equation*}
\mathcal{X}^2 \times \left(C([0,T];L^2(\T)) \right)^2 \times \bar{B}_\alpha^2 \times \R^2 \times \left(H^{-2}((0,T) \times \T)\right)^2 \times C([0,T]; \mathfrak{U}_{-1}),
\end{equation*}
to the random variable 
\begin{equation*}
\tilde{Y} \eqdef \left(\imp^{1}, \imp^{2}, 0, 0, \mathscr{R}^{1}, \mathscr{R}^{2}, \mathscr{S}^{1}, \mathscr{S}^{2}, \aaa^{1}, \aaa^{2}, \tilde{W} \right).
\end{equation*} 
Due to the equality of laws, we have the initial identity 
\begin{equation*}
\imp^1_k(0,\cdot) = \imp^2_k(0,\cdot) =  \underline{\Gamma}(u_0), \quad \forall k \in \N,
\end{equation*}
and, for $i\in\{1,2\}$, the bound (due to \eqref{main_estimates})
\begin{equation}\label{bound pi H1}
	\tilde{\E}\left[\|\imp^i_k\|_{L^2([0,T];H^1(\T))}^p\right]\leqslant C(p)<\infty,
\end{equation}
for all $p>0$.

We recall now a classical result on Sobolev spaces that will be used in this section. By the arguments in \cite[Section 2]{ConstantinMolinet2002}, we can show the following estimates.
%%%%%%%%%%%%%%%%%%%%%%%%%%%%%%%%%%%%%%%%%%%%%%%%%%%%%%%%%%%%%%%%%%%%%%%%
\begin{proposition}[Smooth functions operating on Sobolev spaces]\label{prop:smooth operate}
Let $s \in (1/2,1]$, $v,w \in H^s(\T)$ and let $F\colon\R\to\R$ be a smooth function. Then $F$ sends $H^s(\T)$ in $H^s(\T)$ and we have the estimates
\begin{align*}
\|F(v)-F(0)\|_{H^s} &\leqslant C(\|v\|_{L^\infty}) \|v\|_{H^s}, \\
\|F(v)-F(w)\|_{H^s} &\leqslant C(\|v\|_{L^\infty}) \left(\|v-w\|_{H^s} + \|w\|_{H^s} \|v-w\|_{L^\infty}\right).
\end{align*}
Moreover, if $F'$ is bounded we have 
\begin{equation*}
\|F(v)-F(0)\|_{H^s} \leqslant C \|v\|_{H^s}, 
\end{equation*}
and if both $F'$ and $F''$ are bounded, then 
\begin{equation*}
\|F(v)-F(w)\|_{H^s} \leqslant C \left(\|v-w\|_{H^s} + \|w\|_{H^s} \|v-w\|_{L^\infty}\right).
\end{equation*}
\end{proposition}
%%%%%%%%%%%%%%%%%%%%%%%%%%%%%%%%%%%%%%%%%%%%%%%%%%%%%%%%%%%%%%%%%%%%%%%%

Let us first establish the following result.

%%%%%%%%%%%%%%%%%%%%%%%%%%%%%%%%%%%%%%%%%%%%%%%%%%%%%%%%%%%%%%%%%%%%%%%%
\begin{proposition}\label{pro:Lipconv} Let $G\colon\R\to\R$ be a Lipschitz function. For any $i \in \{1,2\}$, and any $n \in \N^*$, we have the convergence in probability 
\begin{equation*}
\lim_{k\to \infty} \left\|G(\imp^i_k) - G(\imp^i) \right\|_{L^n([0,T]; L^2(\T))  \cap L^2([0,T];H^{1-\frac1n}(\T)) } = 0.
\end{equation*}
\end{proposition}
%%%%%%%%%%%%%%%%%%%%%%%%%%%%%%%%%%%%%%%%%%%%%%%%%%%%%%%%%%%%%%%%%%%%%%%%

%%%%%%%%%%%%%%%%%%%%%%%%%%%%%%%%%%%%%%%%%%%%%%%%%%%%%%%%%%%%%%%%%%%%%%%%
\begin{proof}[Proof of Proposition~\ref{pro:Lipconv}] The Lipschitz continuity of $G$ directly implies the $\tilde{\Pro}$-a.s. convergence, and thus convergence in probability, in $L^n([0,T]; L^2(\T))$.
Now, we use the ``interpolation'' inequality 
\begin{multline}\label{interpolation H}
 \left\|G(\imp^i_k) - G (\imp^i) \right\|_{L^2([0,T]; H^{1-\frac1n}(\T))} 
 \\
 \leqslant C \left\|G(\imp^i_k) - G (\imp^i) \right\|_{L^2([0,T]; H^{1}(\T))}^{\beta}  \left\|G(\imp^i_k) - G (\imp^i) \right\|_{L^2([0,T] \times\T)}^{1-\beta},
\end{multline}
where $\beta \in (0,1)$ if $n \geqslant 2$. Rewrite \eqref{interpolation H} as $X_k\leqslant CY_k Z_k$. For $\eps,\delta>0$, $R>0$ we have 
\begin{equation}
	\tilde{\Pro}\left( X_k>\eps\right)\leqslant \tilde{\Pro}\left( Y_k>R/C\right)+\tilde{\Pro}\left( Z_k>\eps/R\right).
\end{equation}
By the Markov inequality and the fact (due to \eqref{bound pi H1}) that
\begin{multline}
	\tilde{\E}\left[\left\|G (\imp^i) \right\|_{L^2([0,T]; H^{1}(\T))}\right]
	\leqslant C(G)\tilde{\E}\left[\left\|\imp^i \right\|_{L^2([0,T]; H^{1}(\T))}\right]\\
	\leqslant\liminf_{k\to \infty}C(G)\tilde{\E}\left[\left\|\imp^i_k \right\|_{L^2([0,T]; H^{1}(\T))}\right]
	\leqslant C(G),
\end{multline}
we have
\begin{equation}
	\tilde{\Pro}\left( Y_k>R/C\right)\leqslant CR^{-1/\beta},
\end{equation}
so
\begin{equation}
	\tilde{\Pro}\left( Y_k>R/C\right)< \delta,
\end{equation}
for $R$ large enough. For such a $R$, we have then
\begin{equation}
	\tilde{\Pro}\left( Z_k>\eps/R\right)<\delta,
\end{equation}
for $k$ large enough, and thus $\tilde{\Pro}\left( X_k>\eps\right)< 2\delta$ for $k$ large enough.
\end{proof}
%%%%%%%%%%%%%%%%%%%%%%%%%%%%%%%%%%%%%%%%%%%%%%%%%%%%%%%%%%%%%%%%%%%%%%%%

Let $i \in \{1,2\}$. We introduce the limit elements

\begin{equation*}
r^i \eqdef \mathscr{S}^i  \mathscr{R}^i, \qquad
u^i_k \eqdef \underline{\Gamma}^{-1}(\imp^i_k), \qquad u^i \eqdef \underline{\Gamma}^{-1}(\imp^i), \qquad \ww^i_k \eqdef \Gamma(u^i_k), \qquad \ww^i \eqdef \Gamma(u^i).
\end{equation*}

%%%%%%%%%%%%%%%%%%%%%%%%%%%%%%%%%%%%%%%%%%%%%%%%%%%%%%%%%%%%%%%%%%%%%%%%
\begin{proposition}\label{pro:r}
Let $G$ be a Lipschitz function and let $\psi \in C([0,T]; H^\alpha (\T))$. Then, as $k \to \infty$, we have the following convergence in probability:
\begin{equation*}
\int_0^T
\langle  G(u_k^i)r^i_k, \psi \rangle_{H^{-\alpha}(\T), H^{\alpha}(\T)} \ud s \to  \int_0^T
 \langle G(u^i) r^i, \psi \rangle_{H^{-\alpha}(\T), H^{\alpha}(\T)} \ud s.
\end{equation*}
\end{proposition}
%%%%%%%%%%%%%%%%%%%%%%%%%%%%%%%%%%%%%%%%%%%%%%%%%%%%%%%%%%%%%%%%%%%%%%%%

%%%%%%%%%%%%%%%%%%%%%%%%%%%%%%%%%%%%%%%%%%%%%%%%%%%%%%%%%%%%%%%%%%%%%%%%
\begin{proof}[Proof of Proposition~\ref{pro:r}] By Proposition~\ref{pro:Lipconv} (applied with the non-linearity $G\circ \underline{\Gamma}^{-1}$), the sequence $(G(u_k^i))$ converges, in probability, to $G(u^i)$ in $L^2([0,T];H^{\alpha}(\T))$, so
	\begin{equation*}
		V_k^i\eqdef G(u_k^i)\psi \to V^i\eqdef G(u^i)\psi
	\end{equation*}
	in probability in $L^2([0,T];H^{\alpha}(\T))$ as well, since $H^{\alpha}(\T)$ is a Banach algebra (recall that $\alpha\in(1/2,1)$). 
	We also have $r^i_k \rightharpoonup r^i$ in $L^2([0,T];H^{-\alpha}(\T))$-weak $\tilde{\Pro}$-a.s., and this implies the convergence in probability
	\begin{equation}\label{weak-strong}
		\dual{V_k^i}{r_k^i}\to \dual{V^i}{r^i},
	\end{equation}
	essentially by strong-weak convergence. Let us give the details of the proof of \eqref{weak-strong}. We use the expansion
	\begin{equation}
			\dual{V_k^i}{r_k^i}=	\dual{V_k^i-V^i}{r_k^i}+	\dual{V^i}{r_k^i-r^i}+	\dual{V^i}{r^i}.
	\end{equation}
	Since 
	\begin{equation}
		\dual{V^i}{r_k^i-r^i}\to 0,\quad \tilde{\Pro}-\mbox{a.s.},
	\end{equation}
	our aim is to prove that $\dual{V_k^i-V^i}{r_k^i}\to 0$ in probability. We can therefore assume, without loss of generality, that $V^i=0$.  Then, as in the proof of Proposition~\ref{pro:Lipconv}, we decompose
	\begin{equation}
		\tilde{\Pro}\left( |\dual{V_k^i}{r_k^i}|>\eps\right)
		\leqslant
		\tilde{\Pro}\left( \|r^i_k\|_{L^2([0,T];H^{-\alpha}(\T))}>R\right)
		+
		\tilde{\Pro}\left( \|V^i_k\|_{L^2([0,T];H^{\alpha}(\T))}>\eps/R\right),
	\end{equation}
	and use the Markov inequality and the uniform bound \eqref{main_estimates2} to get
	\begin{equation}
		\tilde{\Pro}\left( |\dual{V_k^i}{r_k^i}|>\eps\right)
		\leqslant
		CR^{-1}
		+
		\tilde{\Pro}\left( \|V^i_k\|_{L^2([0,T];H^{\alpha}(\T))}>\eps/R\right).
	\end{equation}
	Choosing $R$ large enough and then $\eps$ small enough gives the desired result.
\end{proof}
%%%%%%%%%%%%%%%%%%%%%%%%%%%%%%%%%%%%%%%%%%%%%%%%%%%%%%%%%%%%%%%%%%%%%%%%

%%%%%%%%%%%%%%%%%%%%%%%%%%%%%%%%%%%%%%%%%%%%%%%%%%%%%%%%%%%%%%%%%%%%%%%%
\subsubsection{Limiting equations}\label{subsubsec:limit equations}
%%%%%%%%%%%%%%%%%%%%%%%%%%%%%%%%%%%%%%%%%%%%%%%%%%%%%%%%%%%%%%%%%%%%%%%%

Let $\left(\tilde{\mathcal{F}}^{0,k}_t\right)$ denote the filtration generated by the process $Y_k$ and let $\left(\tilde{\mathcal{F}}^k_t\right)$ be the augmented filtration, obtained by the completion of the right-continuous filtration $\left(\tilde{\mathcal{F}}^{0,k}_{t+}\right)$. We consider then the stochastic basis $(\tilde{\Omega}, \tilde{\mathcal{F}}, \tilde{\Pro},(\tilde{\mathcal{F}}^k_t),\tilde{W})$ and the associated stochastic integral.  
	Let $\psi \in H_0^2((0,T) \times \T)$. We claim that 
	\begin{gather}\nonumber
		-\int_0^T \int_\T \psi_t \left( \mathcal{V}^i_k+ \ww^i_k \right)  \ud x\, \ud s + \int_0^T \int_\T \psi  \aaa^i_k\, \ud x\, \ud s \\
		\label{eq:ek}
		= \int_0^T \int_\T \psi_{xx} \mathcal{C}_2(u^i_k)\, \ud x\,  \ud s+ \int_0^T \int_\T \psi f(u^i_k)\,  \ud x\, \ud s + \int_0^T \int_\T \psi  \Phi \, \ud x\, \ud \tilde{W}_k (s).
	\end{gather}
	and
	\begin{gather}\nonumber
		- \int_0^T \int_\T \psi_t \left( \tfrac{\mathcal{V}^i_k}{c(u^i_k)} + \imp^i_k  \right)  \ud x\, \ud s
		+ \int_0^T \int_\T \psi  \tfrac{c'(u^i_k)}{c(u^i_k)^2 \gamma(u^i_k)} r^i_k\, \ud x\, \ud s \\ 
		\label{eq:pk}
		= \int_0^T \int_\T \psi_{xx} \mathcal{C}(u^i_k) \, \ud x\, \ud s+ \int_0^T \int_\T \psi \tfrac{f(u^i_k)}{c(u^i_k)} \, \ud x \, \ud s + \int_0^T \int_\T  \tfrac{\psi}{c(u^i_k)} \Phi\, \ud x\, \ud \tilde{W}_k (s),
	\end{gather}
for all $k$. This follows from the weak formulation \eqref{weaku} and the It\^o formula in Proposition~\ref{Proposition:Ito2}, which gives \eqref{eq:ek} and \eqref{eq:pk} for the original set of unknowns
\begin{equation*}
	(\Omega, \mathcal{F}, \Pro,(\mathcal{F}_t),W,u^{\mu_k^i},\ww^{\mu_k^i},\imp^{\mu_k^i},\mathcal{V}^{\mu_k^i},\aaa^{\mu_k^i},r^{\mu_k^i}).
\end{equation*}
To obtain \eqref{eq:ek}-\eqref{eq:pk}, we can then either use a characterization in terms of martingales as in \cite{Ondrejat10,BrzezniakOndrejat11,HofmanovaSeidler12,Hofmanova13b,DebusscheHofmanovaVovelle16} or first start from the identity
\begin{equation*}
	\E\Psi(u^{\mu_k^i},\ww^{\mu_k^i},\imp^{\mu_k^i},\mathcal{V}^{\mu_k^i},\aaa^{\mu_k^i},r^{\mu_k^i},W)=0,
\end{equation*}
where, if we focus on \eqref{eq:pk} for instance,
\begin{gather*}
	\Psi(u,\ww,\imp,\mathcal{V},\aaa,r,W)\eqdef 
	1\wedge\Bigg| \int_0^T \int_\T \psi_t \left( \tfrac{\mathcal{V}}{c(u)} + \imp  \right)  \ud x\, \ud s
	- \int_0^T \int_\T \psi  \tfrac{c'(u)}{c(u)^2 \gamma(u)} r\, \ud x\, \ud s \\ 
	- \int_0^T \int_\T \psi_{xx} \mathcal{C}(u) \, \ud x\, \ud t- \int_0^T \int_\T \psi \tfrac{f(u)}{c(u)} \, \ud x \, \ud s - \int_0^T \int_\T  \tfrac{\psi}{c(u)} \Phi\, \ud x\, \ud W (s)\Bigg|,
\end{gather*}
and then, in a second step, justify that the real-valued map $\Psi$ defined on 
\begin{equation*}
(L^2([0,T] \times \T ))^2 \times	\mathcal{X} \times C([0,T];L^2(\T))  \times H^{-2}((0,T) \times \T) \times L^2([0,T];H^{-\alpha}(\T)) \times C([0,T]; \mathfrak{U}_{-1})
\end{equation*}
is Borel, to conclude that we also have
\begin{equation*}
	\tilde{\E}\Psi(u_k^i,\ww_k^i,\imp_k^i,\mathcal{V}_k^i,\aaa_k^i,r_k^i,\tilde{W}_k)=0.
\end{equation*}
The most delicate point to treat, in the proof that $\Psi$ is Borel, is the stochastic integral. On that aspect, one can apply the trick of Bensoussan, \cite[p.~282]{Bensoussan1995}, which uses a regularization by convolution of the integrand of the stochastic integral and the stochastic Fubini theorem to ``fully integrate'' the $\ud W(s)$. Taking the limit $k \to \infty$ in \eqref{eq:ek} and \eqref{eq:pk} also requires a specific study of the stochastic integrals, and for this we refer to \cite[Lemma 2.1]{DebusscheGlattHoltzTemam2011}, which yields the convergence in probability of the stochastic integral. Taking the limit in the other terms shows no difficulty with the results already established in Section~\ref{subsubsec:CV} (we use Proposition \ref{pro:r} in particular, and the fact that $\ud(X,Y)\eqdef \tilde{\E}(1\wedge\|X-Y\|)$ is a metric for the convergence in probability). We obtain therefore the equations

\begin{gather}\nonumber
-\int_0^T \int_\T \psi_t  \ww^i\, \ud x\, \ud s +  \langle   a^i, \psi \rangle_{H^{-2}, H^2_0}  \\ \nonumber
= \int_0^T \int_\T \psi_{xx} \mathcal{C}_2(u^i)\, \ud  x\,  \ud s+ \int_0^T \int_\T \psi f(u^i)\,  \ud x\, \ud s + \int_0^T \int_\T \psi  \Phi \, \ud x\, \ud \tilde{W} (s),
\end{gather}
and 
\begin{gather}\nonumber
- \int_0^T \int_\T \psi_t \imp^i\, \ud x\, \ud s + \int_0^T \left\langle \tfrac{c'(u^i_k)}{c(u^i_k)^2 \gamma(u^i_k)} r^i,  \psi \right\rangle_{H^{-\alpha}(\T), H^\alpha(\T)} \ud s   \\ \label{underlinerho}
= \int_0^T \int_\T \psi_{xx} \mathcal{C}(u^i)\,  \ud x\, \ud s+ \int_0^T \int_\T \psi \tfrac{f(u^i)}{c(u^i)} \, \ud x \, \ud s + \int_0^T \int_\T  \tfrac{\psi}{c(u^i)} \Phi\, \ud x\, \ud \tilde{W} (s).
\end{gather}

%%%%%%%%%%%%%%%%%%%%%%%%%%%%%%%%%%%%%%%%%%%%%%%%%%%%%%%%%%%%%%%%%%%%%%%%
\subsubsection{Identification of the non-linear terms}\label{subsubsec:identification defect measures}
%%%%%%%%%%%%%%%%%%%%%%%%%%%%%%%%%%%%%%%%%%%%%%%%%%%%%%%%%%%%%%%%%%%%%%%%

We recall now the following classical result.
%%%%%%%%%%%%%%%%%%%%%%%%%%%%%%%%%%%%%%%%%%%%%%%%%%%%%%%%%%%%%%%%%%%%%%%%
\begin{lem}\label{lem:weak cv meas} Let $U \subset \R^N$ be an open set and let $\mathcal{D}'(U)$ be the space of distributions on $U$.
Assume that $v_n$ converges weakly in $L^2(U)$ to $v$ and that $v_n^2$ converges to $w$ in $\mathcal{D}'(U)$. Then $w - v^2$ is a non-negative measure.
\end{lem}
%%%%%%%%%%%%%%%%%%%%%%%%%%%%%%%%%%%%%%%%%%%%%%%%%%%%%%%%%%%%%%%%%%%%%%%%

%%%%%%%%%%%%%%%%%%%%%%%%%%%%%%%%%%%%%%%%%%%%%%%%%%%%%%%%%%%%%%%%%%%%%%%%
\begin{proof}[Proof of Lemma~\ref{lem:weak cv meas}]
Let $\psi \in C_c^\infty(U)$ be a non-negative function. Then, $\sqrt{\psi} v_n$ converges weakly in $L^2(U)$ to $\sqrt{\psi} v$. Therefore
\begin{equation*}
\int_U v^2 \psi\, \ud x \leqslant \liminf_n \int_U v_n^2 \psi\, \ud x = \lim_n \int_U v_n^2 \psi\, \ud x = \langle w,\psi \rangle.
\end{equation*}
This ends the proof.
\end{proof}
%%%%%%%%%%%%%%%%%%%%%%%%%%%%%%%%%%%%%%%%%%%%%%%%%%%%%%%%%%%%%%%%%%%%%%%%
Using the previous Lemma, we define the defect measure 
\begin{equation}\label{DM}
\hat{\aaa}^i \eqdef \aaa^i - c'(u^i) c(u^i) (u^i_x)^2 \geqslant 0
\end{equation}
to obtain the following version of the equation on $\ww^i$:
\begin{gather}\nonumber
-\int_0^T \int_\T \psi_t  \ww^i\, \ud x\, \ud s +  \int_0^T \int_\T  c'(u^i) c(u^i) (u^i_x)^2 \psi\, \ud x\, \ud s +  \langle   \hat{\aaa}^i, \psi \rangle_{H^{-2}, H^2_0} \\ \label{eq ei with defect measure}
= \int_0^T \int_\T \psi_{xx} \mathcal{C}_2(u^i)\, \ud x\,  \ud s+ \int_0^T \int_\T \psi f(u^i)\, \ud x\, \ud s + \int_0^T \int_\T \psi  \Phi \, \ud x\, \ud \tilde{W} (s).
\end{gather}
On the other hand, we have $\ww^i = \Gamma (\underline{\Gamma}^{-1} (\imp^i))$, so, starting from the equation \eqref{underlinerho} on $\imp^i$ and using the It\^o formula in Proposition \ref{Proposition:Ito1}, we obtain 
\begin{gather}\nonumber
- \int_0^T \int_\T \psi_t \ww^i\, \ud x\, \ud s +\int_0^T \left\langle \tfrac{c'(u^i)}{c(u^i) \gamma(u^i)} r^i, \psi \right\rangle_{H^{-\alpha}, H^\alpha} \, \ud s -\int_0^T \int_\T  \tfrac{q \psi c'(u^i) }{2c(u^i) \gamma(u^i)}   
 \, \ud x \, \ud s \\ \label{rho2.5}
= - \int_0^T \int_\T \left( c(u^i) \psi \right)_x c(u^i) u^i_x\, \ud x\, \ud s + \int_0^T \int_\T \psi f(u^i) \, \ud x \, \ud s + \int_0^T \int_\T  \psi \Phi\, \ud x\, \ud \tilde{W} (s).
\end{gather}
Comparing \eqref{rho2.5} with \eqref{eq ei with defect measure}, we obtain the identity
\begin{equation*}
\langle \hat{\aaa}^i,  \psi \rangle_{H^{-2}, H^2_0}  = \int_0^T \left\langle \tfrac{c'(u^i)}{c(u^i) \gamma(u^i)} r^i, \psi \right\rangle_{H^{-\alpha}, H^\alpha} \, \ud s -\int_0^T \int_\T  \tfrac{q \psi c'(u^i) }{2c(u^i) \gamma(u^i)}   
 \, \ud x \, \ud s \geqslant 0,
\end{equation*}
for any non-negative $\psi \in C^\infty_c((0,\infty) \times \T)$. We conclude that, in the sense of distributions, 
\begin{equation}\label{DM2}
  c'(u^i) r^i - \tfrac{q c'(u^i)}{2} = c(u^i) \gamma(u^i) \hat{\aaa}^i \geqslant 0,
\end{equation}
where $\hat{\aaa}^i$ is defined in \eqref{DM}.

%%%%%%%%%%%%%%%%%%%%%%%%%%%%%%%%%%%%%%%%%%%%%%%%%%%%%%%%%%%%%%%%%%%%%%%%
\begin{pro}\label{prop:Identification}
 For $i \in \{1,2\}$, $\tilde{\Pro}$-almost surely we have
\begin{equation}\label{trivial defect measure}
c'(u^i) r^i = \tfrac{q c'(u^i)}{2} , \qquad  \hat{\aaa}^i =0, \qquad \mathrm{in}\ \mathcal{D}'((0,T) \times \T).
\end{equation}
\end{pro}
%%%%%%%%%%%%%%%%%%%%%%%%%%%%%%%%%%%%%%%%%%%%%%%%%%%%%%%%%%%%%%%%%%%%%%%%

%%%%%%%%%%%%%%%%%%%%%%%%%%%%%%%%%%%%%%%%%%%%%%%%%%%%%%%%%%%%%%%%%%%%%%%%
\begin{proof}[Proof of Proposition~\ref{prop:Identification}]
Using \eqref{ene_mu02} and the equality of laws, we obtain for any $i\in \{ 1,2\}$ that 
\begin{equation*}
	\lim_k \E \left( \int_0^T \int_\T \left( 2 \mu \gamma(u^i_k) (\partial_t u^i_k)^2 - q \right) \psi \, \ud x\,  \ud t \right)^+ = 0.
\end{equation*}
Therefore, up to a subsequence, we have $\tilde{\Pro}$-almost-surely, 
\begin{equation*}
\left\langle r^i, \psi\right\rangle_{L^2 H^{-\alpha}, L^2 H^\alpha} =	\lim_k   \int_0^T \int_\T  r^i_k  \psi \, \ud x\,  \ud t  \leqslant \half \int_0^T \int_\T q\psi \, \ud x\,  \ud t.
\end{equation*}
This implies that $r^i \leqslant q/2$ in the sense of distributions. Using now \eqref{DM2}, we conclude to \eqref{trivial defect measure}.
\end{proof}
%%%%%%%%%%%%%%%%%%%%%%%%%%%%%%%%%%%%%%%%%%%%%%%%%%%%%%%%%%%%%%%%%%%%%%%%

	We can now exploit the identity \eqref{trivial defect measure} in \eqref{underlinerho} to state the following result.
	
%%%%%%%%%%%%%%%%%%%%%%%%%%%%%%%%%%%%%%%%%%%%%%%%%%%%%%%%%%%%%%%%%%%%%%%%
\begin{pro}\label{prop:limit equation on p}
	Both $\imp^1$ and $\imp^2$ are solutions, in the sense of Definition~\ref{def:weak sol limit equation p}, to the quasi-linear stochastic equation \eqref{limeq version rhoc}, with initial datum $\imp_0=\underline{\Gamma}(u_0)$.
\end{pro}
%%%%%%%%%%%%%%%%%%%%%%%%%%%%%%%%%%%%%%%%%%%%%%%%%%%%%%%%%%%%%%%%%%%%%%%%
Using the It\^o formula in Proposition \ref{Proposition:Ito1}, we obtain the following. 
%%%%%%%%%%%%%%%%%%%%%%%%%%%%%%%%%%%%%%%%%%%%%%%%%%%%%%%%%%%%%%%%%%%%%%%%
\begin{pro}\label{prop:limit equation on u}
	Both $u^1$ and $u^2$ are solutions, in the sense of Definition~\ref{def:weak sol limit equation u}, to the quasi-linear stochastic equation \eqref{limeq version u}, with initial datum $u_0$.
\end{pro}
%%%%%%%%%%%%%%%%%%%%%%%%%%%%%%%%%%%%%%%%%%%%%%%%%%%%%%%%%%%%%%%%%%%%%%%%

%%%%%%%%%%%%%%%%%%%%%%%%%%%%%%%%%%%%%%%%%%%%%%%%%%%%%%%%%%%%%%%%%%%%%%%%
\subsubsection{Final steps of the proof of convergence}\label{sec:uniquness}
%%%%%%%%%%%%%%%%%%%%%%%%%%%%%%%%%%%%%%%%%%%%%%%%%%%%%%%%%%%%%%%%%%%%%%%%
	Following the arguments in \cite[Theorem 3.1]{HofmanovaZhang17} and \cite[Theorem 6.2]{CerraiXi22}, solutions to \eqref{limeq version rhoc} are unique, so $\imp^1=\imp^2$. By the Gy\"ongy--Krylov argument, \cite[Lemma 1.1]{GyongyKrylov96}, we deduce the convergence in probability $\imp^{\mu_k}\to \imp$ in $\mathcal{X}$, which (due to Proposition \ref{pro:Lipconv}) gives is the convergence property \eqref{CV in proba umu} given in Theorem~\ref{thm:SK}. To identify the limit $u$ and establish \eqref{limeq version u} we just need to repeat the arguments employed for the doubled variable in Sections~\ref{subsubsec:CV} to \ref{subsubsec:identification defect measures}.

\appendix

%%%%%%%%%%%%%%%%%%%%%%%%%%%%%%%%%%%%%%%%%%%%%%%%%%%%%%%%%%%%%%%%%%%%%%%%
\section{Reformulation of the problem}\label{app:equiv fomrulations}
%%%%%%%%%%%%%%%%%%%%%%%%%%%%%%%%%%%%%%%%%%%%%%%%%%%%%%%%%%%%%%%%%%%%%%%%

\paragraph{From $u^\mu$ to $(R^\mu,S^\mu)$.} Let $u^\mu$ be a smooth solution to \eqref{SVWE1}. Let $R^\mu$ and $S^\mu$ be defined by \eqref{defRS}. One has then the two identities
\begin{equation}\label{ux by RS}
	u^\mu_x=\tfrac{S^\mu-R^\mu}{2c(u^\mu)},
\end{equation} 
and
\begin{equation}\label{ut by RS}
	u^\mu_t=\tfrac{S^\mu+R^\mu}{2\sqrt{\mu}}.
\end{equation} 
By differentiation in time and in space in the equation $R^\mu  = \sqrt{\mu} u^\mu_t  -  c(u^\mu)  u^\mu_x$, 
	we obtain the identities
\begin{equation}\label{first dR}
	\sqrt{\mu}\, \ud R^\mu = \mu\, \ud u^\mu_t - \sqrt{\mu}\, c'(u^\mu)u^\mu_t u^\mu_x\, \ud t - \sqrt{\mu}\, c(u^\mu)  u^\mu_{xt} \, \ud t,
\end{equation}
and
\begin{equation}\label{dx R}
	c(u^\mu)  R^\mu_x=\sqrt{\mu} c(u^\mu) u^\mu_{tx}-c(u^\mu)\left(c(u^\mu) u^\mu_x \right)_x.
\end{equation}
By use of the equation \eqref{SVWE1} satisfied by $u^\mu$, \eqref{first dR} and \eqref{dx R} give
\begin{equation}\label{transport R}
	\sqrt{\mu}\, \ud R^\mu  +  c(u^\mu)  R^\mu_x\,  \ud t   +  \gamma(u^\mu) u^\mu_t
	= - \sqrt{\mu}\, c'(u^\mu)u^\mu_t u^\mu_x\, \ud t+  f(u^\mu)\, \ud t + \Phi \,  \ud W.
\end{equation}
We exploit the identities \eqref{ux by RS}-\eqref{ut by RS} in \eqref{transport R} to obtain the equation \eqref{Req} satisfied by $R^\mu$. The procedure is similar for $S^\mu$, and thus
\begin{subequations}\label{SVWE2333}
	\begin{gather}\label{Req333}
		\sqrt{\mu}\, \ud R^\mu  +  c(u^\mu)  R^\mu_x\,  \ud t   +  \gamma(u^\mu) \tfrac{R^\mu+S^\mu}{2\sqrt{\mu}}\, \ud t =  \tilde{c}'(u^\mu) \left[(R^\mu)^2  -  (S^\mu)^2 \right] \ud t  +  f(u^\mu)\, \ud t + \Phi \,  \ud W, \\  \label{Seq333}
		\sqrt{\mu}\, \ud S^\mu  -  c(u^\mu)  S^\mu_x\,  \ud t   +  \gamma(u^\mu) \tfrac{R^\mu+S^\mu}{2\sqrt{\mu}}\, \ud t =  \tilde{c}'(u^\mu) \left[(S^\mu)^2  -  (R^\mu)^2 \right] \ud t  + f(u^\mu)\, \ud t +  \Phi \,  \ud W.
	\end{gather}
\end{subequations}
Using \eqref{ux by RS} again, the system \eqref{SVWE2333} can be rewritten in conservative form as
\begin{subequations}\label{SVWE2 cons}
	\begin{align}
		\sqrt{\mu}\, \ud R^\mu  +  \left(c(u^\mu)  R^\mu\right)_x  \ud t   +&  \gamma(u^\mu) \tfrac{R^\mu+S^\mu}{2\sqrt{\mu}}\, \ud t \nonumber\\
		&=  \tilde{c}'(u^\mu) \left[(R^\mu)^2  -  (S^\mu)^2 +2R^\mu(R^\mu+S^\mu)\right] \ud t  +  f(u^\mu)\, \ud t + \Phi \,  \ud W,\label{Req cons} \\  
		\sqrt{\mu}\, \ud S^\mu  -  \left(c(u^\mu)  S^\mu\right)_x  \ud t   +&  \gamma(u^\mu) \tfrac{R^\mu+S^\mu}{2\sqrt{\mu}}\, \ud t \nonumber\\
		&=  \tilde{c}'(u^\mu) \left[(S^\mu)^2  -  (R^\mu)^2 -2S^\mu(R^\mu+S^\mu)\right] \ud t  + f(u^\mu)\, \ud t +  \Phi \,  \ud W, \label{Seq cons}
	\end{align}
\end{subequations}
In particular, subtracting \eqref{Req cons} from \eqref{Seq cons} yields
\begin{equation}\label{S minus R eq}
	\sqrt{\mu} \left(S^\mu-R^\mu\right)_t- \left(c(u^\mu)  (R^\mu+S^\mu)\right)_x=0.
\end{equation}
The map $(R^\mu,S^\mu)\mapsto u^\mu$ given by \eqref{udef} is obtained as follows. We write first Equation~\eqref{ux by RS}, under the form
\begin{equation}\label{ux by RS C}
	 \left(\mathcal{C}(u^\mu)\right)_x=\tfrac{S^\mu-R^\mu}{2}.
\end{equation}
In the periodic setting that we consider, a necessary and sufficient condition to the solvability of \eqref{ux by RS C} is that
\begin{equation}\label{NSC int ux}
	\int_0^1 \tfrac{S^\mu-R^\mu}{2}(t,y)\, \ud y=0,
\end{equation}
for all $t\geqslant 0$. The equation \eqref{NSC int ux} is satisfied at $t=0$ by the choice \eqref{IC} of the initial data. It remains true for all $t\geqslant 0$, as can be seen by differentiation in time, using \eqref{S minus R eq}. We can therefore integrate \eqref{ux by RS C} to obtain
\begin{equation}\label{to udef v}
	\mathcal{C}(u^\mu(t,x))=\mathcal{C}(\bar{v}^\mu(t))+ \int_0^x {\textstyle \frac{S^\mu  -  R^\mu}{2}}(t,y) \, \ud y,
\end{equation}
for a given function $\bar{v}^\mu(t)$. To determine $\bar{v}^\mu(t)$, we take $x=0$ in \eqref{to udef v}, which yields $\bar{v}^\mu(t)=u(t,0)$. Next we exploit the identity \eqref{ut by RS} at $x=0$ to get the expression
\begin{equation*}
	\bar{v}^\mu(t)=\int_0^t {\textstyle \left(\frac{R^\mu  +  S^\mu}{2 \sqrt{\mu}}  \right) (s,0)} \, \ud s  +  u_0(0).
\end{equation*}

\paragraph{From $(R^\mu,S^\mu)$ to $u^\mu$.} Assume now that $(R^\mu,S^\mu)$ is solution to \eqref{SVWE2} and that $u^\mu$ is defined by \eqref{udef}. Our aim is to establish \eqref{ux by RS} and \eqref{ut by RS}, since then the equation \eqref{SVWE1} on $u^\mu$ is obtained by addition of the equation \eqref{Req} satisfied by $R^\mu$ with the equation \eqref{Seq} on $S^\mu$. By differentiation in $x$ in \eqref{udef}, we obtain first \eqref{ux by RS}. We can therefore use \eqref{ux by RS} to deduce from \eqref{SVWE2} the conservative form \eqref{SVWE2 cons}, and thus \eqref{S minus R eq}. Next, define $\bar{v}^\mu(t)$, so that \eqref{udef} writes as \eqref{to udef v}. Taking $x=0$, we have then $\bar{v}^\mu(t)=u^\mu(t,0)$. By differentiation in time, we also get
\begin{equation*}
c(u^\mu(t,x))u^\mu_t(t,x)=c(u^\mu(t,0))u^\mu(t,0)+\int_0^x {\textstyle \frac{S^\mu_t  -  R^\mu_t}{2}}(t,y) \, \ud y.
\end{equation*}
We integrate \eqref{S minus R eq} on $(0,x)$ to finally obtain \eqref{ut by RS}.

%%%%%%%%%%%%%%%%%%%%%%%%%%%%%%%%%%%%%%%%%%%%%%%%%%%%%%%%%%%%%%%%%%%%%%%%
\section{It\^o formula}\label{app:Ito}
%%%%%%%%%%%%%%%%%%%%%%%%%%%%%%%%%%%%%%%%%%%%%%%%%%%%%%%%%%%%%%%%%%%%%%%%

Our aim is to prove the It\^o formula for 
\begin{equation*}
	\ud u = F \, \ud t + \ud M,
\end{equation*}
which can be written in the form 
\begin{equation}\label{Ito1form1}
-\int_0^T \int_\T \psi_t  u  \, \ud x\, \ud s = \int_0^T \langle F, \psi \rangle_{H^{-1}, H^1} \, \ud s + \sum_{k \geqslant 1} \int_0^T \int_\T \sigma_k(x) H \psi \, \ud x\,  \ud \beta_k(s),
\end{equation}
for any $\psi \in C^2_c((0,T) \times \T)$, or 
\begin{equation}\label{Ito1form2}
\langle (u(t)-u_0), \varphi  \rangle_{H^{-1},H^1} = \int_0^t \langle F, \varphi \rangle_{H^{-1}, H^1} \, \ud s + \sum_{k \geqslant 1} \int_0^t \int_\T \sigma_k(x) H \varphi \, \ud x\,  \ud \beta_k(s),
\end{equation}
for any $\varphi \in C^2(\T)$.
We assume here that $(\beta_1(t),\beta_2(t),\dotsc)$ are independent one-dimensional Wiener processes, and (as in \eqref{defq}) that
\begin{equation*}
 \sum_k \|\sigma_k\|_{L^2(\T)}^2< \infty .
\end{equation*}

%%%%%%%%%%%%%%%%%%%%%%%%%%%%%%%%%%%%%%%%%%%%%%%%%%%%%%%%%%%%%%%%%%%%%%%%
\begin{proposition}\label{Proposition:Ito1}
Assume \eqref{Ito1form1} or \eqref{Ito1form2} with 
\begin{equation*}
u \in L^2(\Omega; C([0,T]; H^{-1}(\T))) \cap L^2(\Omega \times [0,T]; H^1(\T)),
\end{equation*}
and $H$ is predictable, and
\begin{equation*}
F \in L^2(\Omega \times [0,T]; H^{-1}(\T)), \qquad  H \in L^\infty (\Omega \times [0,T] \times \T).
\end{equation*}
Then, if $\psi \in C^2((0,T) \times \T)$ and $\Psi \in C^2(\R)$ satisfies
\begin{equation*}
|\Psi'(u)| + |\Psi''(u)| \leqslant C \quad \forall u \in \R,
\end{equation*}
then we have $\Psi(u)  \in L^2(\Omega; C([0,T]; L^2(\T)))$ and 
\begin{align}\nonumber
-\int_0^T \int_\T \Psi(u)  \psi_t\, \ud x\, \ud s &= \int_0^T \int_\T \Psi'(u) F \psi\, \ud x\, \ud s \\ \nonumber
& \quad + \sum_{k \geqslant 1} \int_0^T \int_\T \sigma_k(x) H \psi \Psi'(u)  \, \ud x\,  \ud \beta_k(s)\\ \nonumber
& \quad + \half  \sum_{k \geqslant 1} \int_0^T \int_\T \psi \Psi''( u)  \sigma_k^2 H^2\, \ud x\, \ud s.
\end{align}
Moreover, for any $\varphi \in C^2(\T)$ and $t \in [0,T]$ we have 
\begin{align}\nonumber
\int_\T \left(\Psi(u(t))-\Psi(u_0) \right) \varphi\, \ud x\, \ud s &= \int_0^t \int_\T \Psi'(u) F \varphi\, \ud x\, \ud s \\ \nonumber
& \quad + \sum_{k \geqslant 1} \int_0^t \int_\T \sigma_k(x) H \varphi \Psi'(u)  \, \ud x\,  \ud \beta_k(s)\\ \nonumber
& \quad + \half  \sum_{k \geqslant 1} \int_0^t \int_\T \varphi \Psi''( u)  \sigma_k^2 H^2\, \ud x\, \ud s.
\end{align}
\end{proposition}
%%%%%%%%%%%%%%%%%%%%%%%%%%%%%%%%%%%%%%%%%%%%%%%%%%%%%%%%%%%%%%%%%%%%%%%%

%%%%%%%%%%%%%%%%%%%%%%%%%%%%%%%%%%%%%%%%%%%%%%%%%%%%%%%%%%%%%%%%%%%%%%%%
\begin{proof}[Proof of Proposition~\ref{Proposition:Ito1}] 

\textbf{Step 0.} 
We start by proving that \eqref{Ito1form1} implies \eqref{Ito1form2}.
Let $t \in (0,T)$, since $u \in C([0,T]; H^{-1}(\T)) $ almost surely, we replace $\psi(s,x)$ by $ \tilde{\psi}^\delta(s) \varphi(x)$ where $\tilde{\psi}^\delta \in C^2_c(0,t)$, $\varphi \in C^2(\T)$  and 
\begin{equation*}
	\tilde{\psi}^\delta= 1\ \text{in} \ (\delta,t-\delta), \qquad \tilde{\psi}^\delta_t \geqslant 0\ \text{in} \ (0, \delta), \qquad \tilde{\psi}^\delta_t \leqslant 0\ \text{in} \ (t-\delta,t).
\end{equation*}
Taking $\delta \to 0$ in \eqref{Ito1form1} we obtain \eqref{Ito1form2}.

\textbf{Step 1.} 
Let $j_\varepsilon$ be a Friedrichs mollifier on $\T$, replacing $\varphi$ by $ j_\varepsilon(y-\cdot)$ in \eqref{Ito1form2} we obtain 
\begin{equation*}
\ud u^\varepsilon = F^\varepsilon \, \ud t + \, \ud M^\varepsilon, 
\end{equation*}
where for any $f$, we have $f^\varepsilon = f \ast j_\varepsilon \eqdef \langle f, j_\varepsilon(y-\cdot) \rangle_{H^{-1},H^1}$ and 
\begin{equation*}
	M^\varepsilon (t) \eqdef  \sum_{k \geqslant 1} \int_0^t \left(  \sigma_k H \right)^\varepsilon  \ud \beta_k(t).
\end{equation*}
For any fixed $x \in \T$, we use the It\^o formula to obtain 
\begin{align}\label{ItoPsi}
\ud \Psi( u^\varepsilon) = \Psi'( u^\varepsilon) F^\varepsilon \, \ud t + \Psi'( u^\varepsilon)\,  \ud M^\varepsilon  + \half  \sum_{k \geqslant 1} \Psi''( u^\varepsilon) \left( \left(  \sigma_k H \right)^\varepsilon \right)^2  \ud t.
\end{align}
Using now the It\^o formula for the product, we obtain
\begin{align}\label{ItoPsi2}
\ud \Psi( u^\varepsilon) \psi
&= \Psi( u^\varepsilon) \psi_t  \ud t+ \psi \Psi'( u^\varepsilon) F^\varepsilon \, \ud t + \psi   \Psi'( u^\varepsilon)\,  \ud M^\varepsilon  + \half  \sum_{k \geqslant 1} \psi\Psi''( u^\varepsilon) \left( \left(  \sigma_k H \right)^\varepsilon \right)^2  \ud t.
\end{align}

\textbf{Step 2.} The aim of this step is to show that $(\|u^\varepsilon\|_{L^2})$ is relatively compact in $C([0,T])$.  Taking $\Psi(u)=u^2$ in \eqref{ItoPsi} we obtain 
\begin{equation}\label{decomposeX}
X^\varepsilon(t) = X^\varepsilon(0)+ A^\varepsilon(t) + \tilde{M}^\varepsilon(t),
\end{equation}
where 
\begin{gather*}
	X^\varepsilon(t) \eqdef \|u^\varepsilon(t)\|_{L^2(\T)}^2, \qquad A^\varepsilon(t) \eqdef \int_0^t \left(\int_\T F^\varepsilon u^\varepsilon \, \ud x + \sum_{k \geqslant 1} \left\| \left(  \sigma_k H \right)^\varepsilon \right\|_{L^2}^2 \right) \ud s, \\
	 \tilde{M}^\varepsilon(t) \eqdef  2   \sum_{k \geqslant 1} \int_0^t \int_\T u^\varepsilon \left(  \sigma_k H \right)^\varepsilon \ud x\,  \ud \beta_k(t).
\end{gather*}
Let $0 \leqslant s \leqslant t \leqslant T$, then 
\begin{equation}\label{Agamma}
	|A^\varepsilon(t) - A^\varepsilon(s)| \leqslant \int_s^t b(\sigma)\, \ud \sigma,
\end{equation}
where 
\begin{equation*}
 b(t) \eqdef \|u(t)\|_{H^1}\|F(t)\|_{H^{-1}}+ \| H(t)\|_{L^\infty}^2 \sum_{k \geqslant 1} \|  \sigma_k  \|_{L^2}^2.
\end{equation*}
Since 
\begin{equation}\label{gammabound}
		\E\left[\|b\|_{L^1(0,T)}\right]\leqslant \E\left[\int_0^T\left(\|u\|_{H^1}^2 +\|F\|_{H^{-1}}^2+ \| H(t)\|_{L^\infty}^2 \sum_{k\geqslant 1 } \|\sigma_k\|_{L^2}^2 \right) \ud t\right]<\infty,
\end{equation}
then $b \in L^1(0,T)$ almost surely, therefore $(A^\varepsilon)$ is equicontinuous. Since $A^\varepsilon (0)=0$, Ascoli's Theorem shows that $(A^\varepsilon)$ is relatively compact in $C([0,T])$. 

Next, consider
	\begin{equation*}
	 \tilde{M}(t) \eqdef  2   \sum_{k \geqslant 1} \int_0^t \int_\T u  \sigma_k H\, \ud x\,  \ud \beta_k(t).
	\end{equation*}
	Then $\tilde{M}^\varepsilon -\tilde{M}$	is a martingale so, by Doob's inequality, we have
	\begin{equation*}
		\Pro\left(\|\tilde{M}^\varepsilon -\tilde{M}\|_{C([0,T])}>\delta \right)\leqslant \tfrac{1}{\delta^2}\E\left[|\tilde{M}^\varepsilon (T)-\tilde{M}(T)|^2\right].
	\end{equation*}
	By It\^o's isometry, we obtain
	\begin{equation*}
	\Pro\left(\|\tilde{M}^\varepsilon -\tilde{M}\|_{C([0,T])}>\delta \right)\leqslant \tfrac{2}{\delta^2}\E \sum_{k \geqslant 1} \int_0^T \left| \int_\T \left( u^\varepsilon \left(  \sigma_k H \right)^\varepsilon- u  \sigma_k H\right) \ud x\right|^2  \ud t,
	\end{equation*}
	so, using dominated convergence theorem, we have
	\begin{equation*}
		\lim_{\varepsilon \to 0}	\Pro\left(\|\tilde{M}^\varepsilon -\tilde{M}\|_{C([0,T])}>\delta \right)=0,
	\end{equation*}
	\textit{i.e.} $\tilde{M}^\varepsilon \to \tilde{M}$ in $C([0,T])$ in probability. In particular, there is a subsequence of $(\tilde{M}^\varepsilon)$ which converges a.s. to $\tilde{M}$ in $C([0,T])$. This gives that $(X^\varepsilon)$ is relatively compact in $C([0,T])$. 
	
	\textbf{Step 3.} The aim of this step is to show that $u$ enjoys the regularity $L^2(\Omega; C([0,T]; L^2(\T)))$. 
We recall that, almost-surely, $u \in C([0,T];H^{-1}(\T))$ and $t\mapsto\|u(t)\|_{L^2}^2$ is continuous. This gives the continuity property $u\in C([0,T];L^2(\T))$, as seen by using the expansion
	\begin{equation*}
		\tfrac12\|u(t)-u(s)\|_{L^2}^2=\left(\tfrac12\|u(s)\|_{L^2}^2-\tfrac12 \|u(t)\|_{L^2}^2 \right)+\dual{u(t)}{u(t)-u(s)}.
	\end{equation*} 
Next, using that $u \in L^2(\Omega \times [0,T] \times \T)$ and the Burkholder--Davis--Gundy inequality we obtain that 
	\begin{equation*}
		\sup_\varepsilon \E\left[\sup_{t\in[0,T]} \left|\tilde{M}^\varepsilon(t)\right|\right]  <\infty.
	\end{equation*} 	
By Fatou's Lemma, \eqref{decomposeX}, \eqref{Agamma} with $s=0$, \eqref{gammabound} and the fact that $X^\varepsilon(0) \leqslant \|u_0\|_{L^2}^2$, we obtain
	\begin{equation*}
		\E\left[\sup_{t\in[0,T]}\|u(t)\|_{L^2}^2\right]\leqslant \liminf_{\varepsilon\to 0} \E\left[\sup_{t\in[0,T]} X^\varepsilon(t)\right]< \infty.
	\end{equation*}
\textbf{Step 4.} Once we have the regularity $u \in L^2(\Omega; C([0,T]; L^2(\T)))$, it remains to take the limit $\varepsilon \to 0$ in \eqref{ItoPsi} and \eqref{ItoPsi2} following \cite[Appendix]{DebusscheHofmanovaVovelle16}, the details are omitted here.
\end{proof}
%%%%%%%%%%%%%%%%%%%%%%%%%%%%%%%%%%%%%%%%%%%%%%%%%%%%%%%%%%%%%%%%%%%%%%%%

We present here another It\^o formula where $u$ is differentiable with respect to $t$, whereas $u_t$ is non-differentiable (singular). The proof is simpler than the one of Proposition \ref{Proposition:Ito1} and is omitted here.  

%%%%%%%%%%%%%%%%%%%%%%%%%%%%%%%%%%%%%%%%%%%%%%%%%%%%%%%%%%%%%%%%%%%%%%%%
\begin{proposition}\label{Proposition:Ito2} 
Assume 
\begin{gather} \nonumber
\int_\T (u_t(t)-u_t(0)) \varphi  \, \ud x + \int_\T (\Gamma(u(t))-\Gamma(u(0))) \varphi  \, \ud x = \int_0^t \langle F, \varphi \rangle_{H^{-1}, H^1} \, \ud s\\ 
\label{AppIto2}
+ \sum_{k \geqslant 1} \int_0^t \int_\T \sigma_k(x) H \varphi \, \ud x\,  \ud \beta_k(s),
\end{gather}
where 
\begin{equation*}
u \in L^2(\Omega; C([0,T]; L^2(\T))) \cap L^2(\Omega \times [0,T]; H^1(\T)), \qquad \Gamma \in Lip(\R;\R)
\end{equation*}
and $H$ is predictable,
\begin{equation*}
F \in L^2(\Omega \times [0,T]; H^{-1}(\T)), \qquad H \in L^\infty (\Omega \times [0,T] \times \T). 
\end{equation*}
Then, if $\varphi \in C^2((0,T) \times \T)$ and $\Psi \in C^2(\R)$ satisfies
\begin{equation*}
|\Psi(u)| + |\Psi'(u)| \leqslant C \quad \forall u \in \R,
\end{equation*}
we have 
\begin{align*}\nonumber
\int_\T \Psi(u(t)) u_t(t)\varphi(t) \, \ud x  
&= \int_\T \Psi(u(0)) u_t(0) \varphi(0,x)  \, \ud x + \int_0^t \int_\T \Psi(u) u_t \varphi_t\, \ud x\, \ud s + \int_0^t \int_\T \Psi'(u) u_t^2 \varphi\, \ud x\, \ud s \\ \nonumber
& \quad - \int_\T \left( \tilde{\Gamma}(u(t))\varphi(t) -\tilde{\Gamma}(u(0)) \varphi(0)  \right) \ud x+ \int_0^t \int_\T  \varphi_t \tilde{\Gamma}(u) \, \ud x\, \ud s\\
& \quad + \int_0^t \int_\T \Psi(u) F \varphi\, \ud x\, \ud s   + \sum_{k \geqslant 1} \int_0^t \int_\T \sigma_k(x) H \varphi \Psi(u)  \, \ud x\,  \ud \beta_k(s),
\end{align*}
where $\tilde{\Gamma}'(u) = \Psi(u) \Gamma'(u)$.
\end{proposition}
%%%%%%%%%%%%%%%%%%%%%%%%%%%%%%%%%%%%%%%%%%%%%%%%%%%%%%%%%%%%%%%%%%%%%%%%

\paragraph{Acknowledgements.} This work was supported by the LABEX MILYON (ANR-10-LABX-0070) of Universit\'e de Lyon, within the program ``Investissements d'Avenir (ANR-11-IDEX-0007) operated by the French National Research Agency (ANR). It was also supported by the project ADA ``Averaging, Diffusion-Approximation in infinite dimension: theory and numerics'' (ANR-19-CE40-0019).

Both authors were supported by the Unit\'e de Math\'ematiques Pure et Appliqu\'ees, UMPA (CNRS and ENS de Lyon). The first author also acknowledges support from NYU Abu Dhabi.

\def\cprime{$'$}

\end{document}